\documentclass[12pt]{article}

\usepackage{mathrsfs}
\usepackage{amsmath}
\usepackage{esint}
\usepackage[all]{xy}
\usepackage{amssymb}
\usepackage{amsfonts}
\usepackage{amsthm}
\usepackage{color}
\usepackage{graphicx}
\usepackage{hyperref}
\usepackage{bm}
\usepackage{indentfirst}
\usepackage{geometry}
\theoremstyle{plain}
\newtheorem{thm}{Theorem}[section]
\newtheorem{lem}[thm]{Lemma}
\newtheorem{prop}[thm]{Proposition}
\newtheorem{ques}[thm]{Question}

\newtheorem{cor}[thm]{Corollary}
\theoremstyle{definition}
\newtheorem{defn}[thm]{Definition}
\theoremstyle{remark}
\newtheorem{remark}[thm]{Remark}
\newtheorem{fact}[thm]{Fact}
\newtheorem{exmp}[thm]{Example}

\numberwithin{equation}{section}

\begin{document}

\title {\bf Isometric immersions of RCD($K,N$) spaces via heat kernels}
\author{\it Zhangkai Huang \thanks{ Tohoku University: huang.zhangkai.t2@dc.tohoku.ac.jp}}
\date{\small\today}

\maketitle

\begin{abstract}
Given an RCD$(K,N)$ space $({X},\mathsf{d},\mathfrak{m})$, one can use its heat kernel $\rho$ to map it into the $L^2$ space by a locally Lipschitz map $\Phi_t(x):=\rho(x,\cdot,t)$. The space $(X,\mathsf{d},\mathfrak{m})$ is said to be an isometrically heat kernel immersing space, if each $\Phi_t$ is an isometric immersion {}{after a normalization}. A main result states that any compact isometrically heat kernel immersing RCD$(K,N)$ space is isometric to an unweighted  closed smooth Riemannian manifold. This is justified by a more general result: if a compact non-collapsed RCD$(K, N)$ space has an isometrically immersing eigenmap, then the space is isometric to an unweighted closed Riemannian manifold, which greatly improves a regularity result in \cite{H21} by Honda. As an application of these results, we give a $C^\infty$-compactness theorem for a certain class of Riemannian manifolds with a curvature-dimension-diameter bound and an isometrically immersing eigenmap.

\end{abstract}
\tableofcontents

\section{Introduction}

\subsection{Isometric immersions on Riemannian manifolds}

Let $( M^n,g)$ be an $n$-dimensional closed, that is, compact without boundary, Riemannian manifold. A map
\[
\begin{aligned}
 F: M^n &\longrightarrow \mathbb{R}^{m}\\
 \ p&\longmapsto (\phi_1(p),\ldots,\phi_m(p))
\end{aligned}
\]
is said to be an \textit{isometrically immersing eigenmap} if each $\phi_i$ is a non-constant eigenfunction of $-\Delta$ and $F$ is an isometric immersion in the following sense:
\begin{align}\label{aaaeqn1.1}
F^\ast g_{\mathbb{R}^m}=\sum\limits_{i=1}^m d\phi_i \otimes d\phi_i=g.
\end{align}

Let us recall a theorem of Takahashi in \cite{Ta66} which states that if $(M^n,g)$ is additionally homogeneous and irreducible, then for any eigenspace $V$ corresponding to some non-zero eigenvalue of $-\Delta$, there exists an $L^2(\mathrm{vol}_g)$-orthogonal basis $\{\phi_i\}_{i=1}^m$ ($m=\mathrm{dim}(V)$) of $V$ realizing (\ref{aaaeqn1.1}).

Besides, $(M^n,g)$ can be also smoothly embedded into an infinite dimensional Hilbert space by using its heat kernel ${}{\rho}: M^n\times  M^n\times (0,\infty)\rightarrow (0,\infty)$. More precisely, B\'{e}rard and B\'{e}rard-Besson-Gallot \cite{B85,BBG94} prove that the following map, which is called \textit{the $t$-time heat kernel mapping} in this paper, 
\[
\begin{aligned}
\Phi_t: M^n&\longrightarrow L^2(\text{vol}_g) \\
x&\longmapsto\left(y\longmapsto\rho(x,y,t)\right),
\end{aligned}
\]
is a smooth embedding. Moreover, one can use $\Phi_t$ to pull-back the flat Riemannian metric $g_{L^2}$ on $L^2(\mathrm{vol}_g)$ to get a metric tensor $g_t:=\Phi_t^\ast\left(g_{L^2}\right)$ with the following asymptotic formula:
\begin{equation}\label{eqn1.1}
4(8\pi)^{\frac{n}{2}}  t^{\frac{n+2}{2}}g_t=g-\frac{2t}{3}\left(\mathrm{Ric}_g-\frac{1}{2}\mathrm{Scal}_g g\right)+O(t^2),\ \ \ \  t\downarrow 0. 
\end{equation} 

Again when $(M^n,g)$ is additionally homogeneous and irreducible, it follows from another theorem by Takahashi \cite[Theorem 3]{Ta66} that there exists a non-negative function $c(t)$ such that for all $t>0$, $\sqrt{c(t)}\Phi_t$ is an isometric immersion.

 The observations above lead us to ask the following two questions.

\begin{ques}\label{q1.2}
How to characterize a manifold admitting an isometrically immersing eigenmap?
\end{ques}

\begin{ques}\label{q1.1}
How to characterize a manifold such that each $t$-time heat kernel mapping is an isometric immersion after a normalization?
\end{ques}

Note that if each $t$-time heat kernel mapping of a closed Riemannian manifold $(M^n,g)$ is an isometric immersion after a normalization, then $(M^n,g)$ admits an isometrically immersing eigenmap. Standard spectral theory of elliptic operators implies that there exists
an orthonormal basis $\{\varphi_i\}_{i=1}^\infty$ in $L^2(\mathrm{vol}_g)$ such that each $\varphi_i$ is an eigenfunction of $-\Delta$ with corresponding eigenvalue $\lambda_i$, and that $\{\lambda_i\}_{i=1}^\infty$ satisfies
\[
0=\lambda_0<\lambda_1\leqslant \lambda_2\leqslant \cdots\leqslant \lambda_i\rightarrow\infty.
\]

Then the classical estimates for eigenvalues $\lambda_i$ show that
\begin{align}\label{aeqn1.3}
g=c(t) g_t=c(t)\sum\limits_{i=1}^\infty e^{-2\lambda_i t}d\varphi_i\otimes d\varphi_i, \ \forall t>0.
\end{align}
These estimates also allow us to let $t\rightarrow \infty$ in (\ref{aeqn1.3}) to get (\ref{aaaeqn1.1}) with $\phi_i=\lim_{t\rightarrow \infty}c(t)e^{-\lambda_1 t}\varphi_i$ ($i=1,\cdots,m$), where $m$ is the dimension of the eigenspace corresponding to $\lambda_1$.

The main purposes of the paper are to give positive answers to the both questions above in a non-smooth setting, so-called RCD$(K, N)$ metric measure spaces, explained in the next subsection.

\subsection{Isometric immersions on RCD$(K,N)$ spaces}

\subsubsection{Metric measure spaces satisfying the RCD$(K,N)$ condition}
A triple $({X},\mathsf{d},\mathfrak{m})$ is said to be a metric measure space if $({X},\mathsf{d})$ is a complete separable metric space and $\mathfrak{m}$ is a nonnegative Borel measure with full support on $X$ and being finite on any bounded subset of ${X}$.

In the first decade of this century, Sturm \cite{St06a, St06b} and Lott-Villani \cite{LV09} independently define a notion of a lower Ricci curvature bound $K\in \mathbb{R}$ and an upper dimension bound $N\in [1,\infty]$ for metric measure spaces in a synthetic sense, which is named as the CD$(K,N)$ condition. A metric measure space is said to be an RCD$(K,N)$ space if it satisfies the CD$(K,N)$ condition, and its associated $H^{1,2}$-Sobolev space is a Hilbert space. The precise definition (and the equivalent ones) can be found in \cite{AGS14b,AMS19,G13,G15,EKS15}.

As an example, any weighted Riemannian manifold $(M^n,\mathsf{d}_g,e^{-f}\mathrm{vol}_g)$ with $f\in C^\infty(M^n)$ and $\mathrm{Ric}_N\geqslant Kg$ is an RCD$(K,N)$ space, where $\mathrm{Ric}_N$ is the Bakry-\'{E}mery $N$-Ricci curvature tensor defined by
\[
\mathrm{Ric}_N:=
\left\{\begin{array}{ll}
\mathrm{Ric}_g+\mathrm{Hess}_g(f)-\frac{df\otimes df}{N-n}&\text{if}\ N>n,\\
\mathrm{Ric}_g& \text{if $N=n$ and $f$ is a constant},\\
-\infty&\text{otherwise}.
\end{array}\right.
\]

In the sequel, we always assume that $N$ is finite.

Given an RCD$(K,N)$ space $({X},\mathsf{d},\mathfrak{m})$, with the aid of a work by Bru\`e-Semola \cite{BS20}, there exists a unique $n\in [1,N]\cap \mathbb{N}$, which is called the essential dimension of $({X},\mathsf{d},\mathfrak{m})$ and is denoted by $n:=\mathrm{dim}_{\mathsf{d},\mathfrak{m}}({X})$, such that the $n$-dimensional regular set $\mathcal{R}_n$ (see Definition \ref{111def2.18}) satisfies that $\mathfrak{m}=\theta \mathcal{H}^n\llcorner \mathcal{R}_n$ for some Borel function $\theta$ (see \cite{AHT18}), where $\mathcal{H}^n$ is the $n$-dimensional Hausdorff measure. It is remarkable that the canonical Riemannian metric $g$ on $({X},\mathsf{d},\mathfrak{m})$ is also well-defined due to a work by Gigli-Pasqualetto \cite{GP16} (see also \cite[Proposition 3.2]{AHPT21} and Definition \ref{111thm2.21}). Then its $\mathfrak{m}$-a.e. pointwise Hilbert-Schmidt norm $|g|_{\mathsf{HS}}$ is equal to $\sqrt{n}$.

Let us introduce a special restricted class of RCD$(K, N)$ spaces introduced in \cite{DG18} by De Philippis-Gigli as a synthetic counterpart of volume non-collapsed Gromov-Hausdorff limit spaces of Riemannian manifolds with a constant dimension and a lower Ricci curvature bound. The definition is simple: an RCD$(K, N)$ space is said to be non-collapsed if the reference measure is $\mathcal{H}^N$. {}{It can be easily shown that in this case $N$ must be an integer}. Non-collapsed RCD$(K, N)$ spaces have nicer properties than general RCD$(K,N)$ spaces. See also for instance \cite{ABS19, KM21}.

\subsubsection{Isometrically heat kernel immersing RCD$(K,N)$ spaces}

Thanks to works by Sturm \cite{St95, St96} and by Jiang-Li-Zhang \cite{JLZ16}, the heat kernel on an RCD$(K,N)$ space $({X},\mathsf{d},\mathfrak{m})$ has a locally Lipschitz representative $\rho$ with Gaussian estimates. This allows us to {}{construct $\Phi_t$ analogously as
\[
\begin{aligned}
\Phi_t:X&\longrightarrow L^2(\mathfrak{m})\\
 x&\longmapsto (y\longmapsto \rho(x,y,t)),
\end{aligned}
\]
which also naturally induces the pull back metric $g_t:=\Phi_t^\ast(g_{L^2(\mathfrak{m})})$.}

One can also generalize formula (\ref{eqn1.1}) to this setting with the $L^p_{\mathrm{loc}}$ convergence as follows, see \cite[Theorem 5.10]{AHPT21} and \cite[Theorem 3.11]{BGHZ21} for the proof.
\begin{thm}\label{20211222a}
Let $({X},\mathsf{d},\mathfrak{m})$ be an $\mathrm{RCD}(K,N)$ space with $\mathrm{dim}_{\mathsf{d},\mathfrak{m}}({X})=n$, then for any $p\in [1,\infty)$ and any bounded Borel set $A\subset X$, we have the following convergence in $L^p(A,\mathfrak{m})$:
\[
\left| t\mathfrak{m}(B_{\sqrt{t}}(\cdot))g_t-c(n) g\right|_{\mathsf{HS}}\rightarrow 0,  \ \ \text{as }t\downarrow 0,
\]
where $c(n)$ is a constant depending only on $n$.
\end{thm}

 
In connection with Question \ref{q1.1} in this setting, let us provide the following definition.
  \begin{defn}[Isometrically heat kernel immersing RCD$(K,N)$ spaces]
An RCD$(K,N)$ space $({X},\mathsf{d},\mathfrak{m})$ is said to be an \textit{isometrically heat kernel immersing} space, or briefly an IHKI space if there exists a non-negative function $c(t)$, such that $\sqrt{c(t)}\Phi_t$ is an isometric immersion for all $t>0$, namely
\[
c(t)g_t=\left(\sqrt{c(t)}\mathop{\Phi_t}\right)^\ast\left(g_{L^2(\mathfrak{m})}\right)=g,\ \forall t>0.
\]. 
  \end{defn}
 
We are now in a position to introduce the first main result of this paper.

\begin{thm}\label{thm1.2}
Let $({X},\mathsf{d},\mathfrak{m})$ be an $\mathrm{RCD}(K,N)$ space. Then the following two conditions are equivalent.
\begin{enumerate} 
\item[$(1)$]\label{thm1.1con1} There exist sequences $\{t_i\}\subset \mathbb{R}$ and $\{s_i\}\subset \mathbb{R}$ such that $t_i\rightarrow t_0$ for some $t_0>0$ and that $s_i\Phi_{t_i}$ is an isometric immersion for any $i$.
\item[$(2)$] $({X},\mathsf{d},\mathfrak{m})$ is an $\mathrm{IHKI}$ $\mathrm{RCD}(K,N)$ space. 
\end{enumerate}
\end{thm}
\begin{remark}
Theorem \ref{thm1.2} is sharp in the following sense: there exists a closed Riemannain manifold $(M^n, g)$ such that it is not IHKI and that $c\Phi_{t_0}$ is an isometric immersion for some $c>0$ and some $t_0>0$. See Example \ref{exmp4.5}.
\end{remark}

Recalling that $g_t$ plays a role of a ``regularization'' of an RCD$(K, N)$ space as discussed in \cite{BGHZ21}, it is expected that IHKI RCD$(K, N)$ spaces have nice regularity properties. Along this, we end this subsection by collecting such regularity results as follows.   


\begin{thm}\label{mainthm1.3}
Let $({X},\mathsf{d},\mathfrak{m})$ be an $\mathrm{IHKI}$ $\mathrm{RCD}(K,N)$ space with $\mathrm{dim}_{\mathsf{d},\mathfrak{m}}({X})=n\geqslant 1$, then there exists $c>0$ such that $\mathfrak{m}=c\mathcal{H}^n$ and that $({X},\mathsf{d},\mathfrak{m})$ is an $\mathrm{RCD}(K,n)$ space. In particular, $({X},\mathsf{d},\mathcal{H}^n)$ is a non-collapsed $\mathrm{RCD}(K,n)$ space.

\end{thm}

\begin{thm}\label{mainthm1.5}
Assume that $({X},\mathsf{d},\mathfrak{m})$ is a non-compact $\mathrm{IHKI}$ $\mathrm{RCD}(0,N)$ space with $\mathrm{dim}_{\mathsf{d},\mathfrak{m}}({X})=n\geqslant 2$, then $({X},\mathsf{d},\mathfrak{m})$ is isometric to $\left(\mathbb{R}^n,\mathsf{d}_{\mathbb{R}^n},c\mathcal{H}^n\right)$ for some $c>0$.
\end{thm}

Let us emphasize that in the compact setting we will be able to provide the best regularity result, namely the smoothness result (see Theorem \ref{thm1.5} and Corollary \ref{cor1.11}).

\subsubsection{Isometrically immersing eigenmaps on RCD$(K,N)$ spaces}

In order to discuss a finite dimensional analogue of the IHKI condition, let us recall the following definition. 

\begin{defn}[Isometric immersion {\cite[Definition 3.1]{H21}}]
Let $m\in \mathbb{N}_+$ and let $(X,\mathsf{d},\mathfrak{m})$ be an RCD$(K,N)$ space. A map 
\[
\begin{aligned}
\Phi:X&\longrightarrow \mathbb{R}^m\\
      x&\longmapsto (\phi_1(x),\ldots,\phi_m(x))
\end{aligned}
\]
is said to be an \textit{isometric immersion} if it is locally Lipschitz and 
\begin{align}\label{20221207a}
\Phi^\ast g_{\mathbb{R}^m}:=\sum\limits_{i=1}^m d\phi_i\otimes d\phi_i =g
\end{align}
\end{defn}

We are now ready to give an answer to Question \ref{q1.2} in the nonsmooth setting.

\begin{thm}\label{thm1.5}
Let $({X},\mathsf{d},\mathcal{H}^n)$ be a compact non-collapsed $\mathrm{RCD}(K,n)$ space. If there exists an isometric immersion
\[
\begin{aligned}
\Phi:X&\longrightarrow \mathbb{R}^m\\
      x&\longmapsto (\phi_1(x),\ldots,\phi_m(x))
\end{aligned}
\]
such that each $\phi_i$ is an eigenfunction of $-\Delta$ $(i=1,\ldots,m)$, then $({X},\mathsf{d})$ is isometric to an $n$-dimensional smooth closed Riemannian manifold $(M^n,g)$.
\end{thm}

It is emphasized again that the theorem above greatly improves a bi-Lipschitz regularity result proved in \cite{H21} and seems to provide the smoothness for a much wider class of RCD spaces than existing results as far as the author knows (see for instance \cite{K15b,GR18,MW19} for the special cases).

\begin{remark}
An isometrically immersing eigenmap may not be an embedding in general. See for instance \cite[Theorem 5]{L81}.
\end{remark}

As a corollary of Theorem \ref{thm1.5}, we obtain the following result, meaning that any compact IHKI RCD$(K,N)$ space must be smooth.

\begin{cor}\label{cor1.11}
Let $({X},\mathsf{d},\mathcal{H}^n)$ be a compact non-collapsed $\mathrm{IHKI}$ $\mathrm{RCD}(K,n)$ space. Let $E$ be the eigenspace with some non-zero corresponding eigenvalue $\lambda$ of $-\Delta$. Then by taking $\{\phi_i\}_{i=1}^m$ $(m=\mathrm{dim}(E))$ as an $L^2(\mathfrak{m})$-orthonormal basis of $E$, the map 
\[
\begin{aligned}
\Phi:{X}&\longrightarrow \mathbb{R}^m\\
x&\longmapsto\sqrt{\dfrac{\mathcal{H}^n({X})}{m}}(\phi_1,\cdots,\phi_m),
\end{aligned}
\]
satisfies that 
\[
\Phi({X})\subset \mathbb{S}^{m-1}\ \ \text{and}\ \ n\Phi^\ast g_{\mathbb{R}^m}=\lambda g.
\]
In particular, $(X,\mathsf{d})$ is isometric to an $n$-dimensional smooth closed Riemannian manifold $(M^n,g)$.
\end{cor}

\subsection{Diffeomorphic finiteness theorems}

As an application of Theorem \ref{thm1.5}, in Section 
\ref{sec5} we first study some special isometry classes of closed Riemannian manifolds admitting isometrically immersing $\tau$-eigenmaps.

\begin{defn}[Isometrically immersing $\tau$-eigenmap on Riemannian manifolds]
Let $(M^n,g)$ be an $n$-dimensional closed Riemannian manifold and let $\tau>0$. A map
\[
\begin{aligned}
F: M^n&\longrightarrow \mathbb{R}^m\\
p&\longmapsto \left(\phi_1(p),\ldots,\phi_m(p)\right),
\end{aligned}
\]
is said to be a \textit{$\tau$-eigenmap into $\mathbb{R}^m$} if each $\phi_i$ is a non-constant eigenfunction of $-\Delta$ and \[
\min\limits_{1\leqslant i\leqslant m} \|\phi_i\|_{L^2( \mathrm{vol}_g)}\geqslant \tau.
 \]
If in addition $F$ is an isometric immersion, then it is said to be an \textit{isometrically immersing $\tau$-eigenmap into $\mathbb{R}^m$}.

\end{defn}

\begin{defn}[Isometric immersion via $\tau$-eigenmaps]\label{defn1.7}
For all $K\in \mathbb{R}$, $D,\tau>0$, denote by $\mathcal{M}(K,n,D,\tau)$ the set of isometry classes of $n$-dimensional closed Riemannian manifolds $( M^n,g)$  such that the Ricci curvature is bounded below by $K$, that the diameter is bounded above by $D$ and that there exists an isometrically immersing $\tau$-eigenmap into $\mathbb{R}^m$ for some $m \in \mathbb{N}$.
\end{defn}

Our main result about $\mathcal{M}(K,n,D,\tau)$ is stated as follows.

\begin{thm}\label{thm1.8}
$\mathcal{M}(K,n,D,\tau)$ is compact in $C^\infty$-topology. That is, for any sequence of Riemannian manifolds $\{( M_i^n,g_i)\}\subset\mathcal{M}(K,n,D,\tau)$, after passing to a subsequence, there exists a Riemannian manifold $(M^n,g)\in \mathcal{M}(K,n,D,\tau)$ and diffeomorphisms $\psi_i: M^n\rightarrow  M^n_i$, such that $\{\psi_i^\ast g_i\}$ $C^{k}$-converges to $g$ on $(M^n,g)$ for any $k\in \mathbb{N}$.
\end{thm}

Finally in order to introduce an improved finiteness result from \cite{H21}, let us introduce the following definition.

\begin{defn}[Almost isometric immersion via $\tau$-eigenmap]
For all $K\in \mathbb{R}$, $D,\tau>0$, $\epsilon\geqslant 0$, denote by $\mathcal{N}(K,n,D,\tau,\epsilon)$ the set of isometry classes of $n$-dimensional closed Riemannian manifolds $(M^n,g)$ such that the Ricci curvature is bounded below by $K$, that the diameter is bounded above by $D$ and that there exists a $\tau$-eigenmap $F_{M^n}$ into $\mathbb{R}^m$ for some $m \in \mathbb{N}$ with
\[
\frac{1}{\mathrm{vol}_{g}(M^n)}\int_{M^n}\left| F_{M^n}^\ast g_{\mathbb{R}^m}-g\right|\mathrm{dvol}_g\leqslant \epsilon.
\]
\end{defn}

Note that $\mathcal{N}(K,n,D,\tau,0)=\mathcal{M}(K,n,D,\tau)$. Combining the intrinsic Reifenberg method established in \cite[Appendix A]{ChCo1} by Cheeger-Colding, with Theorem \ref{thm1.5} gives us the following diffeomorphic finiteness theorem.

\begin{thm}\label{thm1.12}
There exists $\epsilon=\epsilon(K,n,D,\tau)>0$ such that $\mathcal{N}(K,n,D,\tau,\epsilon)$ has finitely many members up to diffeomorphism.
\end{thm}
\subsection{Outline of the proofs}

The proofs of Theorems \ref{mainthm1.3} and \ref{mainthm1.5} are based on blow up and blow down arguments. See also the proofs of \cite[Theorem 2.19]{AHPT21} and \cite[Theorem 3.11]{BGHZ21} for related arguments.

The most delicate part of this paper is in the proof of Theorem \ref{thm1.5}, which makes full use of the equations for eigenfunctions, i.e. $\Delta \phi_i=-\mu_i\phi_i$ ($i=1,\ldots,m$). Note that one can easily obtain $L^\infty$-bounds of the Laplacian and the gradient of each $\phi_i$ from the estimates in \cite{J14,JLZ16,ZZ19,AHPT21} (see also Proposition \ref{heatkernel2}). 

In order to explain it more precisely, let us start with the following key equation:
\begin{equation}\label{aaaaa1111122}
\sum\limits_{i=1}^m |\nabla \phi_i|^2=n.
\end{equation}

Since the lower bound of each $\Delta |\nabla \phi_i|^2$ comes directly from Bochner inequality (see (\ref{bochnerineq})), (\ref{aaaaa1111122}) then guarantees the upper bound of each $\Delta |\nabla \phi_i|^2$ due to the following equality:

\[
\Delta|\nabla \phi_i|^2=\sum\limits_{j\neq i}^m -\Delta|\nabla \phi_j|^2.
\]
Therefore we have a uniform $L^\infty$-bound of all $|\nabla\langle \nabla \phi_i,\nabla\phi_j\rangle|$, which implies the $C^{1,1}$ differentiable structure of the space. Indeed, locally one can pick $\{u_i\}_{i=1}^m$ consisting of linear combinations of eigenfunctions $\phi_i$ and construct a bi-Lipschitz map $x\mapsto (u_1(x),\ldots,u_n(x))$ which satisfies the following PDE:
\[
\sum\limits_{j,k=1}^m \langle \nabla u_j,\nabla u_k\rangle\frac{\partial^2 \phi_i}{ \partial u_j \partial u_k}+\sum\limits_{j=1}^n\Delta u_j \frac{\partial \phi_i}{ \partial u_j }+\mu_i \phi_i=0.
\]
Then the smoothness of the space is justified by applying the elliptic regularity theory.  

Finally, a similar technique as in the proof of Theorem \ref{thm1.5} allows us to control each higher order covariant derivative of the Riemannian metric $g$ of $(M^n, g) \in \mathcal{M}$ quantitatively. Thus we can then apply a theorem of Hebey-Herzlish proved in \cite{HH97} to get the desired smooth compactness result, Theorem \ref{thm1.8}.

\textbf{Acknowledgement.} The author acknowledges the support of JST SPRING, Grant Number JPMJSP2114. He is grateful to the referee for carefully reading the paper and for giving many valuable suggestions. He thanks his supervisor Professor Shouhei Honda for his advice and encouragement. He also thanks Yuanlin Peng and Zuyi Zhang for their comments on this paper.

\section{{}{Notation} and preliminary results}\label{sec2}

Throughout this paper we will use standard {}{notation} in this topic. For example

\begin{itemize}

\item Denote by $C(K_1,\ldots,K_n)$ a positive constant depending on $K_1,\ldots,K_n$, and $\Psi=\Psi(\epsilon_1,\ldots,\epsilon_k|c_1,\ldots c_j)$ some nonnegative function determined by $\epsilon_1,\ldots,\epsilon_k$, $c_1,\ldots, c_j$ such that 
\[
\lim\limits_{\epsilon_1,\ldots,\epsilon_k\rightarrow 0}\Psi=0,\  \text{for any fixed}\ c_1,\ldots c_j.
\]

\item Denote by $\omega_n$ the $n$-dimensional Hausdorff measure of the unit ball in $\mathbb{R}^n$ which coincides with the usual volume of the unit ball in $\mathbb{R}^n$, and by $\mathcal{L}^n$ the standard Lebesgue measure on $\mathbb{R}^n$. 

\end{itemize}

We may use superscripts or subscripts when it is necessary to distinguish objects (for example, the Riemannian metrics, the gradients, etc.) on different spaces in this paper.

\subsection{Metric spaces}
{}{We fix some basic definitions} and {}{notation} about metric spaces in this subsection. {}{Let $({X},\mathsf{d})$ be a complete separable metric space.}
 
 Denote by $ \text{Lip}({X},\mathsf{d})$ (resp. $\text{Lip}_b({X},\mathsf{d})$, $\text{Lip}_c({X},\mathsf{d})$,  $C_\text{c}({X})$) the set of all Lipschitz functions (resp. bounded Lipschitz functions, compactly supported Lipschitz functions, compactly supported continuous functions) on  ${}{({X},\mathsf{d})}$.
 
 For any $f\in \text{Lip}({X},\mathsf{d})$, the local Lipschitz constant of $f$ at {}{a point} $x\in {X}$ is defined  by
\[
\text{lip}\
 f(x)=\left\{\begin{aligned}
\limsup\limits_{y\rightarrow x} \frac{|f(y)-f(x)|}{\mathsf{d}(y,x)}&\ \ \ \text{if $x\in {X}$ is not isolated},\\
0\ \ \ \ \ \ \ \ \ \ \ \ \ \ \ \ \ \ \ \ \ \ \ \ \ &\ \ \ \text{otherwise}.
\end{aligned}\right.
\]

If $({X},\mathsf{d})$ is compact, then the diameter of ${X}$ is defined by 
\[
\mathrm{diam}({X},\mathsf{d}):=\sup_{x,y\in{X}}\mathsf{d}(x,y).
\]

For a map $f:{X}\rightarrow {Y}$ from $({X},\mathsf{d})$ to another complete metric space $({Y},\mathsf{d}_Y)$, $f$ is said to be $C$-bi-Lipschitz from ${X}$ to $f({X})$ for some $C\geqslant 1$ if
\[
C^{-1}\mathsf{d}(x_1,x_2)\leqslant \mathsf{d}_Y(f(x_1),f(x_2))\leqslant C\mathsf{d}(x_1,x_2),\ \forall x_1,x_2\in{X}.
\]

We also denote by $B_R(x)$ the set $\{y\in{X}: \mathsf{d}(x,y)<R\}$, and by $B_\epsilon(A)$ the set $\{x\in {X}:\mathsf{d}(x,A)<\epsilon\}$ for any $A\subset {X}$, $\epsilon>0$. In particular,  denote by $B_r(0_n):=\{x\in \mathbb{R}^n:|x|< r\}$ for any $r>0$.

\subsection{RCD$(K,N)$ spaces: definition and basic properties}\label{sec2.2}


{}{Let $({X},\mathsf{d},\mathfrak{m})$  be a metric measure space}.

\begin{defn}[Cheeger energy]
 The Cheeger energy Ch: $L^2(\mathfrak{m})\rightarrow [0,\infty]$ is defined by
\[
\text{Ch}(f):=\inf\limits_{\{f_i\}}\left\{ \liminf\limits_{i\rightarrow \infty} \int_{{X}} |\text{lip}\mathop{f_i}|^2 \mathrm{d}\mathfrak{m} \right\},
\]
where the infimum is taken among all sequences $\{f_i\}$ satisfying $f_i\in \text{Lip}_b({X},\mathsf{d})\cap L^2(\mathfrak{m})$ and $\left\|f_i-f\right\|_{L^2(\mathfrak{m})}\rightarrow 0$. 
\end{defn}

The domain of the Cheeger energy, denoted by $D\text{(Ch)}$, is the set of all $f\in L^2(\mathfrak{m})$ with $\text{Ch}(f)<\infty$. It is dense in $L^2(\mathfrak{m})$, and is a Banach space when equipped with the norm $\sqrt{\text{Ch}(\cdot)+\left\|\cdot\right\|_{L^2(\mathfrak{m})}^2}$. This Banach space is the Sobolev space $H^{1,2}({X},\mathsf{d},\mathfrak{m})$. In addition, for any $f\in H^{1,2}({X},\mathsf{d},\mathfrak{m})$, it is known that there exists a {}{unique} $|\text{D}f|\in L^2(\mathfrak{m})$ such that
\[
\text{Ch}(f)=\int_{{X}} |\text{D}f|^2 \mathrm{d}\mathfrak{m}.
\]
This $|\text{D}f|$ is called the minimal relaxed slope of $f$ and satisfies the {}{locality property, that is}, for any other $h \in H^{1,2}({X},\mathsf{d},\mathfrak{m})$, $|\mathrm{D}f|=|\mathrm{D}h|$ $\mathfrak{m}$-a.e. on $\{x\in{X}:f=h\}$.


{}{In particular}, $({X},\mathsf{d},\mathfrak{m})$ is said to be infinitesimally Hilbertian if $H^{1,2}({X},\mathsf{d},\mathfrak{m})$ is a Hilbert space. In this case, for any $f,h\in H^{1,2}({X},\mathsf{d},\mathfrak{m})$, the following $ L^1(\mathfrak{m}) $ integrable function is well-defined \cite{AGS14b}: 

  \[
  \langle \nabla f, \nabla h\rangle :=  \lim_{\epsilon \rightarrow 0}\frac{|\text{D}(f+\epsilon h)|^2-|\text{D} f|^2}{2\epsilon}.
  \]
 
\begin{remark}
For any $f\in H^{1,2}({X},\mathsf{d},\mathfrak{m})$, it is clear that
\[
|\nabla f|^2:=\langle \nabla f,\nabla f\rangle=|\mathrm{D}f|^2,\ \mathfrak{m}\text{-a.e.}
\]
\end{remark}

\begin{defn}[The Laplacian \cite{G15}]

  Assume that $({X},\mathsf{d},\mathfrak{m})$ is infinitesimally Hilbertian. The domain of Laplacian, namely $D(\Delta)$, is defined as the set of all $f\in H^{1,2}({X},\mathsf{d},\mathfrak{m})$ such that 
\[
 \int_{{X}} \langle \nabla f, \nabla \varphi\rangle \mathrm{d}\mathfrak{m}= - \int_{{X}} h\varphi \mathrm{d}\mathfrak{m},\ \  \forall \varphi \in H^{1,2}({X},\mathsf{d},\mathfrak{m}),
\]  
for some $h\in L^2(\mathfrak{m})$.
In particular, denote by $\Delta f:= h$ for any $f\in D(\Delta)$ because $h$ is unique if it exists.
\end{defn}

We are now ready to introduce the definition of RCD$(K,N)$ spaces. {}{The following is an equivalent definition with the one proposed in \cite{G15}, and the equivalence is proved in \cite{AGS15,EKS15}. See also \cite{AMS19}.}

\begin{defn}
Let $K\in \mathbb{R}$ and $N\in [1,\infty)$. $({X},\mathsf{d},\mathfrak{m})$ is said to be an RCD$(K,N)$ space {}{if and only if} it satisfies the following conditions.
\begin{enumerate}
  \item  $({X},\mathsf{d},\mathfrak{m})$ is infinitesimally Hilbertian.

  \item There exists $ x \in {X}$ and $C >0$, such that {}{for any $r>0$}, $\mathfrak{m} (B_r(x)) \leqslant C e^{Cr^2}$.

  \item (Sobolev to Lipschitz property) If $f \in H^{1,2}({X},\mathsf{d},\mathfrak{m})$ with $|\text{D} f|\leqslant 1$ $\mathfrak{m}$-a.e., then $f$ has a 1-Lipschitz {}{representative, that is,} there exists {}{a 1-Lipschitz function $h$ such that $h=f$ $\mathfrak{m}$-a.e.}

  \item  ({}{Bochner} inequality) For any {}{$f\in D(\Delta)$ with $\Delta f \in H^{1,2}({X},\mathsf{d},\mathfrak{m})$}, the following holds for any $\varphi \in \mathrm{Test}F\left({X},\mathsf{d},\mathfrak{m}\right)$ with
  $ \varphi \geqslant 0$,
\begin{equation}\label{bochnerineq}
  \frac{1}{2}\int_{X} |\nabla f|^2 \Delta \varphi \mathrm{d}\mathfrak{m} 
  \geqslant \int_{X} \varphi \left(\langle \nabla f , \nabla \Delta f \rangle +K |\nabla f|^2   + \frac{(\Delta f)^2}{N}  \right) \mathrm{d}\mathfrak{m},
  \end{equation}
where $\mathrm{Test}F({X},\mathsf{d},\mathfrak{m})$ is the class of test functions defined by
\end{enumerate}
 \[
\mathrm{Test}F({X},\mathsf{d},\mathfrak{m}):=\{f\in \text{Lip}({X},\mathsf{d})\cap D(\Delta)\cap L^\infty(\mathfrak{m}):\Delta f\in H^{1,2}({X},\mathsf{d},\mathfrak{m})\cap L^\infty(\mathfrak{m})\}.
\]

If in addition $\mathfrak{m}=\mathcal{H}^N$, then $({X},\mathsf{d},\mathfrak{m})$ is said to be a non-collapsed RCD$(K,N)$ space.
\end{defn}

For the class of test functions on an RCD$(K,N)$ space $({X},\mathsf{d},\mathfrak{m})$, by \cite{S14},
\begin{enumerate}
\item  $|\nabla f|^2 \in H^{1,2}({X},\mathsf{d},\mathfrak{m})$ for any $f\in \mathrm{Test}F\left({X},\mathsf{d},\mathfrak{m}\right)${}{.} 
\item Define $\mathrm{Test}F_+({X},\mathsf{d},\mathfrak{m}):=\left\{f\in \mathrm{Test}F\left({X},\mathsf{d},\mathfrak{m}\right): f\geqslant 0\right\}$ and $H^{1,2}_+({X},\mathsf{d},\mathfrak{m}):=\left\{f\in H^{1,2}({X},\mathsf{d},\mathfrak{m}):  f\geqslant 0\ \ \mathfrak{m}\text{-a.e.}\right\}$.  Then $\mathrm{Test}F_+({X},\mathsf{d},\mathfrak{m})$ (resp. $\mathrm{Test}F({X},\mathsf{d},\mathfrak{m})$) is dense in $H^{1,2}_+({X},\mathsf{d},\mathfrak{m})$ (resp. $H^{1,2}({X},\mathsf{d},\mathfrak{m})$).
\end{enumerate}

The following inequality is a generalization of the Bishop-Gromov inequality {}{in Riemannian geometry.}
\begin{thm}[Bishop-Gromov inequality \cite{LV09,St06b}]\label{BGineq}
Assume that $({X},\mathsf{d},\mathfrak{m})$ is an $\mathrm{RCD}(K,N)$ space. Then the following holds for any $x\in {X}$.
\begin{enumerate}
\item If $N>1$, $K\neq 0$, $r<R\leqslant \pi\sqrt{\dfrac{N-1}{K\lor 0}}$, then $\dfrac{\mathfrak{m}\left(B_R(x)\right)}{\mathfrak{m}\left(B_r(x)\right)}\leqslant  \dfrac{\int_0^R V_{K,N}\mathrm{d}t}{\int_0^r V_{K,N}\mathrm{d}t}$, where 
\[ 
V_{K,N}(t):=\left\{
\begin{array}{ll}
\sin\left(t\sqrt{K/(N-1)}\right)^{N-1}, &\text{if}\  K>0,\\
 \sinh\left(t\sqrt{{}{-K}/(N-1)}\right)^{N-1}, &\text{if}\  K<0.
 \end{array}
 \right.
 \]  
\item If $N=1$ and $K\leqslant 0$, or $N\in (1,\infty)$ and $K= 0$, then $\dfrac{\mathfrak{m}\left(B_R(x)\right)}{\mathfrak{m}\left(B_r(x)\right)}\leqslant  \left(\dfrac{R}{r}\right)^{N}$. 
\end{enumerate}
\end{thm}

\begin{remark}
(\ref{BGinequality}) and (\ref{BGinequality111}) are direct consequences of Theorem \ref{BGineq}, where  (\ref{BGinequality111}) is a combination of (\ref{BGinequality}) and the fact that $B_r(x)\subset B_{r+\mathsf{d}(x,y)}(y)$. 
\begin{equation}\label{BGinequality}
\frac{\mathfrak{m}(B_R(x))}{\mathfrak{m}(B_r(x))}\leqslant C(K,N)\exp\left(C(K,N)\frac{R}{r}\right), \ \ \ \forall x\in {X}, \ \forall r<R.
\end{equation}
{}{\begin{equation}\label{BGinequality111}
\frac{\mathfrak{m}(B_r(x))}{\mathfrak{m}(B_r(y))}\leqslant C(K,N)\exp\left(C(K,N)\mathop{\frac{r+\mathsf{d}(x,y)}{r}}\right), \ \ \ \forall x,y\in {X}, \ \forall r>0.
\end{equation}}

\end{remark}

For an RCD$(K,N)$ space $({X},\mathsf{d},\mathfrak{m})$, the heat flow associated with its {}{Cheeger energy} is defined as ${}{\{\mathrm{h}_t:L^2(\mathfrak{m})\rightarrow L^2(\mathfrak{m})\}_{t>0}}$ such that for any $f \in L^2(\mathfrak{m})$,  $\{{}{\mathrm{h}_t} f\}_{t>0}$ satisfies the following properties.

\begin{enumerate}

\item (Solution to the heat equation) {}{For any $t>0$}, $\text{h}_t f\in D(\Delta)$ and $\dfrac{\partial}{\partial t}\text{h}_t(f)=\Delta {}{
\mathrm{h}_t f} \ \ \text{in}\ L^2(\mathfrak{m})$.
\item (Semigroup property) For any $s,t>0$, ${}{\text{h}_{t+s}}f=\text{h}_t ({}{\text{h}_s} f)$.
{}{\item (Contraction on $L^2(\mathfrak{m})$) $\left\|\text{h}_t  f\right\|_{L^2(\mathfrak{m})}\leqslant \left\|f\right\|_{L^2(\mathfrak{m})},\ \ \forall t>0$.
\item (Commutative with $\Delta$) If $f\in D(\Delta)$, then for any $t>0$, $\text{h}_t (\Delta f)=\Delta (\text{h}_t f)$.}
\end{enumerate}

For any $p\in [1,\infty]$, $\{\text{h}_t\}_{t>0}$ also acts on $L^p(\mathfrak{m})$ as a linear family of contractions, namely 
\begin{equation}\label{111eqn2.4}
\left\|\text{h}_t  \varphi\right\|_{L^p(\mathfrak{m})}\leqslant \left\|\varphi\right\|_{L^p(\mathfrak{m})},\ \ \forall t>0,\ \  \forall \varphi\in L^p(\mathfrak{m}).
\end{equation}

Set $\hat{1}\in L^\infty(\mathfrak{m})$ as (the equivalence class in $\mathfrak{m}$-a.e. sense of) the function on ${X}$ identically equal to 1. It is now worth pointing out the stochastic completeness of RCD$(K,N)$ spaces as follows:
 \[
{}{\mathrm{h}_t}(\hat{1})\equiv \hat{1},\ \ \forall t>0.
\]

Sturm's works \cite{St95, St96} guarantee the existence of a locally H\"older continuous {}{representative} $\rho$ on ${X}\times{X}\times (0,\infty)$ of the heat kernel for $({X},\mathsf{d},\mathfrak{m})$. More precisely, the solution to the heat equation can be expressed by using $\rho$ as follows:
\[
\text{h}_t(f)=\int_{{X}} \rho(x,y,t)f(y)\mathrm{d}\mathfrak{m}(y),\  \forall f\in L^2(\mathfrak{m}).
\] 
\begin{remark}[Rescaled RCD space]
For any RCD$(K,N)$ space $({X},\mathsf{d},\mathfrak{m})$ and any $a,b\in (0,\infty)$, the rescaled space $({X},a\mathsf{d},b\mathfrak{m})$ is an RCD$(a^{-1}K,N)$ space whose heat kernel $\tilde{\rho}$ can be written as $\tilde{\rho}(x,y,t)=b^{-1}\rho(x,y,a^{-2}t)$.
\end{remark}

The locally H\"older {}{continuity} of the heat kernel on RCD$(K,N)$ spaces is improved to be locally Lipschitz due to the following Jiang-Li-Zhang's \cite{JLZ16} estimates. 

\begin{thm}\label{thm2.12}
Let $({X},\mathsf{d},\mathfrak{m})$ be an $\mathrm{RCD}(K,N)$ space. Given any $\epsilon>0$, there exist positive constants $C_i=C_i(K,N,\epsilon),i=1,2,3,4$ such that the heat kernel $\rho$ {}{satisfies} the following estimates.
\[
\frac{1}{C_1}\exp\left({-\frac{\mathsf{d}^2(x,y)}{(4-\epsilon)t}}-C_2t\right)\leqslant \mathfrak{m}\left(B_{\sqrt{t}}(y)\right)\rho(x,y,t) \leqslant C_1\exp\left({-\frac{\mathsf{d}^2(x,y)}{(4+\epsilon)t}}+C_2t\right)
\]
holds for all $t>0$, and all $x,y\in {X}$ and

\[
|\nabla_x \rho(x,y,t)| \leqslant \frac{C_3}{\sqrt{t}\ \mathfrak{m}\left(B_{\sqrt{t}}(x)\right)}\exp\left({-\frac{\mathsf{d}^2(x,y)}{(4+\epsilon)t}}+C_4t\right)
\]
holds for all $t>0$ and $\mathfrak{m}$-a.e. $x,y\in {X}$.
\end{thm}

{}{\begin{remark}\label{aaaaarmk2.9}
The theories of \cite{D97} are also applicable to RCD$(K,N)$ spaces. In particular, under the assumption of Theorem \ref{thm2.12}, for any $x,y\in {X}$, the function $t\mapsto \rho(x,y,t)$ is analytic. Moreover, for any $n\geqslant 1$, $t\in (0,1)$, and $x,y\in {X}$, the Bishop-Gromov inequality (\ref{BGinequality}), Theorem \ref{thm2.12} and \cite[Theorem 4]{D97} give that, 
\begin{align}\label{aabbeqn3.7}
\left|\frac{\partial^n}{\partial t^n}\rho(x,y,t)\right|\leqslant \frac{C(K,N)n!}{t^n }\left(\mathfrak{m}(B_{\sqrt{t}}(x))\mathfrak{m}(B_{\sqrt{t}}(y))\right)^{-\frac{1}{2}}\exp\left(-\frac{\mathsf{d}^2(x,y)}{100t}\right).
\end{align}
\end{remark}}

For a compact $\mathrm{RCD}(K,N)$ space $({X},\mathsf{d},\mathfrak{m})$, by \cite{J14,JLZ16}, its heat kernel $\rho$ can be expressed as follows. See also \cite[Appendix]{AHPT21}.
\begin{equation}\label{heatkernel}
\rho(x,y,t)=\sum\limits_{i= 0}^\infty e^{-\mu_i t}\phi_i(x) \phi_i(y) ,
\end{equation}
where eigenvalues of $-\Delta$ counted with multiplicities and the corresponding eigenfunctions are set as follows.
\begin{equation}\label{notation2.7}
\left\{
\begin{aligned}
&0=\mu_0<\mu_1\leqslant \mu_2 \leqslant  \cdots \rightarrow +\infty,\\
&-\Delta \phi_i=\mu_i\phi_i,\\
&\{\phi_i\}_{i\in \mathbb{N}}: \text{an orthonormal basis of $L^2(\mathfrak{m})$}.
\end{aligned}
\right.
\end{equation}

We may use (\ref{notation2.7}) in Proposition \ref{heatkernel2}, Proposition \ref{1prop2.23} without explanation.

The following estimates can be obtained by the Gaussian estimates (Theorem \ref{thm2.12}) and {}{are} useful in this paper. See \cite[Appendix]{AHPT21} and \cite{ZZ19}.
\begin{prop}\label{heatkernel2}
Let $({X},\mathsf{d},\mathfrak{m})$ be a compact $\mathrm{RCD}(K,N)$ space with $\mathfrak{m}({X})=1$, then there exist $C_j=C_j(K,N,\mathrm{diam}({X},\mathsf{d})) $ $(j=5,6)$, such that for {}{all} $i\geqslant 1$,
\[
\left\|\phi_i\right\|_{L^\infty(\mathfrak{m})}\leqslant C_5\mu_i^{N/4},\ \ \ \ \left\|\left|\nabla \phi_i\right|\right\|_{L^\infty(\mathfrak{m})}\leqslant C_5\mu_i^{(N+2)/4},\ \ \ \   C_6 i^{2/N}\leqslant \mu_i\leqslant C_5 i^2.
\]
\end{prop}

The rest of this subsection is based on \cite{GH18,GR20}. We first introduce some basic knowledge of the Euclidean cone over metric measure spaces. Then the background of the product space of metric measure spaces follows.

\begin{defn}[Euclidean cone as a metric measure space]
Let $({X},\mathsf{d},\mathfrak{m})$ be an RCD$(N-2,N-1)$ space with $N\geqslant 2$. We define the Euclidean cone over $({X},\mathsf{d},\mathfrak{m})$ as the metric measure space $\left(\text{C}({X}),\mathsf{d}_{\text{C}({X})},\mathfrak{m}_{\text{C}({X})}\right)$ as follows. 
\begin{enumerate}

\item The space $\mathrm{C}({X})$ is defined as $\text{C}({X}):= [0,\infty)\times {X}/\left(\{ 0\}\times{X}\right)$. The origin is denoted by $o^\ast$.

\item For any two points $(r,x)$ and $(s,y)$, the distance between them is defined as
\[
\mathsf{d}_{\text{C}({X})}\left((r,x),(s,y)\right):=\sqrt{r^2+s^2-2rs \cos\left(\mathsf{d}(x,y)\right)}.
\]

\item The measure of $\mathrm{C}({X})$ is defined as {}{$\mathrm{d}\mathfrak{m}_{\text{C}({X})}(r,x)=r^{N-1}\mathrm{d}r\otimes \mathrm{d}\mathfrak{m}(x)$.}
\end{enumerate}
\end{defn}
 
\begin{remark}\label{rmk2.10}
If $({X},\mathsf{d},\mathfrak{m})$ is an RCD$(N-2,N-1)$ space, then it has an upper diameter bound $\pi$ due to {}{\cite[Theorem 4.3]{O07}}. In addition, by \cite[Theorem 1.1]{K15a}, $\left(\text{C}({X}),\mathsf{d}_{\text{C}({X})},\mathfrak{m}_{\text{C}({X})}\right)$ is an RCD$(0,N)$ space {}{if and only if} $({X},\mathsf{d},\mathfrak{m})$ is an RCD$(N-2,N-1)$ space. 
\end{remark}
 
By \cite[Definition 3.8, Proposition 3.12]{GH18}, for any $f\in H^{1,2}\left(\text{C}({X}),\mathsf{d}_{\text{C}({X})},\mathfrak{m}_{\text{C}({X})}\right)$, it holds that
\[
\left(f^{(x)}:r\longmapsto f(r,x)\right)\in H^{1,2}(\mathbb{R},\mathsf{d}_\mathbb{R},{}{r^{N-1}}\mathcal{L}^1), \ \  \mathfrak{m}\text{-a.e.}\  x\in {X},
\] 
\[
\left(f^{(r)}:x\longmapsto f(r,x)\right)\in H^{1,2}({X},\mathsf{d},\mathfrak{m}),\ \ \ \ {}{r^{N-1}}\mathcal{L}^1\text{-a.e.}\  r\in \mathbb{R},
\]
and $\left|\nabla f\right|^2_{\text{C}({X})}$ can be written as
\[
\left|\nabla f\right|^2_{\text{C}({X})}(r,x)=\left|\nabla f^{(x)}\right|^2_{\mathbb{R}}(r)+\frac{1}{r^2}\left|\nabla f^{(r)}\right|^2_{{X}}(x) \ \text{$\mathfrak{m}_{\text{C}({X})}$-a.e.}\ (r,x)\in \text{C}({X}). 
\]

Thus for any  $f_1, f_2 \in H^{1,2}\left(\text{C}({X}),\mathsf{d}_{\text{C}({X})},\mathfrak{m}_{\text{C}({X})}\right)$, it can be readily checked that for $\text{$\mathfrak{m}_{\text{C}({X})}$-a.e.}\ (r,x)\in \text{C}({X})$,
\begin{equation}\label{neiji1}
\left\langle \nabla f_1 ,\nabla f_2 \right\rangle_{\text{C}({X})}(r,x)= \left\langle \nabla f_1^{(x)},\nabla f_2^{(x)}\right\rangle_{\mathbb{R}}(r)+\frac{1}{r^2}\left\langle \nabla f_1^{(r)},\nabla f_2^{(r)}\right\rangle_{{X}}(x).
\end{equation}

In addition, the heat kernel $\rho^{\text{C}({X})}$ on $\left(\text{C}({X}),\mathsf{d}_{\text{C}({X})},\mathfrak{m}_{\text{C}({X})}\right)$ has the following explicit expression as {}{ \cite[Theorem 6.20]{D02}}. 
\begin{prop}\label{1prop2.23}
Let $({X},\mathsf{d},\mathfrak{m})$ be a compact $\mathrm{RCD}(N-2,N-1)$ space with $N\geqslant 3$. Let $\alpha=(2-N)/2$, $\nu_j=\sqrt{\alpha^2+\mu_j}$ for $j\in \mathbb{N}$. Then $\rho^{\text{C}({X})}$ can be written as follows:  
\begin{equation}\label{Ding}
\rho^{\text{C}({X})}\left((r_1,x_1),(r_2,x_2),t\right)=(r_1 r_2)^\alpha \sum\limits_{j=0}^\infty  \frac{1}{2t} \exp\left(-\frac{r_1^2+r_2^2}{4t}\right)I_{\nu_j}\left(\frac{r_1 r_2}{2t}\right) \phi_j(x_1)\phi_j(x_2).
\end{equation}

Here $I_{\nu}$ is a modified Bessel function defined  by

\begin{equation}\label{Bessel}
 I_{\nu}(z)=\sum\limits_{k=0}^\infty  \frac{1}{k! \Gamma(\nu+k+1)}\left(\frac{z}{2}\right)^{2k+\nu}.
\end{equation}
\end{prop}

\begin{proof}

We claim that for any $f\in C_c(\mathrm{C}({X}))$, by using $\rho^{\mathrm{C}({X})}$ defined in (\ref{Ding}), ${}{\mathrm{h}_t} f$ can be expressed as follows.
\begin{equation}\label{1111eqn2.11}
{}{\mathrm{h}_t} f(r_1,x_1)=\int_{\mathrm{C}({X})}\rho^{\mathrm{C}({X})}((r_1,x_1),(r_2,x_2),t)f(r_2,x_2) \mathrm{d}\mathfrak{m}_{\mathrm{C}({X})}(r_2,x_2).
\end{equation}

Then we are done by combining (\ref{111eqn2.4}) and the fact that $C_c(\text{C}({X}))$ is dense in $L^2\left(\mathfrak{m}_{\text{C}({X})}\right)$.

To show (\ref{1111eqn2.11}), {}{we first set} $u_i(r)=\int_{X} f(r,x)\phi_i(x)\mathrm{d}\mathfrak{m}(x)$ $(i=0,1,\cdots)$. For any $r\in (0,\infty)$, since $f^{(r)}$ is continuous, by Parseval's identity we have 

\[
{}{\sum\limits_{i=0}^\infty u_i^2(r)=\int_{X}\sum\limits_{i=0}^\infty u_i^2(r)\phi_i^2(x)\mathrm{d}\mathfrak{m}(x)= \int_{X} f^2(r,x)\mathrm{d}\mathfrak{m}(x).}
\]
{}{Letting} 
$f_k(r):=\sum\limits_{i=0}^k r^{N-1}u_i^2(r)$, and using the dominated convergence theorem, we get
\[
\lim\limits_{k\rightarrow \infty}\int_{(0,\infty)} f_k(r)\mathrm{d}r=\int_{(0,\infty)}\int_{X} r^{N-1} f^2(r,x)\mathrm{d}\mathfrak{m}(x)\mathrm{d}r.
\]

This yields  
\[
\begin{aligned}
\ &\lim\limits_{k\rightarrow \infty}\int_{\mathrm{C}({X})}\left(f(r,x)-\sum\limits_{i=0}^k u_i(r)\phi_i(x) \right)^2\mathrm{d}\mathfrak{m}_{\mathrm{C}({X})}(r,x)\\
=&\lim\limits_{k\rightarrow \infty}\left(\int_{(0,\infty)}\int_{X} r^{N-1} f^2(r,x)\mathrm{d}\mathfrak{m}(x)\mathrm{d}r-\int_{(0,\infty)} f_k(r)\mathrm{d}r\right)=0.
\end{aligned}
\]

Therefore $f(r,x)=\sum\limits_{i=0}^\infty u_i(r)\phi_i(x) $ {}{for $\mathfrak{m}_{\mathrm{C}({X})}$-a.e. $(r,x)\in \mathrm{C}({X})$}. Applying the separation of variables in classical ways like \cite[Chapter 8]{Ta96}, we complete the proof of (\ref{1111eqn2.11}).

\end{proof}

\begin{defn}[Cartesian product as a metric measure space]\label{cp1}
{}{Let $({X},\mathsf{d}_{X},\mathfrak{m}_{X})$, $({Y},\mathsf{d}_{Y},\mathfrak{m}_{Y})$ be two metric measure spaces. The product metric measure space $({X}\times {Y} ,\mathsf{d}_{{X}\times {Y} }, \mathfrak{m}_{{X}\times {Y} })$} is defined as the product space ${X}\times {Y} $ equipped with the distance 
\[
\mathsf{d}_{{X}\times {Y} }\left((x_1,y_1),(x_2,y_2)\right)=\sqrt{\mathsf{d}_{X}^2(x_1,x_2)+\mathsf{d}_{Y}^2(y_1,y_2)},\ \ \forall (x_1,y_1),(x_2,y_2)\in {X}\times {Y}, 
\]
 and the measure {}{$\mathrm{d} \mathfrak{m}_{{X}\times {Y} }:=\mathrm{d}\mathfrak{m}_{X} \otimes \mathrm{d}\mathfrak{m}_{Y}$.}
\end{defn}

Since \cite[Proposition 4.1]{GR20} applies for RCD$(K,\infty)$ spaces, for any $f\in H^{1,2}\left({X}\times {Y} ,\mathsf{d}_{{X}\times {Y} }, \mathfrak{m}_{{X}\times {Y} }\right)$,  it holds that 
\[
\left(f^{(x)}:y\longmapsto f(x,y)\right)\in H^{1,2}({Y},\mathsf{d}_{Y},\mathfrak{m}_{Y}),\   \mathfrak{m}_{X}\text{-a.e.}\  x\in{X}{}{,}
\]
\[
\left(f^{(y)}:x\longmapsto f(x,y)\right)\in H^{1,2}({X},\mathsf{d}_{X},\mathfrak{m}_{X}),\   \mathfrak{m}_{Y}\text{-a.e.}\  y\in{Y}{}{,}
\]
and $|\nabla f|^2_{{X}\times {Y} }$ can be expressed as
\begin{equation}\label{2.27}
\left|\nabla f\right|^2_{{X}\times {Y} }(x,y)=\left|\nabla f^{(y)}\right|^2_{{X}}(x)+\left|\nabla f^{(x)}\right|^2_{{Y}}(y), \text{ $\mathfrak{m}_{{X}\times {Y} }$-a.e. }(x,y)\in {X}\times {Y}.
\end{equation}

Thus for any  $f_1, f_2 \in H^{1,2}\left({X}\times {Y} ,\mathsf{d}_{{X}\times {Y} }, \mathfrak{m}_{{X}\times {Y} }\right)$, we have the following for $\text{ $\mathfrak{m}_{{X}\times {Y} }$-a.e. }(x,y)\in {X}\times {Y}$:
\begin{equation}\label{1234eqn2.9}
\left\langle \nabla f_1 ,\nabla f_2 \right\rangle_{{X}\times {Y} }(x,y)= \left\langle \nabla f_1^{(y)},\nabla f_2^{(y)}\right\rangle_{{X}}(x)+\left\langle \nabla f_1^{(x)},\nabla f_2^{(x)}\right\rangle_{{Y}}(y).
\end{equation}

It also follows from \cite[Corollary 4.2]{GR20} that for any $f\in L^2(\mathfrak{m}_{{X}\times {Y} })$,
\[
\text{h}_t^{{X}\times {Y} }f=\text{h}_t^{X} \left(\text{h}_t^{Y} f^{(x)}\right)=\text{h}_t^{Y} \left(\text{h}_t^{X} f^{(y)}\right).
\]

As a result, $\rho^{{X}\times {Y} }$ has an explicit expression as follows.{}{
\begin{equation}\label{eqn2.1}
\rho^{{X}\times {Y} }((x_1,y_1),(x_2,y_2),t)=\rho^{X}(x_1,x_2,t) \rho^{Y}(y_1,y_2,t).
\end{equation}}

\subsection{First and second order calculus on RCD($K,N$) spaces}

This subsection is based on \cite{G18}.  We assume that $({X},\mathsf{d},\mathfrak{m})$ is an RCD($K,N$) space in this subsection. 

\begin{defn}[$L^p$-normed $L^\infty$-module]\label{module}
For any $p\in [1,\infty]$, a quadruplet $\left(\mathscr{M},\left\|\cdot\right\|_{\mathscr{M}},\cdot,|\cdot|\right)$ is said to be an $L^p$-normed $L^\infty$-module if it satisfies the following conditions.
\begin{enumerate}
\item The normed vector space $\left(\mathscr{M},\left\|\cdot\right\|_{\mathscr{M}}\right)$ is a Banach space.
\item The multiplication by $L^\infty$-functions $\cdot:L^\infty(\mathfrak{m})\times\mathscr{M}\rightarrow \mathscr{M}$ is a bilinear map such that for every $ f,h\in L^\infty(\mathfrak{m})$ and every $v\in\mathscr{M}$, it holds that
\[
f\cdot (h\cdot v)=(fh)\cdot v, \ \ \hat{1}\cdot v=v.
\]

\item The pointwise norm $|\cdot|:\mathscr{M}\rightarrow L^p(\mathfrak{m})$ satisfies that for every $ f\in L^\infty(\mathfrak{m})$ and every $v\in\mathscr{M}$, it holds that 
\[
|v|\geqslant 0,\  |f\cdot v|=|f\|v|\ \  \mathfrak{m}\text{-a.e.},\ \text{and}\ \ \|v\|_\mathscr{M}=\left\||v|\right\|_{L^p(\mathfrak{m})}.
\]
\end{enumerate}
In particular, $\left(\mathscr{M},\left\|\cdot\right\|_{\mathscr{M}},\cdot,|\cdot|\right)$ is said briefly to be a module when $p=2$.
\end{defn}
\begin{remark} 
The homogeneity and subadditivity of $|\cdot|$ follows directly from Definition \ref{module}. Write $fv$ instead of $f\cdot v$ later on for simplicity.
\end{remark}

To construct the cotangent module, the first step is to define a pre-cotangent module $\mathsf{Pcm}$. Elements of $\mathsf{Pcm}$ are of the form $\left\{ (E_i ,f_i )\right\}_{i=1}^n$ where $\left\{E_i\right\}_{i=1}^n$ is some Borel partition of ${X}$ and $\left\{f_i\right\}_{i=1}^n\subset H^{1,2}({X},\mathsf{d},\mathfrak{m})$. Secondly, define an equivalence relation on $\mathsf{Pcm}$ as follows. 
\[
\left\{(E_i,f_i)\right\}_{i=1}^n\sim \left\{(F_i,h_i)\right\}_{j=1}^m \text{{}{if and only if for any}}\  i,j, \ |\text{D}f_i|=|\text{D}h_j| \text{ holds $\mathfrak{m}$-a.e. on $E_i\cap F_j$}. 
\]

Denote by $\left[E_i,f_i\right]_i$ the equivalence class of $\left\{(E_i,f_i)\right\}_{i=1}^n$ and by $\chi_E$ the characteristic function of $E$ for any Borel set $E\subset {X}$. 

With the help of the locality of minimal relaxed slopes, the following operations on the quotient $\mathsf{Pcm}/\sim$ are well-defined:
\[
\begin{aligned}
\left[E_i,f_i\right]_i+\left[F_j,g_j\right]_j&:=\left[E_i\cap F_j,f_i+g_j\right]_{i,j},\\
\alpha \left[E_i,f_i\right]_i&:=\left[E_i,\alpha f_i\right]_i,\\
\left(\sum\limits_j \alpha_j \chi_{F_j}\right)\cdot \left[E_i,f_i\right]_i&:=\left[E_i\cap F_j,\alpha_j f_i\right]_{i,j},\\
\left|\left[E_i,f_i\right]_i\right|&:=\sum\limits_i \chi_{E_i}|\text{D}f_i|\ \mathfrak{m}\text{-a.e. in }{X},\\
\left\|\left[E_i,f_i\right]_i\right\|&:=\left\|\left|[E_i,f_i]_i\right|\right\|_{L^2(\mathfrak{m})}=\left(\sum\limits_i \int_{E_i}|\text{D}f_i|^2\mathrm{d}\mathfrak{m}\right)^{\frac{1}{2}}.
\end{aligned}
\]

Let $\left(L^2(T^\ast ({X},\mathsf{d},\mathfrak{m})),\|\cdot\|_{L^2(T^\ast ({X},\mathsf{d},\mathfrak{m}))}\right)$ be the completion of $\left(\mathsf{Pcm}/\sim,\left\|\cdot\right\|\right)$. The multiplication $\cdot$ and the pointwise norm $|\cdot|$ in Definition \ref{module} can be continuously extended to
 \[
\begin{aligned}
\cdot&:L^\infty(\mathfrak{m})\times L^2(T^\ast ({X},\mathsf{d},\mathfrak{m}))\rightarrow L^2(T^\ast ({X},\mathsf{d},\mathfrak{m})),\\
|\cdot|&: L^2(T^\ast ({X},\mathsf{d},\mathfrak{m}))\rightarrow L^2(\mathfrak{m}).\\
\end{aligned}
\]
Then the construction of the module $\left(L^2(T^\ast ({X},\mathsf{d},\mathfrak{m})),\left\|\cdot\right\|_{L^2(T^\ast ({X},\mathsf{d},\mathfrak{m}))}, \cdot ,|\cdot|\right)$ is completed. {}{We write $L^2(T^\ast ({X},\mathsf{d},\mathfrak{m}))$ for short if no ambiguity is caused.}

\begin{thm}[Uniqueness of cotangent module]
There is a unique couple $\left(L^2(T^\ast ({X},\mathsf{d},\mathfrak{m})),d\right)$, where $L^2(T^\ast ({X},\mathsf{d},\mathfrak{m}))$ is a module and $d:H^{1,2}({X},\mathsf{d},\mathfrak{m})\rightarrow L^2(T^\ast ({X},\mathsf{d},\mathfrak{m}))$ is a linear operator such that $|df|=|\mathrm{D}f|$ holds $\mathfrak{m}$-a.e. for every $f\in H^{1,2}({X},\mathsf{d},\mathfrak{m})$.
Uniqueness is intended up to unique isomorphism: if another couple $(\mathscr{M},d')$ satisfies the same properties, then there exists a unique module isomorphism $\zeta:L^2(T^\ast ({X},\mathsf{d},\mathfrak{m}))\rightarrow \mathscr{M}$ such that $\zeta\circ d=d'$.
\end{thm}

In this paper, $L^2\left(T^\ast({X},\mathsf{d},\mathfrak{m})\right)$ and $d$ are called the cotangent module and the differential respectively. Elements of $L^2\left(T^\ast({X},\mathsf{d},\mathfrak{m})\right)$ are called 1-forms. 

Likewise, the tangent module $L^2(T({X},\mathsf{d},\mathfrak{m}))$ can be defined as a module generated by $\{\nabla f
:\ f\in H^{1,2} ({X},\mathsf{d},\mathfrak{m})\}$, where $\nabla f$ satisfies that 
\[
dh(\nabla f)=\langle \nabla h,\nabla f\rangle\ \ \mathfrak{m}\text{-a.e.},  \ \ \forall\ h\in H^{1,2}({X},\mathsf{d},\mathfrak{m}).
\]
  
$L^2(T({X},\mathsf{d},\mathfrak{m}))$ is the dual module of $L^2(T^\ast ({X},\mathsf{d},\mathfrak{m}))$, and its elements are called vector fields.

Let us recall the construction of the tensor product of $L^2(T^\ast ({X},\mathsf{d},\mathfrak{m}))$ with itself in \cite{G18}.

For any $f\in L^\infty(\mathfrak{m}),f_1,f_2\in \mathrm{Test}F\left({X},\mathsf{d},\mathfrak{m}\right)$, the tensor $f  d f_1\otimes d f_2$ is defined as
\[
f d f_1\otimes d f_2(\eta_1,\eta_2):=f df_1(\eta_1)  df_2(\eta_2), \  \forall \eta_1,\eta_2\in L^2(T({X},\mathsf{d},\mathfrak{m})).
\]
  Set

\[
\text{Test}(T^\ast)^{\otimes 2}({X},\mathsf{d},\mathfrak{m}):=\left\{ \sum\limits_{i=1}^k f_{1,i}df_{2,i}\otimes d f_{3,i}:\ k\in \mathbb{N},f_{j,i}\in \mathrm{Test}F\left({X},\mathsf{d},\mathfrak{m}\right)\right\}.
\]
and define the $L^\infty(\mathfrak{m})$-bilinear norm 
\[
\left\langle\cdot ,\cdot 
\right\rangle: \text{Test}(T^\ast)^{\otimes 2}({X},\mathsf{d},\mathfrak{m})\times \text{Test}(T^\ast)^{\otimes 2}({X},\mathsf{d},\mathfrak{m}) \rightarrow L^2(\mathfrak{m})
\]
as 
\[
\langle d f_1\otimes d f_2,df_3\otimes d f_4\rangle:= \langle \nabla f_1,\nabla f_3\rangle \langle \nabla f_2,\nabla f_4\rangle, \  \forall f_i\in \mathrm{Test}F\left({X},\mathsf{d},\mathfrak{m}\right)\  (i=1,2,3,4).
\]

{}{The pointwise Hilbert-Schmidt norm is then defined as 
\[
\begin{aligned}
\left|\cdot\right|_{\mathsf{HS}}:\text{Test}(T^\ast)^{\otimes 2}({X},\mathsf{d},\mathfrak{m})&\longrightarrow L^2(\mathfrak{m})\\
A&\longmapsto |A|_{\mathsf{HS}}:=\sqrt{\langle A,A\rangle}. 
\end{aligned}
\]
}

For any $p\in [1,\infty]$, adapting a similar continuous extension procedure of $\text{Test}(T^\ast)^{\otimes 2}({X},\mathsf{d},\mathfrak{m})$ with respect to the norm $\left\|\left|\cdot\right|_{\mathsf{HS}}\right\|_{L^p(\mathfrak{m})}$ gives a construction of the $L^p$-normed $L^\infty$-module $L^p((T^\ast)^{\otimes 2}({X},\mathsf{d},\mathfrak{m}))$.

In addition, denote by $L^p_{\text{loc}}(T^\ast({X},\mathsf{d},\mathfrak{m}))$ the collection of 1-forms $\omega$ with $|\omega|\in L^p_{\text{loc}}(\mathfrak{m})$. Here $L^p_{\mathrm{loc}}(\mathfrak{m})$ is the set of all functions $f$ such that $f\in L^p\left(B_R(x),\mathfrak{m}\right)$ for any $B_R(x)\subset {X}$.  Similarly for other vector fields and other tensors.

The end of this subsection is  {}{aimed at recalling} definitions of two kinds of tensor fields.

\begin{thm}[The Hessian \cite{G18}]
For any $f\in \mathrm{Test}F\left({X},\mathsf{d},\mathfrak{m}\right)$, there exists a unique $T\in L^2\left((T^\ast)^{\otimes 2}({X},\mathsf{d},\mathfrak{m})\right)$, called the Hessian of $f$, denoted by $ \mathop{\mathrm{Hess}}f$, such that for all $f_i\in \mathrm{Test}F\left({X},\mathsf{d},\mathfrak{m}\right)$ $(i=1,2)$,
\begin{equation}
{}{2T(\nabla f_1,\nabla f_2)= \langle \nabla f_1,\nabla\langle \nabla f_2,\nabla f\rangle\rangle +\langle \nabla f_2,\nabla\langle \nabla f_1,\nabla f\rangle\rangle-\langle \nabla f,\nabla\langle \nabla f_1,\nabla f_2\rangle\rangle }
\end{equation}
holds for $\mathfrak{m}$-a.e. $x\in {X}$. Moreover, the following holds for any $f\in \mathrm{Test}F\left({X},\mathsf{d},\mathfrak{m}\right)$, $\varphi\in \mathrm{Test}F_+({X},\mathsf{d},\mathfrak{m})$.

\begin{equation}\label{abc2.14}
\frac{1}{2}\int_{X}  \Delta \varphi \cdot |\nabla f|^2\mathrm{d}\mathfrak{m}\geqslant \int_{X}\varphi \left(|\mathop{\mathrm{Hess}}f|_{\mathsf{HS}}^2+ \langle \nabla \Delta f,\nabla f\rangle+K|\nabla f|^2\right) \mathrm{d}\mathfrak{m}.
\end{equation}

\end{thm}

Since $\mathrm{Test}F({X},\mathsf{d},\mathfrak{m})$ is dense in $D(\Delta)$, $\mathop{\mathrm{Hess}}f\in L^2\left((T^\ast)^{\otimes 2}({X},\mathsf{d},\mathfrak{m})\right)$ is well-defined for any $f\in D(\Delta)$. In addition, if $f_i\in \mathrm{Test}F\left({X},\mathsf{d},\mathfrak{m}\right)$ $(i=1,2)$, then $\langle \nabla f_1,\nabla f_2 \rangle\in H^{1,2}({X},\mathsf{d},\mathfrak{m})$, and the following holds for any $
\varphi\in H^{1,2}({X},\mathsf{d},\mathfrak{m})$.
\begin{equation}\label{11eqn2.16}
\langle \nabla \varphi, \nabla \langle \nabla f_1,\nabla f_2 \rangle \rangle= \mathop{\mathrm{Hess}}f_1\left(\nabla f_2,\nabla\varphi\right)+ \mathop{\mathrm{Hess}}f_2\left(\nabla f_1,\nabla\varphi\right) \ \ \mathfrak{m}\text{-a.e.}
\end{equation}
\begin{defn}[The Riemannian metric]
A tensor field $\bar{g}\in L^\infty_\text{loc}((T^\ast)^{\otimes 2}({X},\mathsf{d},\mathfrak{m}))$ is said to be a (resp. semi) Riemannian metric on $({X},\mathsf{d},\mathfrak{m})$ if it satisfies the following properties.
\begin{enumerate}
\item (Symmetry) $\bar{g}(V,W)=\bar{g}(W,V)$ $\mathfrak{m}$-a.e. for any $V,W\in L^2(T({X},\mathsf{d},\mathfrak{m}))$.
\item (Non (resp. {}{Non semi-}) degeneracy) For any $V\in L^2(T({X},\mathsf{d},\mathfrak{m}))$, it holds that
\[
\bar{g}\left(V,V\right)>0\ \ (\text{resp.}\ \bar{g}\left(V,V\right)\geqslant 0) \ \ \mathfrak{m}\text{-a.e. on}\ \left\{|V|>0\right\}. 
\] 
\end{enumerate}
\end{defn}

\subsection{Convergence of RCD$(K,N)$ spaces}

For a sequence of pointed RCD$(K,N)$ spaces $({X}_i,\mathsf{d}_i,\mathfrak{m}_i,x_i)$, the equivalence between pointed measured Gromov Hausdorff (pmGH) convergence and pointed measured Gromov (pmG) convergence is established in \cite{GMS13}. We only introduce the definition of pmGH convergence and a precompactness theorem of a sequence of pointed RCD$(K,N)$ spaces. It is remarkable that for compact metric measure spaces there is a more convenient convergence named measured Gromov-Hausdorff (mGH) convergence (see \cite{F87}).

\begin{defn}[Pointed measured Gromov-Hausdorff (pmGH) convergence]\label{1defn2.5}
A sequence of pointed metric measure spaces $\{({X}_i,\mathsf{d}_i,\mathfrak{m}_i,x_i)\}$ is said to be convergent to a pointed metric measure space $ ({X}_\infty,\mathsf{d}_\infty,\mathfrak{m}_\infty,x_\infty)$ in the pointed measured Gromov-Hausdorff (pmGH) sense, if there {}{exists} a complete separable metric space $({Y},\mathsf{d}_{Y})$ and a sequence of isometric embeddings $\{\iota_i:{X}_i\rightarrow {Y}\}_{i\in \mathbb{N}\cup \{\infty\}}$, such that
\begin{enumerate}
\item $\mathsf{d}_{Y}(\iota_i(x_i), \iota_\infty(x_\infty))\rightarrow 0${}{,}
\item for any $R,\epsilon>0$, there exists $N>0$, such that for any $i>N$, we have $\iota_\infty\left(B_R^{{X}_\infty}(x_\infty)\right)\subset B^{Y}_\epsilon \left(\iota_i\left(B_R^{{X}_i}(x_i)\right)\right) $ and {}{$\iota_i\left(B_R^{{X}_i}(x_i)\right)\subset B^{Y}_\epsilon \left(\iota_\infty\left(B_R^{{X}_\infty}(x_\infty)\right)\right) $,}
\item for every {}{$f\in C_{c}({Y})$}, $\lim\limits_{i\rightarrow \infty}\int_{Y}f \mathrm{d}(\iota_i)_\sharp \mathfrak{m}_i= \int_{Y} f \mathrm{d}(\iota_\infty)_\sharp \mathfrak{m}_\infty$.
\end{enumerate}
In particular, we say that $ X_i\ni x_i'\rightarrow x_\infty'\in X_\infty$ if $\mathsf{d}_{Y}\left(\iota_i(x_i'), \iota_\infty(x_\infty')\right)\rightarrow 0$.
\end{defn}

\begin{defn}[Measured Gromov-Hausdorff convergence]
Let $\{ ({X}_i,\mathsf{d}_i,\mathfrak{m}_i)\}$ be a sequence of compact metric measure spaces with {}{$\sup_i \mathrm{diam}({X}_i,\mathsf{d}_i)<\infty$}. Then $\{ ({X}_i,\mathsf{d}_i,\mathfrak{m}_i)\}$ is said to be convergent to a metric measure space $({X}_\infty,\mathsf{d}_\infty,\mathfrak{m}_\infty)$ in the measured Gromov-Hausdorff (mGH) sense if there exists a sequence of points $\{x_i\in {X}_i\}_{i\in \mathbb{N}\cup \{\infty\}}$, such that  
\[
({X}_i,\mathsf{d}_i,\mathfrak{m}_i,x_i)\xrightarrow{\mathrm{pmGH}}({X}_\infty,\mathsf{d}_\infty,\mathfrak{m}_\infty,x_\infty).
\]
\end{defn}

\begin{thm}[Precompactness of pointed RCD$(K,N)$ spaces under pmGH convergence \cite{GMS13}]\label{sta}
Let $\left\{({X}_i,\mathsf{d}_i,\mathfrak{m}_i,x_i)\right\}$ be a sequence of pointed $\mathrm{RCD}(K,N)$ spaces such that
\[
0<\liminf\limits_{i\rightarrow \infty} \mathfrak{m}_i\left(B_1^{{X}_i}(x_i)\right)<\limsup\limits_{i\rightarrow \infty} \mathfrak{m}_i\left(B_1^{{X}_i}(x_i)\right)<\infty.
\]
Then there exists a subsequence $\left\{\left({X}_{i(j)},\mathsf{d}_{i(j)},\mathfrak{m}_{i(j)},x_{i(j)}\right)\right\}$, such that it $\mathrm{pmGH}$ converges to a pointed $\mathrm{RCD}(K,N)$ space $({X}_\infty,\mathsf{d}_\infty,\mathfrak{m}_\infty,x_\infty)$.  
\end{thm}
{}{Especially, non-collapsed pmGH convergent sequences of non-collapsed RCD$(K,N)$ spaces preserve the Hausdorff measure.}
\begin{thm}[Continuity of Hausdorff measure {\cite[Theorem 1.3]{DG18}}]\label{11thm2.15}

If a sequence of pointed non-collapsed $\mathrm{RCD}(K,N)$ spaces $\left\{\left({X}_i,\mathsf{d}_i,\mathcal{H}^N,x_i\right)\right\}$ $\mathrm{pmGH}$ converges to a pointed $\mathrm{RCD}(K,N)$ space $ ({X}_\infty,\mathsf{d}_\infty,\mathfrak{m}_\infty,x_\infty)$ and satisfies $\inf_i \mathcal{H}^N\left(B_1^{{X}_i}(x_i)\right)>0$, then $\mathfrak{m}_\infty=\mathcal{H}^N$.
\end{thm}
It is also worth recalling the following definition.

\begin{defn}[Regular set]\label{111def2.18}
Let $({X},\mathsf{d},\mathfrak{m})$ be an RCD$(K,N)$ space. Given any integer $k\in [1,N]$,  the $k$-dimensional regular set $\mathcal{R}_k:=\mathcal{R}_k({X})$ of ${X}$ is defined as the set of all points of $x$ such that 
\[
\left({X},\frac{1}{r_i}\mathsf{d},\frac{\mathfrak{m}}{\mathfrak{m}(B_{r_i}(x))},x\right)\xrightarrow{\mathrm{pmGH}} \left(\mathbb{R}^k,\mathsf{d}_{\mathbb{R}^k},\frac{1}{\omega_k}\mathcal{L}^k,0_k\right)\ \  \forall\{ r_i \}\subset (0,\infty)\  \text{with}\  r_i \rightarrow 0.
\] 
\end{defn}

It is time to introduce the definition of the essential dimension of RCD spaces. Compare \cite{CN12}.
\begin{thm}[Essential dimension \cite{BS20}]\label{1111thm2.22}
Let $({X},\mathsf{d},\mathfrak{m})$ be an $\mathrm{RCD}(K,N)$ space. Then there exists a unique $n\in \mathbb{N}\cap [1,N]$ such that $\mathfrak{m}({X}\setminus \mathcal{R}_n)=0$. The essential dimension $\mathrm{dim}_{\mathsf{d},\mathfrak{m}}({X})$ of $({X},\mathsf{d},\mathfrak{m})$ is defined as this $n$.
\end{thm}

\begin{remark}{}{Under the assumption of Theorem \ref{1111thm2.22}, for any $m\in \mathbb{N}_+$, define the Bishop-Gromov density of $(X,\mathsf{d},\mathfrak{m})$ as
\[
\begin{aligned}
\vartheta_m({X},\mathsf{d},\mathfrak{m}) :{X}&\longrightarrow [0,\infty]\\
x&\longmapsto \left\{\begin{aligned}\lim\limits_{r\rightarrow 0} \frac{\mathfrak{m}(B_r(x))}{\omega_m r^m},&\ \ \text{    if it exists,}\\
\infty, &\ \ \text{  otherwise.}
\end{aligned}
\right.
\end{aligned}
\]
}
The measure $\mathfrak{m}$ then can be represented as $\vartheta_n({X},\mathsf{d},\mathfrak{m})(x) \mathcal{H}^n\llcorner\mathcal{R}_n$. Moreover, $\mathfrak{m}(\mathcal{R}_n\setminus \mathcal{R}_n^\ast)=0$, where $\mathcal{R}_n^\ast:=\left\{x\in \mathcal{R}_n: \vartheta_n({X},\mathsf{d},\mathfrak{m})\in (0,\infty)\right\}$. See \cite{AHT18}.

\end{remark}
In particular, for non-collapsed RCD$(K,N)$ spaces, the following statement holds.

\begin{thm}[Bishop inequality {\cite[Corollary 1.7]{DG18}}]\label{1111thm2.20}
Let $({X},\mathsf{d},\mathcal{H}^N)$ be a non-collapsed $\mathrm{RCD}(K,N)$ space. Then $\mathrm{dim}_{\mathsf{d},\mathcal{H}^N}(X)=N\in \mathbb{N}$, and  $\vartheta_N({X},\mathsf{d},\mathcal{H}^N)\leqslant 1$ holds for any $x\in {X}$. Moreover, the equality holds {}{if and only if} $x\in \mathcal{R}_N$.
\end{thm}

Given an RCD$(K,N)$ space $({X},\mathsf{d},\mathfrak{m})$, there is a canonical Riemannian metric $g$ in the following sense.

\begin{thm}[The canonical Riemannian metric \cite{GP16, AHPT21}]\label{111thm2.21}
There exists a unique Riemannian metric $g$ such that for any $f_1,f_2 \in H^{1,2}({X},\mathsf{d},\mathfrak{m})$, it holds that
\[
g\left(\nabla f_1,\nabla f_2\right)=\left\langle \nabla f_1,\nabla f_2\right\rangle\ \ \text{$\mathfrak{m}$-a.e. in ${X}$}.
\]
Moreover, $\left|g\right|_{\mathsf{HS}}=\sqrt{\mathrm{dim}_{\mathsf{d},\mathfrak{m}}({X})}$ $\mathfrak{m}$-a.e. in ${X}$.
\end{thm}

Let us use this canonical Riemannian metric to define the trace {}{as 
\[
\begin{aligned}
\mathrm{Tr}: L^2_{\text{loc}}\left((T^\ast)^{\otimes 2}({X},\mathsf{d},\mathfrak{m})\right)&\longrightarrow L^2_{\text{loc}}\left((T^\ast)^{\otimes 2}({X},\mathsf{d},\mathfrak{m})\right)\\
T&\longmapsto  \langle T,g\rangle.
\end{aligned}
\]
}

{}{The convergence of functions and tensor fields on pmGH convergent pointed RCD$(K,N)$ spaces are also well-defined} as in \cite{GMS13}, \cite[Definition 1.1]{H15} and \cite{AH17,AST16}. In the rest of this subsection, we assume that $({X}_i,\mathsf{d}_i,\mathfrak{m}_i,x_i)\xrightarrow{\mathrm{pmGH}}({X}_\infty,\mathsf{d}_\infty,\mathfrak{m}_\infty,x_\infty)$, and use the {}{notation} in Definition \ref{1defn2.5}.

\begin{defn}[$L^2$-convergence of functions defined on varying spaces] A sequence $\{f_i:{X}_i\rightarrow \mathbb{R}\}$ is said to be $L^2$-weakly convergent to $f_\infty \in L^2(\mathfrak{m}_\infty)$ if 
 \[
 \left\{
 \begin{aligned}
 &\sup_i \left\|f_i\right\|_{L^2(\mathfrak{m}_i)}<\infty,\\
  &\lim\limits_{i\rightarrow \infty}\int_{Y}hf_i \mathrm{d}(\iota_i)_\sharp \mathfrak{m}_i= \int_{Y} hf_\infty \mathrm{d}(\iota_\infty)_\sharp \mathfrak{m}_\infty, \ \ \forall h\in C_c({Y}).
 \end{aligned}
 \right.
 \]
 If moreover $\{f_i\}$ satisfies $\limsup_{i\rightarrow \infty}\left\|f_i\right\|_{L^2(\mathfrak{m}_i)}\leqslant \left\|f\right\|_{L^2(\mathfrak{m}_\infty)}$, then $\{f_i\}$ is said to be $L^2$-strongly convergent to $f$.
\end{defn}

\begin{defn}[$H^{1,2}$-convergence of functions defined on varying spaces] A sequence $\{f_i:{X}_i\rightarrow \mathbb{R}\}$ is said to be $H^{1,2}$-weakly convergent to $f_\infty \in H^{1,2}({X}_\infty,\mathsf{d}_\infty,\mathfrak{m}_\infty)$ if 
\[
f_i\xrightarrow{L^2\text{-weakly}}f\ \text{and}\  \sup_i \text{Ch}^{{X}_i}(f_i)<\infty.
\]
 If moreover, $\{f_i\}$ satisfies 
 \[
 \limsup_{i\rightarrow \infty}\left\|f_i\right\|_{L^2(\mathfrak{m}_i)}\leqslant \left\|f\right\|_{L^2(\mathfrak{m}_\infty)}\  \text{and}\  \limsup_{i\rightarrow \infty}\text{Ch}^{{X}_i}(f_i)=\text{Ch}^{{X}_\infty}(f_\infty),
 \] 
 then $\{f_i\}$ is said to be $H^{1,2}$-strongly convergent to $f$.

\end{defn}

\begin{defn}[Convergence of tensor fields defined on varying spaces]
Assume  {}{$T_i\in L^2_\mathrm{loc}\left((T^\ast)^{\otimes 2}({X}_i,\mathsf{d}_i,\mathfrak{m}_i)\right)$, $(i\in \mathbb{N})$}. For any $R>0$, $\{T_i\} $ is said to be $L^2$-weakly convergent to $T_\infty\in L^2\left((T^\ast)^{\otimes 2}(B_R^{{X}_\infty}(x_\infty),\mathsf{d}_\infty,\mathfrak{m}_\infty)\right)$ on $B_R^{{X}_\infty}(x_\infty)$ if it satisfies the following conditions.
\begin{enumerate}
\item (Uniform upper $L^2$ bound) $\sup_i \left\||T_i|_{\mathsf{HS}}\right\|_{L^2\left(B_R^{{X}_i}(x_i),\mathfrak{m}_i\right)}<\infty$.
\item For any $f_{j,i}\in \mathrm{Test}F({X}_i,\mathsf{d}_i,\mathfrak{m}_i)$ $(i\in\mathbb{N},\ j=1,2)$ {}{such that} $\{f_{j,i}\}$ $L^2$-strongly converges to $f_{j,\infty}\in \mathrm{Test}F({X}_\infty,\mathsf{d}_\infty,\mathfrak{m}_\infty)$ ($j=1,2$) and that  
\[
\sup_{i,j}\left(\left\|f_{j,i}\right\|_{L^\infty(\mathfrak{m}_i)}+\left\||\nabla^{{X}_i}f_{j,i}|\right\|_{L^\infty(\mathfrak{m}_i)}+\left\|\Delta^{{X}_i}f_{j,i}\right\|_{L^\infty(\mathfrak{m}_i)}\right)<\infty,
\]
we have {}{$\{\chi_{B_R^{{X}_i}(x_i)}\left\langle T_i, df_{1,i}\otimes d f_{2,i}\right\rangle \}$ $L^2$-weakly converges to $\chi_{B_R^{{X}_\infty}(x_\infty)}\langle T_\infty,d f_{1,\infty}\otimes df_{2,\infty} \rangle$.}

\end{enumerate}
If moreover,  $\limsup_{i\rightarrow \infty}\left\||T_i|_{\mathsf{HS}}\right\|_{L^2\left(B_R^{{X}_i}(x_i),\mathfrak{m}_i\right)}\leqslant \left\||{}{T_\infty}|_{\mathsf{HS}}\right\|_{L^2\left(B_R^{{X}_\infty}(x_\infty),\mathfrak{m}_\infty\right)}$, then $\{T_i\}$ is said to be $L^2$-strongly convergent to $T_\infty$ on $B_R^{{X}_\infty}(x_\infty)$.
\end{defn}

Let us recall two convergences to end this section. 

\begin{thm}[$H^{1,2}$-strong convergence of heat kernels {\cite[Theorem 2.19]{AHPT21}}]\label{thm2.26}
For any $\{t_i\}\subset (0,\infty)$ with $t_i\rightarrow t_0 \in (0,\infty)$ and any $\{y_i\}$ with ${X}_i\ni y_i\rightarrow y_\infty \in {X}_\infty$,  $\left\{\rho^{{X}_i}(\cdot,y_i,t_i)\right\}$ $H^{1,2}$-strongly converges to $\rho^{{X}_\infty}(\cdot,y,t)\in H^{1,2}({X}_\infty,\mathsf{d}_\infty,\mathfrak{m}_\infty)$.
\end{thm}

\begin{thm}[Lower semicontinuity of essential dimension {\cite[Theorem 1.5]{K19}}]\label{11thm2.26}
\[
\liminf\limits_{i\rightarrow \infty}\mathrm{dim}_{\mathsf{d}_i,\mathfrak{m}_i}({X}_i)\leqslant \mathrm{dim}_{\mathsf{d}_\infty,\mathfrak{m}_\infty}({X}_\infty).
\]
\end{thm}

\section{The isometric immersion into $L^2$ space via heat kernel}\label{sec3}

Recently the equivalence between weakly non-collapsed RCD spaces and non-collapsed RCD spaces is proved in \cite[Theorem 1.3]{BGHZ21}, which states as follows.

\begin{thm}\label{BGHZmainthm}
Assume that $({X},\mathsf{d},\mathfrak{m})$ is an $\mathrm{RCD}(K,N)$ space. If
\[
\mathfrak{m}\left(\left\{x\in {X}:\limsup\limits_{r\rightarrow 0^+}\frac{\mathfrak{m}(B_r(x))}{r^N}<\infty\right\}\right)>0,
\]
then $\mathfrak{m}=c\mathcal{H}^N$ for some $c>0$. Therefore, $\left({X},\mathsf{d},c^{-1}\mathfrak{m}\right)$ is a non-collapsed $\mathrm{RCD}(K,N)$ space.
\end{thm}

The key to prove Theorem \ref{BGHZmainthm} is Theorem \ref{eqnBGHZ21}, and the asymptotic formula (Theorem \ref{20211222a}) of $g_t$ plays an important role in the proof of Theorem \ref{eqnBGHZ21}. The precise definition of $g_t$ shall be given in Theorem \ref{thm2.18}.

\begin{thm}[{\cite[Theorem 1.5, Theorem 2.22]{BGHZ21}}]\label{eqnBGHZ21}
Assume that $({X},\mathsf{d},\mathcal{H}^n)$ is an $\mathrm{RCD}(K,N)$ space with $\mathrm{dim}_{\mathsf{d},\mathfrak{m}}({X})=n$ and $U$ is a connected open subset of ${X}$ such that for any compact subset $A\subset U$, 

 \begin{equation}\label{BGHZ}
\inf\limits_{r\in (0,1),x\in A}\frac{\mathcal{H}^n\left(B_r(x)\right)}{r^n}>0.
\end{equation}
Then for any $ f\in \mathrm{Test}F\left({X},\mathsf{d},\mathcal{H}^n\right)$,  any $\varphi\in D(\Delta)$ with
  $ \varphi \geqslant 0$, $\text{supp}(\varphi)\subset U$ and $\Delta \varphi \in L^\infty (\mathcal{H}^n)$, it holds that 
 \[
  \frac{1}{2}\int_U |\nabla f|^2 \Delta \varphi \ \mathrm{d}\mathcal{H}^n 
  \geqslant \int_U \varphi \left(\langle \nabla f , \nabla \Delta f \rangle +K |\nabla f|^2   + \frac{(\Delta f)^2}{n}  \right) \mathrm{d}\mathcal{H}^n.
  \]
\end{thm}

In addition, for a weakly non-collapsed (and is now non-collapsed) RCD$(K,n)$ space $({X},\mathsf{d},\mathcal{H}^n)$, it follows from \cite[Theorem 1.12]{DG18} that
\[
\Delta f=\langle \mathop{\mathrm{Hess}}f,g\rangle \ \ \ \mathfrak{m}\text{-a.e.}, \ \forall f\in \text{D}(\Delta).
\]



\subsection{The pullback metric $g_t$}\label{sec3.1}

On $\mathbb{R}^n$, it is obvious that
\begin{equation} 
g_t^{\mathbb{R}^n}=\frac{c_1^{\mathbb{R}^n}}{t^{\frac{n+2}{2}}}g_{\mathbb{R}^n},\ \ \ \text{with } c_1^{\mathbb{R}^n}=\int_{\mathbb{R}^n}\left(\frac{\partial}{\partial x_1}\rho^{\mathbb{R}^n}(x,y,t)\right)^2\mathrm{d}\mathcal{L}^n (y).
\end{equation}

In \cite{Ta66}, Takahashi proves that any compact homogeneous irreducible Riemannian manifold $( M^n,g)$ is IHKI, which is even true provided that $( M^n,g)$ is a non-compact homogeneous irreducible Riemannian manifold. 

To generalize such isometric immersions to RCD$(K,N)$ spaces, let us first introduce the following locally Lipschitz {}{$t$-time heat kernel mapping on an RCD$(K,N)$ space $({X},\mathsf{d},\mathfrak{m})$ by using its heat kernel $\rho$ analogously}	:
\[
\begin{aligned}
\Phi_t:{X}&\longrightarrow L^2(\mathfrak{m})\\
x&\longmapsto \left(y\mapsto \rho(x,y,t)\right),
\end{aligned}
\]
which is well-defined due to the estimates in Theorem \ref{thm2.12}.
 
The natural pull-back semi-Riemannian metric of the flat metric of $L^2(\mathfrak{m})$, namely $g_t:=(\Phi_t)^\ast(g_{L^2(\mathfrak{m})})$, is defined as follows, see \cite[Proposition 4.7]{AHPT21} and \cite[Proposition 3.7]{BGHZ21}. 
\begin{thm}[The pull-back semi-Riemannian metrics]\label{thm2.18}
For all $t>0$, there is a unique semi-Riemannian metric $g_t\in L_{\mathrm{loc}}^\infty\left((T^\ast)^{\otimes 2}({X},\mathsf{d},\mathfrak{m})\right)$ such that 
\begin{enumerate}

\item For any $\eta_i\in L^2\left(T^\ast({X},\mathsf{d},\mathfrak{m})\right)$ with bounded support $(i=1,2)$, 
\[
\int_{{X}} \left\langle g_t,\eta_1 \otimes \eta_2 \right\rangle \mathrm{d}\mathfrak{m}=\int_{{X}} \int_{{X}} \left\langle d_x \rho(x,y,t),\eta_1\right\rangle \left\langle d_x \rho(x,y,t),\eta_2\right\rangle\mathrm{d}\mathfrak{m}(x)\mathrm{d}\mathfrak{m}(y).
\]
In particular, if $({X},\mathsf{d},\mathfrak{m})$ is compact, then $g_t=\sum\limits_{i=1}^\infty e^{-2\mu_i t}d\phi_i\otimes d\phi_i$.

\item For any $t\in (0,1)$, the rescaled semi-Riemannian metric $t\mathfrak{m}(B_{\sqrt{t}}(\cdot))g_t$ satisfies
\begin{equation}\label{tsuikaeqn3.2}
t\mathfrak{m}(B_{\sqrt{t}}(\cdot))g_t\leqslant C(K,N) g,
\end{equation}
which means that for any $\eta\in L^2\left(T^\ast({X},\mathsf{d},\mathfrak{m})\right)$, it holds that
\[
t\mathfrak{m}(B_{\sqrt{t}}(x))\langle g_t,\eta\otimes \eta \rangle (x) \leqslant C(K,N) |\eta|^2(x)\ \  \text{$\mathfrak{m}$-a.e. $x\in {X}$}.
\]

\end{enumerate}
\end{thm}

The rest part of this subsection proves Theorem \ref{thm1.2}. The following inequality is needed. See for instance \cite[Lemma 2.3]{AHPT21} and \cite[Lemma 2.7]{BGHZ21}.
\begin{lem}\label{aaaalem3.11}
Let $({X},\mathsf{d},\mathfrak{m})$ be an $\mathrm{RCD}(K,N)$ space. Then for any $\alpha\in \mathbb{R}$, $\beta>0$ and any $x\in{X}$, it holds that
\begin{equation}
\int_{X}\mathfrak{m}\left(B_{\sqrt{t}}(y)\right)^\alpha \exp\left(-\frac{\beta \mathsf{d}^2(x,y)}{t}\right)\mathrm{d}\mathfrak{m}(y)\leqslant C\left(K,N,\alpha,\beta\right) \mathfrak{m}\left(B_{\sqrt{t}}({}{x})\right)^{\alpha+1}.
\end{equation}
\end{lem}
\begin{remark}
When $({X},\mathsf{d},\mathfrak{m})$ is an RCD$(0,N)$ space, by \cite[Corollary 1.1]{JLZ16} and Lemma \ref{aaaalem3.11}, (\ref{tsuikaeqn3.2}) becomes
\begin{equation}\label{tsukaeqn3.3}
t\mathfrak{m}(B_{\sqrt{t}}(\cdot))g_t\leqslant C(N) g,\ \forall {}{t>0}.
\end{equation}
\end{remark}
Jiang's gradient estimate \cite[Theorem 3.1]{J14} is also important in this paper, which states as follows. 
\begin{thm}\label{aaaathm3.12}
Let $({X},\mathsf{d},\mathfrak{m})$ be an $\mathrm{RCD}(K,N)$ space and $\Omega$ be {}{an} open subset. If for some $u\in D(\Delta)\cap L^\infty(\Omega,\mathfrak{m})$, $\Delta u \in L^\infty(\Omega,\mathfrak{m})$, then for every $B_R(x)$ with $R\leqslant 1$ and $B_{8R}(x)\Subset \Omega$, it holds that
\begin{equation}
\left\| |\nabla u|\right\|_{L^\infty\left(B_{R}(x),\mathfrak{m}\right)}\leqslant C(K,N)\left(\frac{1}{R} \left\| u\right\|_{L^\infty\left(B_{8R}(x),\mathfrak{m}\right)}+ R\left\|\Delta u\right\|_{L^\infty\left(B_{8R}(x),\mathfrak{m}\right)}\right).
\end{equation}  
\end{thm}

Finally, we need the following proposition.
\begin{prop}\label{llem3.4}
{}{Suppose that $({X},\mathsf{d},\mathfrak{m})$ is an $\mathrm{RCD}(K,N)$ space which is not a single point. Then for any $t>0$, 
\[
\mathfrak{m}\left(\{x\in {X}:|g_t|_{\mathsf{HS}}>0\}\right)>0. 
\]}
\end{prop}
\begin{proof}
Assume by contradiction the existence of $t_0>0$ such that $\mathfrak{m}(\{x\in {X}:|g_{t_0}|_{\mathsf{HS}}>0\})=0$. Clearly this implies $|\nabla_x \rho(x,y,t_0)|=0$, $\mathfrak{m}$-a.e. $x,y \in {X}$. For any fixed $x\in{X}$, the locally Lipschitz continuity of $y\mapsto \rho(x,y,t_0)$ as well as the Sobolev to Lipschitz property then yields that $\Phi_{t_0}\equiv c\hat{1}$ for some constant $c$. Therefore, it follows from the stochastic completeness of RCD$(K,N)$ spaces that $\mathfrak{m}({X})<\infty$. Without loss of generality, assume that $\mathfrak{m}({X})=1$. Notice that $\Phi_{2t_0}(x)=h_{t_0}(\Phi_{t_0}(x))\equiv \hat{1}$, which implies $\rho(x,y,t)\equiv 1$ on ${X}\times{X}\times [t_0,2t_0]$ by (\ref{111eqn2.4}). {}{Then applying Remark \ref{aaaaarmk2.9} shows that 
\[
\rho(x,y,t)=1,\ \forall (x,y,t)\in X\times X\times (0,\infty).
\]
As a consequence, for any $f\in L^2(\mathfrak{m})$, we have 
\[
\mathrm{h}_t f =\int_X \rho(x,y,t) f\mathrm{d}\mathfrak{m}= \int_X  f\mathrm{d}\mathfrak{m},\ \forall t>0.
\]
Since $\mathrm{h}_t f$ converges to $f$ in $L^2(\mathfrak{m})$ as $t\rightarrow 0$, $f$ is nothing but a constant function, which is enough to conclude that ${X}$ is a single point. A contradiction.
}

\end{proof}
\begin{proof}[Proof of Theorem \ref{thm1.2}]
{}{Let $n=\mathrm{dim}_{\mathsf{d},\mathfrak{m}(X)}$.} For any fixed $B_R(x_0)\subset {X}$, set
{}{\[
\begin{aligned}
f: (0,\infty)&\longrightarrow [0,\infty)\\
t&\longmapsto n\mathfrak{m}(B_R(x_0))\int_{B_R(x_0)}\langle g_t,g_t\rangle\mathrm{d}\mathfrak{m}-\left(\int_{B_R(x_0)}\langle g,g_t\rangle \mathrm{d}\mathfrak{m}\right)^2.
\end{aligned} 
\]
}

Since we can rescale the space, it suffices to show that $f$ is analytic at any $t\in (0,1)$. {}{Because then by applying Proposition \ref{llem3.4} we are done.}

For any {}{$m\geqslant 1$}, the commutativity  of $\dfrac{\partial}{\partial t}$ and $\Delta $ allows us to fix an arbitrary $y\in {X}$ and apply Theorem \ref{aaaathm3.12} on $B_{8\sqrt{t}}(x)$ for {}{$u:z\mapsto \dfrac{\partial^m}{\partial t^m}\rho(z,y,t)$.} (\ref{aabbeqn3.7}) then implies

\[
\left\||\nabla u| \right\|_{L^\infty(B_{\sqrt{t}}(x),\mathfrak{m})}
\leqslant {}{\frac{C(K,N)m!}{t^{m+\frac{1}{2}} }}\sup\limits_{z\in B_{8\sqrt{t}}(x)}\left(\mathfrak{m}(B_{\sqrt{t}}(z))\mathfrak{m}(B_{\sqrt{t}}(y))\right)^{-\frac{1}{2}}\exp\left(-\frac{\mathsf{d}^2({}{z,y})}{100t}\right).
\]

Using (\ref{BGinequality111}), for any $z\in B_{8\sqrt{t}}(x)$, we know
\[
\frac{\mathfrak{m}\left(B_{\sqrt{t}}(x)\right)}{\mathfrak{m}\left(B_{\sqrt{t}}(z)\right)}\leqslant C(K,N)\exp\left(\frac{\sqrt{t}+\mathsf{d}(x,z)}{\sqrt{t}}\right)\leqslant C(K,N).
\] 
{}{This as well as the inequality $-\mathsf{d}^2(z,y)\leqslant \mathsf{d}^2(z,x)-\dfrac{\mathsf{d}^2(x,y)}{2}$} implies that for $\mathfrak{m}$-a.e. $x\in {X}$,
\begin{equation}\label{aaaaeqn3.8}
\left|\nabla_x {}{\frac{\partial^m}{\partial t^m}}\rho(x,y,t)\right|\leqslant {}{\frac{C(K,N)m!}{t^{m+\frac{1}{2}}}}\left(\mathfrak{m}(B_{\sqrt{t}}(x))\mathfrak{m}(B_{\sqrt{t}}(y))\right)^{-\frac{1}{2}}\exp\left(-\frac{\mathsf{d}^2(x,y)}{{}{200t}}\right).
\end{equation}

Let {}{ $f=n\mathfrak{m}(B_R(x_0))f_1-f_2^2$, with $f_2(t)= \int_{B_R(x_0)}\langle g,g_t\rangle \mathrm{d}\mathfrak{m}$. We only give a proof of the analyticity of $f_1$, since the analyticity of $f_2$ will follow from similar arguments.}

Rewrite {}{$f_1$} as
\[
{}{f_1}(t)=\int_{B_R(x_0)}\int_{X}\int_{X} \left\langle \nabla_x \rho(x,y,t),\nabla_x \rho(x,z,t)\right\rangle^2 \mathrm{d}\mathfrak{m}(z) \mathrm{d}\mathfrak{m}(y) \mathrm{d}\mathfrak{m}(x). 
\]

It is enough to estimate derivatives of each order of ${}{f_1}$ at any fixed $t\in (0,1)$. We first show that {}{$f_1$} is differentiable.
 
For any sufficiently small $s$, {}{$\dfrac{f_1(t+s)-f_1(t)}{s}$} can be written as the sum of the integrals of functions like
\begin{equation}\label{0324eqn1}
 \left\langle \nabla_x \frac{\rho(x,y,t+s)-\rho(x,y,t)}{s},\nabla_x \rho(x,z,t)\right\rangle \left\langle \nabla_x \rho(x,y,t+s),\nabla_x \rho(x,z,t+s)\right\rangle 
\end{equation}
on $B_R(x_0)
\times {X}\times {X}$. 

In order to use the dominated convergence theorem, we need estimates of $\left|\nabla_x \dfrac{\rho(x,y,t+s)-\rho(x,y,t)}{s}\right|$ and $|\nabla_x \rho(x,y,t+s) |$ for any sufficiently small $s$. By Theorem \ref{thm2.12} and the Bishop-Gromov inequality, for $\mathfrak{m}$-a.e. $x\in{X}$, 
\begin{equation}\label{0324eqn3}
\begin{aligned}
|\nabla_x \rho(x,y,t+s) |&\leqslant \dfrac{C(K,N)}{\sqrt{t+s}\ \mathfrak{m}\left(B_{\sqrt{t+s}}(x)\right)}\exp\left(-\dfrac{\mathsf{d}^2(x,y)}{100(t+s)}\right)\\
\ &\leqslant \dfrac{C(K,N)}{\sqrt{t}\ \mathfrak{m}\left(B_{\sqrt{t}}(x)\right)}\dfrac{\mathfrak{m}\left(B_{\sqrt{t}}(x)\right)}{\mathfrak{m}\left(B_{\sqrt{t+s}}(x)\right)}\exp\left(-\dfrac{\mathsf{d}^2(x,y)}{200t}\right) \\
\ &\leqslant \dfrac{C(K,N)}{\sqrt{t}\ \mathfrak{m}\left(B_{\sqrt{t}}(x)\right)}\exp\left(-\dfrac{\mathsf{d}^2(x,y)}{200t}\right) .\\
\end{aligned}
\end{equation}
The last inequality of (\ref{0324eqn3}) is obvious when $s>0$, and is guaranteed by the Bishop-Gromov inequality when $s<0$. 

Applying (\ref{aaaaeqn3.8}), Theorem \ref{aaaathm3.12} and the Lagrange mean value theorem, the {}{following estimate} can also be obtained as in (\ref{0324eqn3}): 
\begin{equation}\label{0324eqn2}
\begin{aligned}
\ &\left|\nabla_x \left(\dfrac{\rho(x,y,t+s)-\rho(x,y,t)}{s}-\dfrac{\partial}{\partial t}\rho(x,y,t)\right)\right|\\
\leqslant\ & \dfrac{C(K,N)2!|s|}{t^{\frac{5}{2}}}\left(\mathfrak{m}\left(B_{\sqrt{t}}(x)\right)\mathfrak{m}\left(B_{\sqrt{t}}(y)\right)\right)^{-\frac{1}{2}}\exp\left(-\dfrac{\mathsf{d}^2(x,y)}{{}{300t}}\right).
\end{aligned}
\end{equation}

 Therefore the $L^1(\mathfrak{m}\otimes \mathfrak{m}\otimes \mathfrak{m})$ convergence of 
(\ref{0324eqn1}) as $s\rightarrow 0$ can be verified by (\ref{0324eqn3}), (\ref{0324eqn2}) and Lemma \ref{aaaalem3.11}. The limit of (\ref{0324eqn1}) as $s\rightarrow 0$ is actually 
\[
\int_{B_R(x_0)\times {X}\times {X}}\left\langle \nabla_x \frac{\partial}{\partial t}\rho(x,y,t),\nabla_x \rho(x,z,t)\right\rangle  \left\langle \nabla_x \rho(x,y,t),\nabla_x \rho(x,z,t)\right\rangle \mathrm{d}\mathfrak{m}(z) \mathrm{d}\mathfrak{m}(y) \mathrm{d}\mathfrak{m}(x).
\]

   The proof of any higher order differentiability of {}{$f_1$} can follow from similar arguments as above.
   
    On the other hand, the higher order derivatives of {}{$f_1$} shall be written as
\[
{}{f_1^{(m)}(t)}=\sum\limits_{k=0}^m\sum\limits_{i=0}^k\sum\limits_{j=0}^{{}{m-k}}\int_{B_R(x_0)}\int_{X}\int_{X}I_{k,i}I_{{}{m-k},j}\mathrm{d}\mathfrak{m}(z) \mathrm{d}\mathfrak{m}(y) \mathrm{d}\mathfrak{m}(x),
\]
where 
\[
I_{k,i}=\left\langle \nabla_x \frac{\partial^i}{\partial t^i}\rho(x,y,t),\nabla_x \frac{\partial^{k-i}}{\partial t^{k-i}}\rho(x,z,t)\right\rangle.
\]

{}{Letting \[
I_i=\left|\nabla_x\frac{\partial^i}{\partial t^i}\rho(x,y,t)\right|,\ \  J_{i}=\left|\nabla_x\frac{\partial^i}{\partial t^i}\rho(x,z,t)\right|,
\]
we obtain
\[
|I_{k,i}I_{m-k,j}|\leqslant I_i I_j J_{k-i} J_{m-k-j},\ \mathfrak{m}\text{-a.e.}
\]
}

Finally Theorem \ref{thm2.12}, Lemma \ref{aaaalem3.11} and (\ref{aaaaeqn3.8}) yield that 
\[
\left|\int_{X}I_i I_j \mathrm{d}\mathfrak{m}(y)\right|\leqslant C(K,N)\frac{i!j!}{t^{i+j+1}},
\]
\[
\left|\int_{X}J_{k-i} {}{J_{m-k-j}} \mathrm{d}\mathfrak{m}(z)\right|\leqslant C(K,N){}{\frac{(k-i)!(m-k-j)!}{t^{m-i-j+1}}.}
\]
Thus ${}{|f_1^{(m)}(t)|}\leqslant \mathfrak{m}(B_R(x_0))C(K,N){}{m!t^{-(m+2)}}$. This completes the proof.
\end{proof}

\subsection{A regularity result about IHKI RCD$(K,N)$ spaces}\label{sec3.2}


{}{This subsection is aimed at proving Theorem \ref{mainthm1.3}.} The following statement is trivial for the pmGH convergence of geodesic spaces, which is frequently used in the proof of  Theorem \ref{mainthm1.3}. We shall call no extra attention to this well-known fact in this paper.
\begin{fact}\label{11lem3.7}
Assume that $({X},\mathsf{d},\mathfrak{m})$ is an RCD$(K,N)$ space  {}{and is not a single point}. Then for any sequence of points $\{x_i\}\subset {X}$, and any $\{r_i\}$ with $r_i \rightarrow 0$, after passing to a subsequence, the pmGH limit of $\left\{\left({X}_{i},\dfrac{1}{r_{i}}\mathsf{d}_{i},\dfrac{\mathfrak{m}}{\mathfrak{m}(B_{r_{i}}(x_{i}))},x_i\right)\right\}$ is not a single point.
\end{fact}

{}{Let us fix an IHKI RCD$(K,N)$ space $({X},\mathsf{d},\mathfrak{m})$ which is not a single point. According to Proposition \ref{llem3.4}, we make a convention that there exists a function $c(t)$ such that 
\[
c(t)g_t= g,\ \forall t>0,
\] 
in the rest of this subsection.}

\begin{proof}[Proof of Theorem \ref{mainthm1.3}]
The proof consists of three steps.

\textbf{Step 1} There exists $\tilde{c}>0$, such that

\begin{equation}\label{11eqn3.3}
\lim\limits_{r\rightarrow 0}\frac{\mathfrak{m}(B_r(x))}{r^n}=\tilde{c},\ \  \forall x\in \mathcal{R}_n^\ast, 
\end{equation}
and the function $c$ satisfies
\begin{equation}\label{11eqn3.11}
\lim\limits_{t\rightarrow 0} \frac{t^{n+2}}{c(t^2)}=\tilde{c}^{-1}\omega_n c_1^{\mathbb{R}^n}.
\end{equation}

Fix $x\in\mathcal{R}_n^\ast$. From the very definition of $\mathcal{R}_n^\ast$, $\lim\limits_{r\rightarrow 0} r^{-n}\mathfrak{m}(B_{r}(x))=\tilde{c}$ for some $\tilde{c}=\tilde{c}(x)>0$. For any $\{r_i\}$ with $r_i \rightarrow 0$, we have
\begin{equation}\label{1pmGHconvergence}
({X}_i,\mathsf{d}_i,\mathfrak{m}_i,x):=\left({X},\frac{1}{r_i}\mathsf{d},\frac{\mathfrak{m}}{\mathfrak{m}(B_{r_i}(x))},x\right)\xrightarrow{\mathrm{pmGH}} \left(\mathbb{R}^n,\mathsf{d}_{\mathbb{R}^n},\frac{1}{\omega_n}\mathcal{L}^n,0_n\right).
\end{equation}

On each ${X}_i$, $c(r_i^2 t)g_t^{{X}_i}=r_i^2 \mathfrak{m}(B_{r_i}(x))g_{{X}_i}$. By \cite[Theorem 3.11]{BGHZ21}, $\{g_t^{{X}_i}\}$ $L^2$-strongly converges to $\omega_n g_t^{\mathbb{R}^n}$ on any $B_R(0_n)\subset \mathbb{R}^n$, from which we know
\[
\lim\limits_{i\rightarrow \infty}r_i^2 \frac{\mathfrak{m}(B_{r_i}(x))}{c(r_i^2 t)}=\omega_n c_t^{\mathbb{R}^n}.
\]
Since the above limit does not depend on the choice of the sequence $\{r_i\}$, we have
\begin{equation}\label{11eqn3.5}
\lim\limits_{r\rightarrow 0} r^2 \frac{\mathfrak{m}(B_{r}(x))}{c(r^2 t)}=\lim\limits_{r\rightarrow 0} \frac{ \mathfrak{m}(B_{r}(x))}{r^n} \frac{r^{n+2}}{c(r^2 t)}=\omega_n c_t^{\mathbb{R}^n}.
\end{equation}  

As a result, we get (\ref{11eqn3.11}). Observe that the limit in (\ref{11eqn3.5}) also does not depend on the choice of $x\in \mathcal{R}_n^\ast$, which suffices to show (\ref{11eqn3.3}).

\textbf{Step 2} $\mathfrak{m}=\tilde{c}\mathcal{H}^n$, for the constant $\tilde{c}$ obtained in Step 1.

Reprising the same arguments as in Step 1, we know that $\mathcal{R}_n=\mathcal{R}_n^\ast$ (In fact, $L^2$-strong convergence of $\{g_t^{{X}_i}\}$ on any $B_R(0_n)\subset \mathbb{R}^n$ is also valid when $x\in \mathcal{R}_n$ by \cite[Theorem 3.11]{BGHZ21}). This implies $\mathfrak{m}=\tilde{c}\mathcal{H}^n\llcorner\mathcal{R}_n$. To complete the proof of Step 2, we need nothing but $\mathcal{H}^n\ll\mathfrak{m}$.  {}{Because then a combination with Theorem \ref{1111thm2.22} gives $\mathcal{H}^n({X}\setminus \mathcal{R}_n)=0$, which is sufficient to conclude.}

For any $x\in {X}\setminus \mathcal{R}_n$, and any sequence $\{r_i\}$ with $r_i\rightarrow 0$, after passing to a subsequence, there exists a  pointed RCD$(0,N)$ space $({X}_\infty,\mathsf{d}_\infty,\mathfrak{m}_\infty,x_\infty)$ such that
\[
\left({X}_i,\mathsf{d}_i,\mathfrak{m}_i,x\right):=\left({X},\frac{1}{r_i}\mathsf{d},\frac{\mathfrak{m}}{\mathfrak{m}(B_{r_i}(x))},x\right)\xrightarrow{\mathrm{pmGH}} ({X}_\infty,\mathsf{d}_\infty,\mathfrak{m}_\infty,x_\infty).
\]

When $i$ is sufficiently large, again on each ${X}_i$, $c(r_i^2 t)g_t^{{X}_i}=r_i^2 \mathfrak{m}(B_{r_i}(x))g_{{X}_i}$. In particular, we know from Theorem \ref{thm2.18} that $r_i^2 \mathfrak{m}(B_{r_i}(x))\leqslant C(K,N)c(r_i^2 t)$. {}{Since $(X_\infty,\mathsf{d}_\infty)$ is not a single point}, using {}{Theorems \ref{thm2.26} and \ref{11thm2.26}}, and (\ref{11eqn3.11}), we see 
\[
\lim\limits_{i\rightarrow \infty} \frac{\mathfrak{m}(B_{r_i}(x))}{r_i^n}\in \left(0,C(K,N)\right).
\]

In particular, 
\begin{equation}\label{111eqn3.7}
C(K,N)\geqslant \limsup\limits_{r\rightarrow 0} \frac{\mathfrak{m}(B_{r}(x))}{r^n}\geqslant   \liminf\limits_{r\rightarrow 0} \frac{\mathfrak{m}(B_{r}(x))}{r^n}> 0.
\end{equation}

Set
\[
{X}_\tau:=\left\{x\in{X}:\liminf\limits_{r\rightarrow 0}\frac{\mathfrak{m}(B_r(x))}{r^n}\geqslant \tau\right\},
\]
and notice that ${X}=\bigcup_{\tau>0}{X}_\tau$ by (\ref{111eqn3.7}). Applying \cite[Theorem 2.4.3]{AT04} then implies
\[
\mathcal{H}^n\llcorner {X}_\tau \ll \mathfrak{m}\llcorner {X}_\tau,\ \forall \tau>0,
\]
 from which we conclude.

\textbf{Step 3} $({X},\mathsf{d},\mathcal{H}^{n})$ is an RCD$(K,n)$ space.

Without loss of generality, assume $\mathfrak{m}=\mathcal{H}^n$. We first treat the case that $({X},\mathsf{d},\mathcal{H}^{n})$ is compact. By Theorem \ref{eqnBGHZ21}, it suffices to show
\begin{equation}\label{eqn20220203}
\inf\limits_{x\in{X}} \inf\limits_{s\in (0,1)} \frac{\mathcal{H}^n(B_s(x))}{s^n}>0.
\end{equation}

Assume on the contrary that  (\ref{eqn20220203}) does not hold, then for any $\epsilon>0$, there exists $x_\epsilon \in {X}$, such that $ \inf\limits_{s\in (0,1)} s^{-n}\mathcal{H}^n(B_s(x_\epsilon))<\epsilon$. By (\ref{BGinequality}),  
\[
\frac{\mathcal{H}^n(B_{r}(x_\epsilon))}{r^n}<\epsilon, \ \ \text{for some}\  r=r(\epsilon)\leqslant \Psi\left(\epsilon|K,N,\mathrm{diam}({X},\mathsf{d}),\mathcal{H}^n({X})\right).
\]

As a consequence, there {}{exists} a sequence $\{x_i\}\subset {X}$, a sequence $\{r_i\}\subset (0,\infty)$ with $r_i\rightarrow 0$ and a pointed RCD {}{$(0,N)$} space $({X}_\infty,\mathsf{d}_\infty,\mathfrak{m}_\infty,x_\infty)$, such that 
\begin{equation}\label{111eqn3.9}
\lim\limits_{i\rightarrow \infty}\frac{\mathcal{H}^n(B_{r_i}(x_i))}{r_i^n}=0,
\end{equation}
and 

\[
({X}_i,\mathsf{d}_i,\mathfrak{m}_i,x_i):=\left({X}_i,\frac{1}{r_i}\mathsf{d},\frac{\mathfrak{m}}{\mathfrak{m}\left(B_{r_i}(x_i)\right)} ,x_i\right)\xrightarrow{\mathrm{pmGH}} ({X}_\infty,\mathsf{d}_\infty,\mathfrak{m}_\infty,x_\infty).
\]

Again $c(r_i^2 t)g_t^{{X}_i}=r_i^2 \mathfrak{m}\left(B_{r_i}(x_i)\right) g_{{X}_i}$ on each ${X}_i$, and $\left\{g_t^{{X}_i}\right\}$ $L^2$-strongly converges to 0 on {}{$B_R({x_\infty})$ for any $R>0$} by (\ref{111eqn3.9}), which contradicts Proposition \ref{llem3.4}.

As for the non-compact case, it suffices to repeat Step 1-3 and apply Theorem \ref{eqnBGHZ21} again on any $B_R(x)\subset {X}$.

\end{proof}


\subsection{Non-compact IHKI RCD$(0,n)$ spaces }\label{sec3.3}

We start by proving the following theorem in this subsection.

\begin{thm}\label{thm4.5}
Suppose $({X},\mathsf{d},\mathcal{H}^{n-1})$ is a non-collapsed $\mathrm{RCD}(n-2,n-1)$ space with $n\geqslant 2$. If $ g_1^{\text{C}({X})}\geqslant c g_{\text{C}({X})}$ for some $c>0$, then $({X},\mathsf{d})$ is isometric to $(\mathbb{S}^{n-1},\mathsf{d}_{S^{n-1}})$.

\end{thm}

We need some preparations. According to Remark \ref{rmk2.10}, $\left(\text{C}({X}),\mathsf{d}_{\text{C}({X})},\mathfrak{m}_{\text{C}({X})}\right)$ is an RCD$(0,n)$ space. In addition, by applying Theorem \ref{1111thm2.20}, Theorem \ref{BGHZmainthm} and the splitting theorem for RCD$(0,n)$ spaces (see \cite[Theorem 1.4]{G13}, \cite{G14}), $\left(\text{C}({X}),\mathsf{d}_{\text{C}({X})},\mathfrak{m}_{\text{C}({X})}\right)$ is also non-collapsed, which means that $\mathfrak{m}_{\text{C}({X})}=\mathcal{H}^{n}$.

To fix the {}{notation}, we use (\ref{notation2.7}), and set $\alpha=(2-n)/2$, $\nu_j=\sqrt{\alpha^2+\mu_j}$ for {}{every} $j\in \mathbb{N}$.  It is notable that $\mu_1\geqslant n$ by \cite[Corollary 1.3]{K15b}. {}{For any RCD$(K,N)$ space $(Y,\mathsf{d}_Y,\mathfrak{m}_Y)$, we define
\[
\begin{aligned}
\rho_t^Y:Y&\longrightarrow (0,\infty)\\
                y&\longmapsto \rho^Y(y,y,t).
\end{aligned}
\]
}

The validity of limit processes in the proof of Theorem \ref{thm4.5} can be verified by the following estimates. We check one of them for reader's convenience.
\begin{lem}\label{20211220b}
There exists $C=C(n,\mathrm{diam}({X},\mathsf{d}))$, such that the following estimates hold.
\begin{enumerate}

\item\label{lem3.192}
  $\ \sup\limits_{x\in{X}}\sum\limits_{j=k}^\infty I_{\nu_j}(r)\phi^2_j(x) \leqslant C\left(\dfrac{r}{2}\right)^{k^{\frac{1}{2(n-1)}}}, \ \forall r
\in (0,1),\ \forall k\in \mathbb{N}_+.
$

\item

$\ 
I_{\nu_j}(r)\mu_j \leqslant Cj^2 \left(\dfrac{r}{2}\right)^{\nu_j}\leqslant Cj^2 \left(\dfrac{r}{2}\right)^{j^{\frac{1}{n-1}}}, \ \forall r
\in (0,1),\  \forall j\in \mathbb{N}.$
\item$\ \sum\limits_{j=k}^\infty I_{\nu_j}(r)\mu_j \leqslant C\left(\dfrac{r}{2}\right)^{k^{\frac{1}{2(n-1)}}},\ \forall r
\in (0,1),\  \forall k\in \mathbb{N}_+.$

\end{enumerate}
\end{lem}

\begin{proof}[Proof of \ref{lem3.192}.] According to Proposition \ref{heatkernel2}, there exists $C=C(n,\mathrm{diam}({X},\mathsf{d}))$, such that for any $x\in {X}$,
\[
\begin{aligned}
\sum\limits_{j=k}^\infty I_{\nu_j}(r)\phi^2_j(x)&\leqslant C\sum\limits_{j=k}^\infty I_{\nu_j}(r)j^{n-1}\\
\ &=C \sum\limits_{j=k}^\infty j^{n-1}\sum\limits_{l=0}^\infty \frac{1}{l! \Gamma(\nu_j+l+1)}\left(\frac{r}{2}\right)^{2l+\nu_j}\\
\ &\leqslant  C \sum\limits_{j=k}^\infty j^{n-1}  \left(\frac{r}{2}\right)^{\nu_j}\exp\left(\frac{r^2}{4}\right)\\
\ &\leqslant  C \sum\limits_{j=k}^\infty  j^{n-1} \left(\frac{r}{2}\right)^{j^{\frac{1}{n-1}}}\\
\ &\leqslant  C \left(\frac{r}{2}\right)^{k^{\frac{1}{2(n-1)}}}\sum\limits_{j=k}^\infty  j^{n-1} \left(\frac{r}{2}\right)^{j^{\frac{1}{2(n-1)}}}\leqslant C\left(\frac{r}{2}\right)^{k^{\frac{1}{2(n-1)}}}. \\
\end{aligned}
\]

\end{proof}

Notice that $(\text{C}({X}),\mathsf{d}_{\text{C}({X})},\mathcal{H}^n)$ has maximal volume growth, and its blow down is itself. Applying the large time behavior of the heat kernel \cite[Theorem 1.3]{JLZ16} shows

\begin{equation}\label{1prop4.3}
\rho^{\text{C}({X})}_{t}\equiv \frac{n\omega_n}{\mathcal{H}^{n-1}({X})} (4\pi t)^{-\frac{n}{2}},\ \ \forall t>0.
\end{equation}

Lemma \ref{llem3.1} and Lemma \ref{1lem3.15} are also useful in the proof of Theorem \ref{thm4.5}. 

\begin{lem}\label{llem3.1}
Let $({Y_i},\mathsf{d}_{i},\mathfrak{m}_{i})$ be two $\mathrm{RCD}(K,N)$ spaces such that $\rho^{Y_i}_{2t}$ are constant functions for some $t>0$ $(i=1,2)$. Then on $Y_1\times Y_2$,
\[
 g_t^{Y_1\times Y_2 } (y_1,y_2)=\rho^{Y_1}_{2t}(y_1)g_t^{Y_2}(y_2) + \rho^{Y_2}_{2t}(y_2)g_t^{Y_1}(y_1).
\]
That is, for any $f\in \mathrm{Lip}_c\left(Y_1\times Y_2,\mathsf{d}_{Y_1\times Y_2}\right)$, denote by $f^{(y_1)}:y_2\mapsto f(y_1,y_2)$ for any fixed $y_1$, and $f^{(y_2)}:y_1\mapsto f(y_1,y_2)$  for any fixed $y_2$, it holds that  
\[
\begin{aligned}
 \ &g_t^{Y_1\times Y_2 }\left(\nabla^{Y_1\times Y_2 } f, \nabla^{Y_1\times Y_2 } f\right)(y_1,y_2)\\
=\ &\rho^{Y_1}_{2t}(y_1)g_t^{Y_2}\left(\nabla^{Y_2} f^{(y_1)},\nabla^{Y_2} f^{(y_1)}\right)(y_2) + \rho^{{Y_2}}_{2t}(y_2)g_t^{Y_1}\left(\nabla^{Y_1} f^{(y_2)},\nabla^{Y_1} f^{(y_2)}\right)(y_1),
\end{aligned}
\]
for $\mathfrak{m}_{Y_1\times Y_2}$-a.e. $(y_1,y_2)$ in $Y_1\times Y_2$.
\end{lem}
\begin{proof}
Recalling (\ref{1234eqn2.9}),(\ref{eqn2.1}) and the definition of $g_t^{Y_1\times Y_2 }$ in Theorem \ref{thm2.18}, we have
{}{\[
\begin{aligned}
\ &\ g_t^{Y_1\times Y_2 }(y_1,y_2)\\
=\ &\int_{Y_1\times Y_2} \sum\limits_{i=0}^1 \rho^{Y_{i+1}}(y_{i+1},y_{i+1}',t)d_{y_{2-i}}\rho^{Y_{2-i}}(y_{2-i},y_{2-i}',t)\\ 
\ &\otimes \sum\limits_{i=0}^1 \rho^{Y_{i+1}}(y_{i+1},y_{i+1}',t)d_{y_{2-i}}\rho^{Y_{2-i}}(y_{2-i},y_{2-i}',t)\mathrm{d}\mathfrak{m}_1(y_1')\mathrm{d}\mathfrak{m}_2(y_2')\\
=\ &\rho^{Y_1}_{2t}(y_1)g_t^{Y_2}(y_2) + \rho^{Y_2}_{2t}(y_2)g_t^{Y_1}(y_1)+I_1(y_1,y_2)+I_2(y_1,y_2),
\end{aligned}
\]}
where 
\[
I_1(y_1,y_2)=\frac{1}{4}\int_{Y_1\times {Y}_2} d_{y_1}\left(\rho^{Y_1}(y_1,y_1',t)\right)^2\otimes d_{y_2}\left(\rho^{{Y}_2}(y_2,y'_2,t)\right)^2\mathrm{d}\mathfrak{m}_1(y_1')\mathrm{d}\mathfrak{m}_2(y_2'),
\]
\[
I_2(y_1,y_2)=\frac{1}{4}\int_{Y_1\times {Y}_2}d_{y_2}\left(\rho^{{Y}_2}(y_2,y'_2,t)\right)^2\otimes d_{y_1}\left(\rho^{Y_1}(y_1,y_1',t)\right)^2\mathrm{d}\mathfrak{m}_1(y_1')\mathrm{d}\mathfrak{m}_2(y_2'),
\]

By our assumption, for $i=1,2$, we have
{}{\[
\left(y_i\mapsto d_{y_i} \int_{Y_i} \left(\rho^{Y_i}(y_i,y_i',t)\right)^2 \mathrm{d}\mathfrak{m}_i(y_i')\right)=0\ \  \text{in}\  L^2(T^\ast (Y_i,\mathsf{d}_i,\mathfrak{m}_i)).
\] 
}

Therefore $I_1(y_1,y_2)=0$ and $I_2(y_1,y_2)=0$ follow from the local Hille's theorem (see {}{for example} \cite[Proposition 3.4]{BGHZ21}). 

\end{proof}

\begin{lem}\label{1lem3.15}
Under the assumption of Lemma \ref{llem3.1}, if moreover there exist $c_1,c_2,{}{t}>0$, such that $g_t^{Y_1}= c_1 g_{Y_1}$ and
\[
g_t^{{Y_1}\times {Y_2}}\geqslant c_2 g_{Y_1\times {Y}_2} \ (\text{resp. }g_t^{Y_1\times {Y}_2}= c_2 g_{{Y}_1\times {Y}_2}),
\]
then there exists $c_3>0$, such that 
\[
g_t^{Y_2}\geqslant c_3 g_{Y_2}\ (\text{resp. } g_t^{Y_2}=c_3 g_{Y_2}){}{.}
\]
\end{lem}
\begin{proof}
Since the {}{proofs} of both cases are almost the {}{same}, we only give the proof of the case that $g_t^{{Y_1}\times {Y_2}}\geqslant c_2 g_{Y_1\times {Y}_2} $. 

Fix a ball $B_R^{Y_1}(\tilde{y}_1)\subset Y_1$, by \cite[Lemma 3.1]{MN19}, there exists a cut-off function $\phi\in \mathrm{Lip}_c(Y_1,\mathsf{d}_1)$ such that 
\[
\phi|_{B_R^{Y_1}(\tilde{y}_1)}\equiv 1, \  \phi|_{Y_1\setminus B_{2R}^{Y_1}(\tilde{y}_1)}\equiv 0.
\]

Now for any $\varphi \in H^{1,2}(Y_2,\mathsf{d}_2,\mathfrak{m}_2)$, 
 set $f:(y_1,y_2)\mapsto \phi(y_1)\varphi(y_2)$. Then it follows from (\ref{2.27}) and Lemma \ref{llem3.1} that for $\mathfrak{m}_{Y_1\times Y_2}$-a.e. $(x,y)$ in $ B_R^{Y_1}(\tilde{y}_1)\times {Y_2}$,  
 
 \[
\begin{aligned}
\ &\rho^{Y_1}_{2t}(y_1)g_t^{Y_2} \left(\nabla^{Y_2} \varphi,\nabla^{Y_2} \varphi\right)(y_2)\\
=\ &\phi^2(y_1)\rho^{Y_1}_{2t}(y_1)g_t^{Y_2} \left(\nabla^{Y_2} \varphi,\nabla^{Y_2} \varphi\right)(y_2)+c_1 \varphi^2(y_2)\rho^{{Y}_2}_{2t}(y_2)  \left|\nabla \phi\right|^2(y_1)\\  
=\ &\rho^{Y_1}_{2t}(y_1)g_t^{{Y}_2}\left(\nabla^{Y_2} f^{(y_1)},\nabla^{Y_2} f^{(y_1)}\right)(y_2)+\rho^{{Y}_2}_{2t}(y_2)g_t^{Y_1}\left(\nabla^{Y_1} f^{(y_2)},\nabla^{Y_1} f^{(y_2)}\right)(y_1)\\  
=\ &g_t^{Y_1\times {Y}_2 }\left(\nabla^{Y_1\times {Y}_2 } f, \nabla^{Y_1\times {Y}_2 } f\right)(y_1,y_2)\\
\geqslant\ &c_2 g_{Y_1\times {Y}_2} \left(\nabla^{Y_1\times {Y}_2 } f, \nabla^{Y_1\times {Y}_2 } f\right)(y_1,y_2)=c_2|\nabla^{Y_2}\varphi|^2(y_2).
\end{aligned}
\]

In particular, 
\[
\rho^{Y_1}_{2t}(y_1)g_t^{Y_2} \left(\nabla^{Y_2} \varphi,\nabla^{Y_2} \varphi\right)(y_2)\geqslant c_2|\nabla^{Y_2}\varphi|^2(y_2), \ \ \mathfrak{m}_2\text{-a.e.}\ y_2\in{Y_2}.
\]

Since $\varphi \in H^{1,2}(Y_2,\mathsf{d}_2,\mathfrak{m}_2)$ is taken to be arbitrary,  we complete the proof by setting $c_3:=c_2 \left(\rho_{2t}^{Y_1}\right)^{-1}$.
\end{proof}

\begin{proof}[Proof of Theorem \ref{thm4.5}]
We start by considering the case that $n\geqslant 4$.

For any fixed $(r_0,x_0)\in \text{C}({X})$ and any $\varphi \in \text{Lip}({X},\mathsf{d})$, take $f\in C^\infty((0,\infty))$ such that $\text{supp}f\in (r_0/4,3r_0)$ and $f\equiv 1$ on $(r_0/2,2r_0)$. Then {}{Proposition \ref{1prop2.23} and (\ref{neiji1}) yield} that for $\mathcal{H}^n$-a.e. $(r,x)\in B_{r_0/2}^{\text{C}({X})}\left(r_0,x_0\right)$, 

 \begin{equation}\label{111eqn3.21}
\begin{aligned}
cr^{-2} \left| \nabla \varphi\right|^2(x)&=c \left| \nabla (f\varphi)\right|^2_{\text{C}({X})}(r,x)\\
\ &\leqslant  g_1^{\text{C}({X})}\left(\nabla (f\varphi),\nabla (f\varphi) \right)(r,x)\\
\ &=\frac{1}{4} r^{2\alpha}\sum\limits_{j=1}^\infty\int_0^\infty  s\exp\left(-\frac{r^2+s^2}{2}\right)I_{\nu_j}\left(\frac{rs}{2}\right)^2 \mathrm{d}s\left\langle \nabla(f\varphi), \nabla\phi_j \right\rangle_{\text{C}({X})}^2(r,x)  \\
\ &=\frac{1}{4} r^{2\alpha-4}\sum\limits_{j=1}^\infty\int_0^\infty  s\exp\left(-\frac{r^2+s^2}{2}\right)I_{\nu_j}\left(\frac{rs}{2}\right)^2 \mathrm{d}s\left\langle \nabla \varphi, \nabla\phi_j \right\rangle^2(x)  \\
\ &=\frac{1}{2} r^{2\alpha-4}\sum\limits_{j=1}^\infty  \exp\left(-\frac{r^2}{2}\right)I_{\nu_j}\left(\frac{r^2}{2}\right) \left\langle \nabla \varphi, \nabla\phi_j \right\rangle^2(x),  \\
\end{aligned}
\end{equation}
where the last equality follows from the semigroup property of $\{h^{\text{C}({X})}_t\}_{t>0}$.

In the remaining part of the proof, we just  denote by $|\cdot|$ the pointwise norm on $L^2(T^\ast ({X},\mathsf{d},\mathcal{H}^{n-1}))$ for notation convenience. 

{}{Combining the fact that $\left|\langle \nabla \varphi, \nabla\phi_j \rangle\right|\leqslant |\nabla \varphi||\nabla \phi_j|$, $\mathcal{H}^{n-1}$-a.e. in ${X}$, with last equality of (\ref{111eqn3.21})} implies 
\[
c \left| \nabla \varphi\right|^2 \leqslant \frac{1}{2} r^{-n}\sum\limits_{j=1}^\infty  \exp\left(-\frac{r^2}{2}\right)I_{\nu_j}\left(\frac{r^2}{2}\right) \left|\nabla \varphi\right|^2  \left|\nabla \phi_j\right|^2\ \ \mathcal{H}^n\text{-a.e. } (r,x)\in B_{r_0/2}^{\text{C}({X})}(r_0,x_0).
\]

In particular, taking $\varphi=\mathsf{d}(x_0,\cdot)$ which satisfies that $|\nabla \varphi|\equiv 1$, we have
\begin{equation}\label{3.9}
c \leqslant  \frac{1}{2}r^{-n}\exp\left(-\frac{r^2}{2}\right) \sum\limits_{j=1}^\infty I_{\nu_j}\left(\frac{r^2}{2}\right) | \nabla\phi_j |^2\  \ \mathcal{H}^n\text{-a.e. } (r,x)\in B_{r_0/2}^{\text{C}({X})}(r_0,x_0).
\end{equation}
 Integration of (\ref{3.9}) on ${X}$ then gives
\begin{equation}\label{3.10}
c \mathcal{H}^{n-1}({X}) \leqslant  \frac{1}{2}r^{-n}\exp\left(-\frac{r^2}{2}\right) \sum\limits_{j=1}^\infty I_{\nu_j}\left(\frac{r^2}{2}\right) \mu_j\ \ \mathcal{L}^1\text{-a.e.}\ r\in(r_0/2,2r_0).
\end{equation}
In fact, (\ref{3.10}) holds for any $r>0$ due to the arbitrarity of $r_0>0$, which is still denoted as (\ref{3.10}).

{}{If $n\geqslant 4$ and $\mu_1>n-1$, then $\nu_j\geqslant \nu_1>n/2$, for all $ j\in \mathbb{N}_+$}. However, Lemma \ref{20211220b} implies that the right hand side of (\ref{3.10}) vanishes as $r\rightarrow 0$. Thus a contradiction occurs. Therefore $\mu_1=n-1$ {}{when $n\geqslant 4$}. 

By Theorem \ref{BGHZmainthm} and Obata's first eigenvalue rigidity theorem \cite[Theorem 1.2]{K15b}, there exists a non-collapsed RCD$(n-3,n-2)$ space $({X}',\mathsf{d}_{{X}'},\mathcal{H}^{n-2})$, such that {}{$\left(\text{C}({X}),\mathsf{d}_{\text{C}({X})}\right)$ is isometric to $\left(\mathbb{R}\times \text{C}({X}'),\sqrt{\mathsf{d}_{\mathbb{R}}^2+\mathsf{d}_{\text{C}({X}')}^2}\right)$.}

From (\ref{eqn2.1}) and (\ref{1prop4.3}), we know
\[
\rho^{\text{C}(X)}_{t}\equiv \frac{n\omega_{n}}{\mathcal{H}^{n-1}({X})} (4\pi t)^{\frac{n-1}{2}}.
\]

Using Lemmas \ref{llem3.1} and \ref{1lem3.15}, we see that  $g_1^{\mathrm{C}({X}')}\geqslant c' g_{\mathrm{C}({X}')}$ for some $ c'>0$. It is now sufficient to deal with the case that $n=3$.  

Repeating the previous arguments, we have $\mu_1=2$. We claim that $\mu_2=2$. If $\mu_2>2$, then the integration of (\ref{3.9}) on any measurable set $\Omega \subset {X}$ yields

\[
\begin{aligned}
c \mathcal{H}^2(\Omega)\leqslant &\ Cr^{-2} \sum\limits_{j=1}^\infty I_{\nu_j}\left(\frac{r^2}{2t}\right)\int_\Omega \left| \nabla\phi_j \right|^2 \mathrm{d}\mathcal{H}^2\\
\leqslant &\ Cr^{-2} I_{\nu_1}\left(\frac{r^2}{2t}\right)\int_\Omega \left| \nabla\phi_1\right|^2 \mathrm{d}\mathcal{H}^2+r^{-2}\sum\limits_{j=2}^\infty I_{\nu_j}\left(\frac{r^2}{2t}\right)\int_{{X}} \left| \nabla\phi_j \right|^2 \mathrm{d}\mathcal{H}^2\\
\rightarrow &\  C \int_\Omega \left| \nabla\phi_1\right|^2\mathrm{d}\mathcal{H}^2 \ \ \text{as }r\rightarrow 0. 
\end{aligned}
\]  
for some $C=C(n,\mathrm{diam}({X},\mathsf{d}))$. The arbitrarity of $\Omega$, together with the Lebesgue differentiation theorem shows that $|\nabla \phi_1|^2 \geqslant c_0:=c^{-1}C>0$,  $\mathcal{H}^2$-a.e. 

Consider the Laplacian of $\phi_1^\alpha$ for any even integer $\alpha$, and calculate as follows:
\[
\begin{aligned}
\Delta \phi_1^\alpha &=\alpha (\alpha-1)|\nabla \phi_1|^2 \phi_1^{\alpha-2}+\alpha \phi_1^{\alpha-1}\Delta \phi_1\\
\ &=\alpha (\alpha-1)|\nabla \phi_1|^2 \phi_1^{\alpha-2}-\alpha \phi_1^{\alpha-1}(n-1) \phi_1\\
\ &=\alpha \phi_1^{\alpha-2}\left((\alpha-1)|\nabla \phi_1|^2 - (n-1)\phi_1^2 \right)\\
\ &\geqslant \alpha \phi_1^{\alpha-2}
\left((\alpha-1)c_0 -C(n,\mathrm{diam}({X},\mathsf{d}))\right), \ \ \mathcal{H}^{2}\text{-a.e.}
\end{aligned}
\]

As a result, the integer $\alpha$ can be chosen to be sufficiently large such that $\phi_1^\alpha$ is superharmonic.  However, any superharmonic function on a compact RCD {}{space} must be a constant function (see like \cite[Theorem 2.3]{GR19}). A contradiction. Therefore $\mu_2=2$. 	

According to \cite[Theorem 1.4]{K15b}, $({X},\mathsf{d})$ must be isometric to either $(\mathbb{S}^2,\mathsf{d}_{\mathbb{S}^2})$ or $\left(\mathbb{S}^2_+,\mathsf{d}_{\mathbb{S}^2_+}\right)$. Thus $\left(\text{C}({X}),\mathsf{d}_{\text{C}({X})}\right)$ must be isometric to either $(\mathbb{R}^3,\mathsf{d}_{\mathbb{R}^3})$ or $\left(\mathbb{R}^3_+,\mathsf{d}_{\mathbb{R}^3_+}\right)$.

 Notice that on $\mathbb{R}^n_+:=\{(x_1,\cdots,x_n)\in \mathbb{R}^n:x_n>0\}$, 
 \[
g_t^{\mathbb{R}^n_+}\left(\frac{\partial}{\partial x_n},\frac{\partial}{\partial x_n}\right)(x_1,\cdots,x_n)=c_n t^{-\frac{n+2}{2}}\left(\frac{1-\exp\left(-\frac{x_n^2}{2t}\right)}{2}+\frac{x_n^2}{4t}\exp(-\frac{x_n^2}{2t})\right).
\]

It is clear that
\[
\lim\limits_{x_3\rightarrow 0^+} g_t^{\mathbb{R}^3_+}\left(\frac{\partial }{\partial x_3},\frac{\partial }{\partial x_3}\right)(x_1,x_2,x_3)=0, 
\] 
which contradicts our assumption.

When $n=2$, set ${Y}=\text{C}({X})\times \mathbb{R}$, and notice that $g_1^{{Y}}\geqslant c' g_{Y}$ for some $c'>0$ by (\ref{1prop4.3}), Lemma \ref{llem3.1} and Lemma \ref{1lem3.15}, which shall be verified in the same way as previous arguments. Thus $({Y},\mathsf{d}_{Y})$ must be isometric to $\left(\mathbb{R}^3,\mathsf{d}_{\mathbb{R}^3}\right)$ and $\left(\text{C}({X}),\mathsf{d}_{\mathrm{C}({X})}\right)$ must be isometric to $(\mathbb{R}^2,\mathsf{d}_{\mathbb{R}^2})$.

\end{proof}

As an application of Theorem \ref{thm4.5}, we prove Theorem \ref{mainthm1.5}.

\begin{proof}[Proof of Theorem \ref{mainthm1.5}]
It follows from Theorem \ref{mainthm1.3} that $\mathfrak{m}=c\mathcal{H}^n$ for some $c>0$, and $({X},\mathsf{d},\mathcal{H}^n)$ is an RCD$(0,n)$ space. Without loss of generality, we may assume that $\mathfrak{m}=\mathcal{H}^n$.

The subsequent pmGH arguments in this proof are almost the same as that in the proof of Theorem \ref{mainthm1.3}, and we omit the details.

Take $\{r_i\}$ with $r_i\rightarrow \infty$, and a pointed RCD$(0,n)$ space $({X}_\infty,\mathsf{d}_{\infty},\mathfrak{m}_\infty,x_\infty)$ such that
\[
({X}_i,\mathsf{d}_i,\mathfrak{m}_i,x):=\left({X},\frac{1}{r_i}\mathsf{d},\frac{\mathfrak{m}}{\mathfrak{m}\left(B_{r_i}(x)\right)},x\right)\xrightarrow{\mathrm{pmGH}} ({X}_\infty,\mathsf{d}_\infty,\mathfrak{m}_\infty,x_\infty).
\]

Again on each ${X}_i$, $c(r_i^2 t)g_t^{{X}_i}=r_i^2 \mathfrak{m}\left(B_{r_i}(x_i)\right)g_{{X}_i}$. Applying  (\ref{tsukaeqn3.3}) and Proposition \ref{llem3.4} implies that
\begin{equation}\label{1eqn4.11}
\liminf\limits_{i\rightarrow \infty} \frac{\mathfrak{m}(B_{r_i}(x))}{r_i^n}=a>0, 
\end{equation}
and 
\begin{equation}\label{1eqn4.12}
\lim\limits_{i\rightarrow \infty} r_i^{-(n+2)}c(r_i^2)=b>0.
\end{equation}

By Theorem \ref{11thm2.15}, there {}{exists} a subsequence of $\{r_i\}$ which is still denoted as $\{r_i\}$ and a pointed RCD$(0,n)$ space $({Y}_\infty,\mathsf{d}_\infty',\mathcal{H}^n,y_\infty)$ such that 
\[
({Y}_i,\mathsf{d}_i',\mathcal{H}^n,y):=\left({X},\frac{1}{r_i}\mathsf{d},\frac{1}{r_i^n}\mathfrak{m},x\right)\xrightarrow{\mathrm{pmGH}} ({Y}_\infty,\mathsf{d}_\infty',\mathcal{H}^n,y_\infty).
\]

As a result, combining \cite[Theorem 1.1]{DG16} with (\ref{1eqn4.11}) and (\ref{1eqn4.12}) yields that $({Y}_\infty,\mathsf{d}_\infty',\mathcal{H}^n,y_\infty)$ is an Euclidean cone with $g_1^{{Y}_\infty}= b g^{{Y}_\infty}$ .

Therefore Theorem \ref{thm4.5} implies that $({Y}_\infty, \mathsf{d}_\infty')$ must be isometric to $(\mathbb{R}^n,\mathsf{d}_{\mathbb{R}^n})$. Finally, it remains to use the volume rigidity theorem for non-collapsed almost RCD$(0,n)$ spaces \cite[Theorem 1.6]{DG18} to conclude.
\end{proof}

The following corollary can be proved by using similar arguments as in the proof of Theorem \ref{mainthm1.5}.

\begin{cor}\label{cor4.7}
Let $({X},\mathsf{d},\mathcal{H}^n)$ be a non-collapsed $\mathrm{RCD}(0,n)$ space. {}{If there exists a function $c(t)$ such that 
\begin{enumerate}
\item $c(t)g_t \geqslant g$, $\forall t>0$, 
\item $\liminf\limits_{t\rightarrow \infty} t^{-(n+2)}c(t^2)>0$.
\end{enumerate}}
Then  $({X},\mathsf{d})$ is isometric to $\left(\mathbb{R}^n,\mathsf{d}_{\mathbb{R}^n}\right)$.
\end{cor}

\section{The isometric immersion into Euclidean space}\label{sec4}
The main purpose of this section is to prove Theorem \ref{thm1.5}. To begin with, let us recall a useful result (Theorem \ref{111thm4.3}) in \cite{H21}, which plays a important role in this section.
\begin{defn}[Regular map]
Let $({X},\mathsf{d},\mathfrak{m})$ be an RCD$(K,N)$ space. Then a map $F:=(\varphi_1,\ldots,\varphi_k):{X}\rightarrow \mathbb{R}^k$ is said to be regular if each $\varphi_i$ is in $D(\Delta)$ with $\Delta \varphi_i\in L^\infty(\mathfrak{m})$.
\end{defn}

\begin{defn}[Locally uniformly $\delta$-isometric immersion]
Let $({X},\mathsf{d},\mathfrak{m})$ be an RCD$(K,N)$ space and $F:=(\varphi_1,\ldots,\varphi_k):{X}\rightarrow \mathbb{R}^k$ be a locally Lipschitz map. $F$ is said to be a locally uniformly $\delta$-isometric immersion on $B_r(x_0)\subset {X}$ if for any $x\in B_r(x_0)$ it holds that
\[
\frac{1}{\mathfrak{m}(B_s(x))}\int_{B_{\delta^{-1}s}(x)}|F^\ast g_{\mathbb{R}^k}-g_{X}|\mathrm{d}\mathfrak{m}<\delta,\ \forall s\in (0,r).
\]
\end{defn}

\begin{thm}[{\cite[Theorem 3.4]{H21}}]\label{111thm4.3}
Let $({X},\mathsf{d},\mathfrak{m})$ be an $\mathrm{RCD}(K,N)$ space with $\mathrm{dim}_{\mathsf{d},\mathfrak{m}}({X})=n$ and let $F:=(\varphi_1,\ldots,\varphi_k):{X}\rightarrow \mathbb{R}^k$ be a regular map with
\[
\sum\limits_{i=1}^k \| |\nabla \varphi_i|\|_{L^\infty(\mathfrak{m})}\leqslant C.
\]
If $F$ is a locally uniformly $\delta$-isometric immersion on some ball $B_{4r}(x_0)\subset {X}$. Then the following {}{statements hold.}
\begin{enumerate}
\item For any $s\in (0,r)$, $\mathsf{d}_{\mathrm{GH}}(B_s(x_0),B_s(0_n))\leqslant \Psi(\delta|K,N,k,C)s$, where $\mathsf{d}_{\mathrm{GH}}$ is the Gromov-Hausdorff distance.
\item $F|_{B_{r}(x_0)}$ is $(1+\Psi(\delta|K,N,k,C))$-bi-Lipschitz from $B_{r}(x_0)$ to $F(B_{r}(x_0))\subset \mathbb{R}^k$.
\end{enumerate}
\end{thm}

{}{
From now on, we let $({X},\mathsf{d},\mathcal{H}^n)$ be a fixed compact non-collapsed $\mathrm{RCD}(K,n)$ space, and we assume that
\begin{equation}\label{111111eqn1.5}
g=\sum\limits_{i=1}^m d\phi_i\otimes d\phi_i,
\end{equation}
where $g$ is the canonical Riemannian metric of $(X,\mathsf{d},\mathcal{H}^n)$ and each $\phi_i$ is an eigenfunction of $-\Delta$ with corresponding eigenvalue $\mu_i$ ($i=1,\ldots,m$). To fix the notation}, denote by $C$ a constant with 
{}{\[
C=C\left(K,m,n,\mathrm{diam}({X},\mathsf{d}),\mathcal{H}^n({X}),\mu_1,\ldots,\mu_m,\left\|\phi_1\right\|_{L^2(\mathcal{H}^n)},\ldots,\left\|\phi_m\right\|_{L^2(\mathcal{H}^n)}\right),\] }
which may vary from line to line, and by $\mathsf{M}_{n\times n}(\mathbb{R})$ the set of all $n\times n$ real matrices equipped with the Euclidean metric on $\mathbb{R}^{n^2}$, and by $I_n$ the $n\times n$ identity matrix.

\begin{lem}\label{1lem4.2}
Each $\langle \nabla \phi_i,\nabla \phi_j \rangle $ is a Lipschitz function $(i,j=1,\ldots,m)$. In particular, 
\begin{equation}
\sum\limits_{i,j=1}^m  \left\| |\nabla \left\langle \nabla \phi_i,\nabla \phi_j \right\rangle| \right\|_{L^\infty(\mathcal{H}^n)}\leqslant  C.
\end{equation}
\end{lem}
\begin{proof}
We first show that $|\nabla \phi_1|^2\in \text{Lip}({X},\mathsf{d})$. Taking trace of {}{(\ref{111111eqn1.5})} gives
\begin{equation}\label{1eqn4.1}
\sum\limits_{i=1}^m \left|\nabla \phi_i\right|^2 =\langle g,g\rangle =n.
\end{equation}

Using the Bochner's inequality (\ref{bochnerineq}), for any $\varphi\in \mathrm{Test}F_+({X},\mathsf{d},\mathcal{H}^n)$, we get

\begin{equation}\label{1eqn4.2}
\int_{X} \left|\nabla \phi_1\right|^2 \Delta \varphi \mathrm{d}\mathcal{H}^n
  \geqslant 2\int_{X} \varphi \left( (K-\mu_1) \left|\nabla \phi_1\right|^2   + \frac{1}{n}\mu_1^2\phi_1^2  \right) \mathrm{d}\mathcal{H}^n
\geqslant  -C\int_{X} \varphi \mathrm{d}\mathcal{H}^n,
 \end{equation}
where the last inequality comes from Proposition \ref{heatkernel2}. Owing to (\ref{1eqn4.1}) and (\ref{1eqn4.2}), 
\begin{equation}\label{111eqn4.4}
 \int_{X} \left|\nabla \phi_1\right|^2 \Delta \varphi \mathrm{d}\mathcal{H}^n\\
=  -\sum\limits_{j=2}^m  \int_{X} \left|\nabla \phi_j\right|^2 \Delta \varphi \mathrm{d}\mathcal{H}^n
\leqslant C\int_{X} \varphi \mathrm{d}\mathcal{H}^n.
\end{equation}

Since $\mathrm{Test}F_+({X},\mathsf{d},\mathcal{H}^n)$ is dense in $  H^{1,2}_+({X},\mathsf{d},\mathcal{H}^n)$, and $\phi_1\in \mathrm{Test}F({X},\mathsf{d},\mathcal{H}^n)$ with {}{$|\nabla \phi_1|^2 \in H^{1,2}({X},\mathsf{d},\mathcal{H}^n)$}, the combination of {}{these} facts with (\ref{1eqn4.2}) and (\ref{111eqn4.4}) yields that for any $\varphi\in  H^{1,2}_+({X},\mathsf{d},\mathcal{H}^n)$, 
\begin{equation}\label{1eqn4.3}
\left |\int_{X} \langle \nabla \left|\nabla \phi_1\right|^2,  \nabla \varphi \rangle \mathrm{d}\mathcal{H}^n\right|
=\left|\int_{X} \left|\nabla \phi_1\right|^2 \Delta \varphi \mathrm{d}\mathcal{H}^n\right|
\leqslant C\int_{X} |\varphi |\mathrm{d}\mathcal{H}^n
\leqslant {}{C}\left\|\varphi\right\|_{L^2(\mathcal{H}^n)}.
\end{equation}

Note that (\ref{1eqn4.3}) also holds for any $\varphi\in \text{Lip}({X},\mathsf{d})$ because $\varphi+|\varphi|$, $|\varphi|-\varphi\in \text{Lip}({X},\mathsf{d})$. Since $\mathrm{Test}F({X},\mathsf{d},\mathcal{H}^n)$ is dense in $H^{1,2}({X},\mathsf{d},\mathcal{H}^n)$, we have
\[
 \left|\int_{X} \langle \nabla \left|\nabla \phi_1\right|^2,  \nabla \varphi \rangle \mathrm{d}\mathcal{H}^n\right|\leqslant {}{C}\left\|\varphi\right\|_{L^2(\mathcal{H}^n)}, \ \forall \varphi \in H^{1,2}({X},\mathsf{d},\mathcal{H}^n).
\]

Consequently, the linear functional 
\[
\begin{aligned}
T:H^{1,2}({X},\mathsf{d},\mathcal{H}^n)&\longrightarrow \mathbb{R}\\
\varphi &\longmapsto \int_{X} \langle \nabla \left|\nabla \phi_1\right|^2,  \nabla \varphi \rangle \mathrm{d}\mathcal{H}^n
\end{aligned}
\]
can be continuously extended to a bounded linear functional on $L^2(\mathcal{H}^n)$. Applying the Riesz representation theorem, there exists a unique $h\in L^2(\mathcal{H}^n)$, such that 
\[
T(\varphi)={}{-}\int_{X} \varphi h \mathrm{d}\mathcal{H}^n, \ \ \forall \varphi \in L^2(\mathcal{H}^n).
\]

Therefore $|\nabla \phi_1|^2\in D(\Delta)$ with $\left\|\Delta |\nabla \phi_1|^2\right\|_{L^2(\mathcal{H}^n)}\leqslant {}{C}$. Using $(\ref{1eqn4.3})$ again, and repeating the previous arguments, we have
\[
 \left|\int_{X}  \Delta \left|\nabla \phi_1\right|^2 \varphi  \mathrm{d}\mathcal{H}^n\right|
\leqslant C\int_{X} |\varphi |\mathrm{d}\mathcal{H}^n,\ \forall \varphi\in L^1(\mathcal{H}^n), 
\]
because $\mathrm{Test}F({X},\mathsf{d},\mathcal{H}^n)$ is also dense in $L^1(\mathcal{H}^n)$.
Thus $\left\|\Delta \left|\nabla \phi_1\right|^2\right\|_{L^\infty(\mathcal{H}^n)}\leqslant C$.

According to Theorem \ref{aaaathm3.12}, $\left\||\nabla |\nabla \phi_1|^2|\right\|_{L^\infty(\mathcal{H}^n)}\leqslant C$. For any other $|\nabla \phi_i|^2$, the estimates of $\left\|\Delta |\nabla \phi_i|^2\right\|_{L^\infty(\mathcal{H}^n)}$ and $\left\||\nabla |\nabla \phi_i|^2|\right\|_{L^\infty(\mathcal{H}^n)}$ can be obtained along the same lines. Rewrite these estimates as
\begin{equation}\label{1eqn4.4}
\sum\limits_{i=1}^m\left(\left\|\Delta |\nabla \phi_i|^2\right\|_{L^\infty(\mathcal{H}^n)}+ \left\|\left|\nabla |\nabla \phi_i|^2\right|\right\|_{L^\infty(\mathcal{H}^n)}\right)\leqslant  C.
\end{equation}

Applying (\ref{abc2.14}), (\ref{1eqn4.4}) and Proposition \ref{heatkernel2}, we have
\[
 \int_{X}\varphi \left|\mathop{\mathrm{Hess}}\phi_i\right|_{\mathsf{HS}}^2\mathrm{d}\mathcal{H}^n\leqslant C\int_{X}\varphi \mathrm{d}\mathcal{H}^n,\ \ \forall\varphi\in \mathrm{Test}F_+({X},\mathsf{d},\mathcal{H}^n), \ \ i=1,\ldots,m,
\]
which implies that
\begin{equation}\label{1eqn4.5}
\sum\limits_{i=1}^m \left\|\left|\mathop{\mathrm{Hess}}\phi_i\right|_{\mathsf{HS}}\right\|_{L^\infty(\mathcal{H}^n)}\leqslant C.
\end{equation}

For {}{each} $\langle \nabla \phi_i,\nabla \phi_j \rangle$ ($i,j=1,\ldots,m$), from (\ref{11eqn2.16}) we obtain that
\begin{equation}\label{1111eqn4.7}
\begin{aligned}
|\langle \nabla \varphi, \nabla \langle \nabla \phi_i,\nabla \phi_j \rangle \rangle|&=\left| \mathop{\mathrm{Hess}}\phi_i(\nabla \phi_j,\nabla\varphi)+ \mathop{\mathrm{Hess}}\phi_j(\nabla \phi_i,\nabla\varphi)\right|\\
\ &\leqslant \left(\left|\mathop{\mathrm{Hess}}\phi_i\right|_{\mathsf{HS}}|\nabla \phi_j|+\left|\mathop{\mathrm{Hess}}\phi_j\right|_{\mathsf{HS}}|\nabla \phi_i| \right)|\nabla \varphi|\\
&\leqslant C |\nabla \varphi| \ \  \mathcal{H}^n\text{-a.e.}, \ \ \ \forall \varphi\in H^{1,2}({X},\mathsf{d},\mathcal{H}^n).
\end{aligned}
\end{equation}

As a result, $\langle \nabla\phi_i,\nabla \phi_j \rangle\in H^{1,2}({X},\mathsf{d},\mathcal{H}^n)$. We complete the proof by letting $\varphi=\langle \nabla\phi_i,\nabla \phi_j \rangle$ in (\ref{1111eqn4.7}), which shows that
\begin{equation}\label{1eqn4.6}
 \left\|\nabla \langle \nabla \phi_i,\nabla \phi_j \rangle \right\|_{L^\infty(\mathcal{H}^n)}\leqslant C.
\end{equation} 

\end{proof}

\begin{lem}\label{1lem4.3}
For any $\epsilon>0$, there exists $0<\delta\leqslant\Psi(\epsilon|C)$, such that for any $0<r<\delta$ and any arbitrary but fixed $x_0\in {X}$, the following holds.
\begin{enumerate} 
\item\label{1lem4.3a} The map 
\begin{equation}\label{1eqn4.9}
\begin{aligned}
\mathbf{x}_0:B_r(x_0)&\longrightarrow \mathbb{R}^n\\
x&\longmapsto (u_1(x),\ldots,u_n(x))
\end{aligned}
\end{equation}
is $(1+\epsilon)$-bi-Lipschitz from $B_r(x_0)$ to $\mathbf{x}_0(B_r(x_0))$, where each $u_i$ is a linear combination of $\phi_1,\ldots,\phi_m$ with coefficients only dependent on $x_0$. 
\item\label{1lem4.3b}  The matrix-valued function 
\[
\begin{aligned}
U: B_r(x_0)&\longrightarrow \mathsf{M}_{n\times n}(\mathbb{R})\\
 x&\longmapsto (u^{ij}(x)):=\left(\langle \nabla u_i,\nabla u_j\rangle(x)\right),
\end{aligned}
\]
is Lipschitz {}{continuous and satisfies} $(1-\epsilon)I_n\leqslant U\leqslant (1+\epsilon)I_n$ on $B_r(x_0)$. Moreover, there exists a matrix-valued Lipschitz function 
\[
\begin{aligned}
B: B_r(x_0)&\longrightarrow  \mathsf{M}_{n\times n}(\mathbb{R})\\
 x&\longmapsto \left(b_{ij}(x)\right),
\end{aligned}
\]
such that 
\[
BUB^{T}(x)=I_n, \ \ \ \forall {}{x\in B_r(x_0)}.
\]

\end{enumerate}

\end{lem}
\begin{proof}
Consider the matrix-valued function \[
\begin{aligned}
E:{X}&\longrightarrow  \mathsf{M}_{m\times m}(\mathbb{R})\\
x&\longmapsto \left(\langle \nabla \phi_i,\nabla \phi_j\rangle(x)\right),
\end{aligned}
\] 
which is Lipschitz continuous by Lemma \ref{1lem4.2}. For any fixed $x_0\in {X}$, since $E(x_0)$ is a symmetric matrix of trace $n$ and satisfies $E(x_0)^2=E(x_0)$, there exists an $m\times m$ orthogonal matrix $A=(a_{ij})$, such that 
\[
AE(x_0)A^{T}=\left(
\begin{array}{rl}
I_n & 0\\
0 & 0
\end{array}
\right).
\]

Letting $u_i=\sum\limits_{j=1}^m a_{ij}\phi_j$, $g$ then can be written as $g=\sum\limits_{i=1}^m d u_i \otimes d u_i $ with 
\begin{equation}\label{1111eqn4.11}
\sum\limits_{i,j=n+1}^m \left\langle \nabla u_i,\nabla u_j \right\rangle^2(x_0)=0.
\end{equation}

In order to use Theorem \ref{111thm4.3}, we need

\begin{equation}\label{1eqn4.8}
\sum\limits_{i=1}^m\left\|\left|\nabla u_i\right|^2\right\|_{L^\infty(\mathcal{H}^n)} +\sum\limits_{i=1}^m\left\| \Delta u_i\right\|_{L^\infty(\mathcal{H}^n)}+\sum\limits_{i,j=1}^m \left\||\nabla \left\langle \nabla u_i,\nabla u_j \right\rangle|\right\|_{L^\infty(\mathcal{H}^n)}\leqslant C,
\end{equation}
which follows directly from the {}{Proposition \ref{heatkernel2} and} Lemma \ref{1lem4.2}. We claim that for any $\epsilon\in (0,1)$, there exists $0<\delta\leqslant \Psi(\epsilon|C)$, such that $\mathbf{x}_0$ is a locally uniformly $\epsilon$-isometric immersion on $B_r(x_0)$ for any $0<r<\delta$. 

For any $y_0\in B_r(x_0)$, $0<s<r$, we have
\begin{equation}\label{1eqn4.10}
\begin{aligned}
\ &\frac{1}{\mathcal{H}^n\left(B_s(y_0)\right)}\int_{B_{\epsilon^{-1}s}(y_0)}\left|g-\sum\limits_{i=1}^n du_i \otimes du_i\right|_{\mathsf{HS}}\mathrm{d}\mathcal{H}^n \\
\leqslant \ &\frac{\mathcal{H}^n\left(B_{\epsilon^{-1}s}(y_0)\right)}{\mathcal{H}^n\left(B_s(y_0)\right)}\left({}{\fint_{B_{\epsilon^{-1}s}(y_0)}}\left|g-\sum\limits_{i=1}^n du_i \otimes du_i\right|_{\mathsf{HS}}^2\mathrm{d}\mathcal{H}^n\right)^{\frac{1}{2}}\\
=\ &\frac{\mathcal{H}^n\left(B_{\epsilon^{-1}s}(y_0)\right)}{\mathcal{H}^n\left(B_s(y_0)\right)}\left({}{\fint_{B_{\epsilon^{-1}s}(y_0)}}\sum\limits_{i,j=n+1}^m \left\langle \nabla u_i , \nabla u_j \right\rangle^2 \mathrm{d}\mathcal{H}^n\right)^{\frac{1}{2}} \leqslant C\epsilon^{-1}\exp(C\epsilon^{-1})\delta^2{}{,}
\end{aligned}
\end{equation}
where the last inequality comes from (\ref{BGinequality}), (\ref{1111eqn4.11}) and (\ref{1eqn4.8}). 

Thus applying Theorem \ref{111thm4.3}, there exists $0<\delta\leqslant \Psi(\epsilon|C)$, such that for any $0<r<\delta$, the function $\textbf{x}_0$ defined in (\ref{1eqn4.9}) is  $(1+\epsilon)$-bi-Lipschitz from $B_r(x_0)$ to $\mathbf{x}_0(B_r(x_0))$. We may also require $\delta$ to satisfy condition \ref{1lem4.3b}, which is again due to  (\ref{1eqn4.8}). Finally, the choice of the matrix $B(x)$ follows from a standard congruent transformation of $U(x)$.


\end{proof}

\begin{lem}\label{1lem4.4}
${X}$ admits a $C^{1,1}$ differentiable structure.
\end{lem}
\begin{proof}
Since $({X},\mathsf{d})$ is compact, by taking $\epsilon=\frac{1}{2}$ in Lemma \ref{1lem4.3}, there exists a finite index set $\Gamma$, such that the finite family of pairs $\{(B_r(x_\gamma),\mathbf{x}_\gamma)\}_{\gamma\in\Gamma}$ satisfies the following properties.

\begin{enumerate}
\item It is a covering of ${X}$, i.e. ${X}\subset \bigcup_{\gamma\in \Gamma} B_r(x_\gamma)$.
\item For every $\gamma\in \Gamma$, $\mathbf{x}_\gamma$ is $\frac{3}{2}$-bi-Lipschitz from $B_r(x_\gamma)$ to $\mathbf{x}_\gamma (B_r(x_\gamma))\subset \mathbb{R}^n$, and each component of $\mathbf{x}_\gamma$ is a linear combination of $\phi_1,\ldots,\phi_m$ with coefficients only dependent on $x_\gamma$.
\end{enumerate}

We only prove the $C^{1,1}$ regularity of $\phi_1,\ldots,\phi_m$ on $(B_r(x_0),\mathbf{x}_0)$, since the $C^{1,1}$ regularity of $\phi_1,\ldots,\phi_m$ on any other $(B_r(x_\gamma),\mathbf{x}_\gamma)$ can be proved in a same way. 

For any $y_0\in B_r(x_0)$, without loss of generality, assume that $B_s(y_0) \subset B_r(x_0)$ for some $s>0$ and $\mathbf{x}_0(y_0)=0_n \in \mathbb{R}^n$. Since $\mathbf{x}_0$ is a $\frac{3}{2}$-bi-Lipschitz map (thus also a homeomorphism) from $B_r(x_0)$ to $\mathbf{x}_0 (B_r(x_0))$, for any sufficiently small $t>0$, there exists a unique $y_t \in B_r(x_0)$ such that $\mathbf{x}_0(y_t)=(t,0,\ldots,0)$.

 For $i=1,\ldots,n$, set 
\begin{equation}\label{11eqn4.16}
\begin{aligned}
v_i: B_s(y_0)&\longrightarrow \mathbb{R}\\
x&\longmapsto\sum\limits_{j=1}^n b_{ij}(y_0)u_j(x),
\end{aligned}
\end{equation}
{}{where $B=(b_{ij})$ is taken as in Lemma \ref{1lem4.3}}. It can be immediately checked that $\langle \nabla v_i,\nabla v_j\rangle(y_0)=\delta_{ij}$ $(i,j=1,\ldots,n)$. 

Notice that
\begin{equation}\label{0417efghi}
\begin{aligned}
\ &{}{\fint_{B_\tau (y_0)}} \left|g-\sum\limits_{i=1}^n dv_i\otimes dv_i\right|_{\mathsf{HS}}^2 \mathrm{d}\mathcal{H}^n \\
=&{}{\fint_{B_\tau(y_0)}} \left(n+\sum\limits_{i,j=1}^n \langle \nabla v_i,\nabla v_j\rangle-2\sum\limits_{i=1}^n \left|\nabla v_i\right|^2 
\right)\mathrm{d}\mathcal{H}^n 
\rightarrow 0\ \ \text{as }\tau\rightarrow 0^+.
\end{aligned}
\end{equation}

Thus arguing as in the proof of Lemma \ref{1lem4.3} and applying Theorem \ref{111thm4.3} to $B_{2\mathsf{d}(y_0,y_{t})}(y_0)$ for any sufficiently small $t>0$, we know

\begin{equation}\label{11111eqn4.15}
\sum\limits_{i=1}^n \left(\frac{v_i(y_t)-v_i(y_0)}{\mathsf{d}(y_t,y_0)}\right)^2\rightarrow 1,\ \ \text{as }t\rightarrow 0^+.
\end{equation}

Recall $u_i(y_t)=u_i(y_0)=0$ ($i=2,\ldots,n$). This together with (\ref{11111eqn4.15}) shows
\begin{equation}\label{202204041}
\sum\limits_{i=1}^n b_{i1}^2(y_0) \lim\limits_{t \rightarrow 0^+} \frac{t^2}{\mathsf{d}(y_t,y_0)^2}=1.
\end{equation}

Next is to calculate values of $\lim\limits_{t\rightarrow 0^+}\dfrac{u_{i}(y_t)-u_{i}(y_0)}{t}$ for $i=n+1,\ldots,m$. 

For $i=n+1,\ldots,m$, set 

\[
\begin{aligned}
f_i:B_s(y_0)&\longrightarrow  [0,\infty)\\
x&\longmapsto u_i(x)- \sum\limits_{j=1}^n \langle\nabla u_i, \nabla v_j \rangle (y_0)  v_j(x).  
\end{aligned}
\]

Observe that 
\begin{equation}\label{123eqn4.16}
\lim\limits_{x\rightarrow y_0}\langle \nabla f_i,\nabla v_k\rangle(x)=0,\ \ i=n+1,\ldots,m,\  k=1,\ldots,n. 
\end{equation}

Thus (\ref{0417efghi}) and (\ref{123eqn4.16}) yield that $|\nabla f_i|(y_0)=0$ ($i=n+1,\ldots, m$). From the definition of the {}{local} Lipschitz constant of {}{a} Lipschitz function, we get

\[
\frac{1}{\mathsf{d}(y_t,y_0) }\left((u_{i}(y_t)-u_{i}(y_0))-\sum\limits_{j=1}^n \langle\nabla u_i, \nabla v_j \rangle (y_0) \left(v_j(y_t)-v_{j}(y_0)\right)  \right) \rightarrow 0,\ \ \text{as $t\rightarrow 0^+$}.
\]

  Therefore 
 \begin{equation}\label{111eqn4.17}
\begin{aligned}
\lim\limits_{t\rightarrow 0^+}\dfrac{u_{i}(y_t)-u_{i}(y_0)}{\mathsf{d}(y_t,y_0)}
=&\sum\limits_{j=1}^n \langle \nabla u_i,\nabla v_j \rangle (y_0) \lim\limits_{t\rightarrow 0^+}\dfrac{v_{j}(y_t)-v_{j}(y_0)}{\mathsf{d}(y_t,y_0)}\\
=& \sum\limits_{j=1}^n b_{j1}(y_0)\langle \nabla u_i,\nabla v_j \rangle (y_0)\lim\limits_{t \rightarrow 0^+}\dfrac{u_1(y_t)-u_1(y_0)}{\mathsf{d}(y_t,y_0)}\\
=&\sum\limits_{j,k=1}^n b_{j1}(y_0)b_{jk}(y_0)\langle \nabla u_i,\nabla u_k \rangle (y_0)\lim\limits_{t \rightarrow 0^+}\dfrac{t} {\mathsf{d}(y_t,y_0)}.
\end{aligned}
\end{equation}

As a result of (\ref{202204041}) and (\ref{111eqn4.17}), 
\[
\lim\limits_{t\rightarrow 0^+}\dfrac{u_{i}(y_t)-u_{i}(y_0)}{t}=\sum\limits_{j,k=1}^n b_{j1}(y_0)b_{jk}(y_0)\langle \nabla u_i,\nabla u_k \rangle (y_0).
\]

Analogously,
\[
\lim\limits_{t\rightarrow 0^-}\dfrac{u_{i}(y_t)-u_{i}(y_0)}{t}=\sum\limits_{j,k=1}^n b_{j1}(y_0)b_{jk}(y_0)\langle \nabla u_i,\nabla u_k \rangle (y_0).
\]

Hence for $i=n+1,\ldots,m$, $k=1,\ldots,n$, we get
\begin{equation}\label{111eqn4.18}
\frac{\partial u_i}{\partial u_k}(x)=\sum\limits_{j,l=1}^n b_{jk}(x)b_{jl}(x)\langle \nabla u_i,\nabla u_l \rangle (x), \ \ \forall x\in B_r(x_0).
\end{equation}

According to the fact that each $\phi_i$ is a linear combination of $u_1,\ldots,u_m$ with coefficients only dependent on $x_0$, each $\dfrac{\partial \phi_i}{\partial u_j}$ is Lipschitz continuous on $B_r(x_0)$ and is also Lipschitz continuous on $\mathbf{x}_0(B_r(x_0))$ ($i=1,\ldots,m$, $j=1,\ldots,n$). If $B_r(x_{\gamma'})\cap B_r(x_0)\neq \emptyset$ for some $\gamma' \in \Gamma\setminus \{0\}$, since each component of the coordinate function $\mathbf{x_{\gamma'}}$ is a linear combination of $\phi_1,\ldots,\phi_m$, the transition function from $(B_r(x_0),\mathbf{x}_0)$ to $(B_r(x_{\gamma'}),\mathbf{x}_{\gamma'})$ is $C^{1,1}$ on $(B_r(x_0)\cap B_r(x_{\gamma'}),\mathbf{x}_0)$. 

Therefore, $\{(B_r(x_\gamma),\mathbf{x}_\gamma)\}_{\gamma\in\Gamma}$ gives a $C^{1,1}$ differentiable structure of ${X}$.


\end{proof}

\begin{lem}\label{1lem4.5}

For the sake of brevity, we only state the following assertions for $\left(B_r(x_0),\mathbf{x}_0\right)$ by using the 
{}{notation} of Lemma \ref{1lem4.3}. 

\begin{enumerate}
\item\label{1lem4.52} For any $f_1,f_2\in C^1({X})$, we have 
\begin{equation}\label{1eqn4.16}
\langle \nabla f_1,\nabla f_2\rangle=\sum\limits_{j,k=1}^n u^{jk}\frac{\partial f_1}{\partial u_j}\frac{\partial f_2}{\partial u_k} \ \text{  on } B_r(x_0).
\end{equation}
\item\label{1lem4.53} $(\mathbf{x}_0)_\sharp \left(\mathcal{H}^n\llcorner B_r(x_0)\right)=\left(\mathrm{det}(U)\right)^{-\frac{1}{2}} \mathcal{L}^n\llcorner \mathbf{x}_0\left(B_r(x_0)\right)$.
\end{enumerate}
\end{lem}
\begin{proof}

Statement \ref{1lem4.52} follows directly from the chain rule of $\nabla$. As for statement \ref{1lem4.53}, according to the bi-Lipschitz property of $\mathbf{x}_0$, there exists a Radon-Nikodym derivative $h$ of $(\mathbf{x}_0^{-1})_\sharp\left(\mathrm{det}(U))^{-\frac{1}{2}} \mathcal{L}^n\llcorner \mathbf{x}_0(B_r(x_0))\right)$ with respect to $\mathcal{H}^n\llcorner B_r(x_0)$. 

Again for any $B_{2s}(y_0)\subset B_r(x_0)$, we choose $\{v_i\}_{i=1}^n$ as in (\ref{11eqn4.16}) and set
\[
\begin{aligned}
\mathbf{y}_0:B_s(y_0)&\longrightarrow \mathbb{R}^n\\
x&\longmapsto (v_1(x),\ldots,v_n(x)).
\end{aligned}
\]

By Theorem \ref{111thm4.3},
\begin{equation}\label{11eqn4.18}
\lim\limits_{\tau\rightarrow 0^+} \frac{ \mathcal{L}^n\left(\mathbf{y}_0\left(B_\tau(y_0)\right)\right) }{\mathcal{H}^n(B_\tau(y_0))}=1.
\end{equation}

Set $\tilde{B}=B(y_0)$. Then it follows from the choice of the matrix $B$ that   
\begin{equation}\label{11eqn4.20}
\mathrm{det}(\tilde{B})^2\mathrm{det}\left(U(y_0)\right)=1.
\end{equation}

Using the commutativity of the following diagram, 
\[
\xymatrix{
B_s(y_0)\ar[r]^{\mathbf{y}_0\ \ }\ar[dr]_{\mathbf{x}_0}\ \ & \mathbf{y}_0(B_s(y_0))\ar[d]^{\tilde{B}^{-1}}\\
\ & \mathbf{x}_0(B_s(y_0))}
\]
for any $0<\tau\leqslant s$, it holds that
\begin{equation}\label{11eqn4.19}
\int_{\mathbf{x}_0\left(B_\tau(y_0)\right)}\left(\mathrm{det}(U)\right)^{-\frac{1}{2}} \mathrm{d}\mathcal{L}^n=\int_{\mathbf{y}_0\left(B_\tau(y_0)\right)}\left(\mathrm{det}(U)\left(\tilde{B}^{-1}(x)\right)\right)^{-\frac{1}{2}}\mathrm{det}(\tilde{B})^{-1} \mathrm{d}\mathcal{L}^n(x).
\end{equation}

Thus combining the continuity of $\mathrm{det}(U)$ with (\ref{11eqn4.18}), (\ref{11eqn4.20}) and (\ref{11eqn4.19}) implies
\[
\lim\limits_{\tau\rightarrow 0^+} \frac{1}{\mathcal{H}^n(B_\tau(y_0))} \int_{\mathbf{x}_0\left(B_\tau(y_0)\right)}\left(\mathrm{det}(U)\right)^{-\frac{1}{2}} \mathrm{d}\mathcal{L}^n =1.
\]

Therefore, $h= 1$ $\mathcal{H}^n$-a.e. on $B_r(x_0)$, which suffices to conclude.
\end{proof}

\begin{proof}[Proof of Theorem \ref{thm1.5}]
We start by improving the regularity of each $\phi_i$ on each coordinate chart $\left(B_r(x_\gamma),\mathbf{x}_\gamma \right)$. It suffices to verify the case $\gamma=0$. 

We still use the notation in Lemma \ref{1lem4.3}. For any fixed $B_{2s}(y_0)\subset B_r(x_0)$, without loss of generality, assume that $\mathbf{x}_0(y_0)=0_n$ and $B_s(0_n)\subset \mathbf{x}_0\left(B_{2s}(y_0)\right)$.

We first claim that for $j=1,\ldots,n$, 
\begin{equation}\label{1eqn4.21}
\sum\limits_{k=1}^n\frac{\partial }{\partial u_k} \left(u^{jk}\mathrm{det}(U)^{-\frac{1}{2}}\right)=\Delta u_j \mathrm{det}(U)^{\frac{1}{2}} \ \ \mathcal{L}^n\text{-a.e. in }B_s(0_n).
\end{equation}

Notice that for any $\varphi \in C_c\left(B_s(0_n)\right)\cap C^1({X})$, in view of Lemma \ref{1lem4.5}, we have
\[
\begin{aligned}
\int_{B_s(0_n)} \varphi \Delta u_j \mathrm{det}(U)^{-\frac{1}{2}}\mathrm{d}\mathcal{L}^n &=\int_{\mathbf{x}_0^{-1}\left(B_s(0_n)\right)}  \varphi\Delta u_j \mathrm{d}\mathcal{H}^n\\
&=-\int_{\mathbf{x}_0^{-1}\left(B_s(0_n)\right)} \langle \nabla u_j,\nabla \varphi  \rangle \mathrm{d}\mathcal{H}^n\\
&=-\int_{B_s(0_n)}  \sum\limits_{k=1}^n u^{jk}\dfrac{\partial \varphi}{\partial u_k} \mathrm{det}(U)^{-\frac{1}{2}}\mathrm{d}\mathcal{L}^n, 
\end{aligned}
\]
which suffices to show (\ref{1eqn4.21}) since each $u^{jk}$ is Lipschitz continuous on $B_s(0_n)$. Similarly, for $i=1,\ldots,m$ and any $\varphi \in C_c\left(B_s(0_n)\right)\cap C^1({X})$, it holds that
\begin{equation}\label{1eqn4.22}
\int_{B_s(0_n)} \varphi \mu_i \phi_i \mathrm{det}(U)^{-\frac{1}{2}}\mathrm{d}\mathcal{L}^n =\int_{B_s(0_n)} \sum\limits_{j,k=1}^n u^{jk}\frac{\partial \phi_i}{\partial u_j}\frac{\partial \varphi}{\partial u_k} \mathrm{det}(U)^{-\frac{1}{2}}\mathrm{d}\mathcal{L}^n. 
\end{equation}

Therefore the $C^{1,1}$-regularity of $\phi_i$ as well as (\ref{1eqn4.21}), (\ref{1eqn4.22}) gives a PDE as follows.
\begin{equation}\label{111eqn4.26}
\sum\limits_{j,k=1}^n u^{jk}\frac{\partial^2 \phi_i}{ \partial u_j \partial u_k}+\sum\limits_{j=1}^n\Delta u_j \frac{\partial \phi_i}{ \partial u_j }+\mu_i \phi_i=0 \ \ \mathcal{L}^n\text{-a.e. in } B_s(0_n).
\end{equation}

Since each $\Delta u_j$ is some linear combination of $\phi_1,\ldots,\phi_m$, it is also $C^{1,1}$ with respect to $\{(B_r(x_\gamma),\mathbf{x}_\gamma)\}_{\gamma\in \Gamma}$. From the classical PDE theory (see for instance \cite[Theorem 6.13]{GT01}),  $\phi_i\in C^{2,\alpha}(B_s(0_n))$ for any $\alpha\in (0,1)$. Hence, ${X}$ admits a $C^{2,\alpha}$ differentiable structure $\{(B_r(x_\gamma),\mathbf{x}_\gamma)\}_{\gamma\in \Gamma}$.

Let us use this differentiable structure to define the following $(0,2)$-type symmetric tensor: 
\[
\tilde{g}:=\sum\limits_{i=1}^m \tilde{d}\phi_i\otimes \tilde{d}\phi_i,
\]
which is $C^{1,\alpha}$ with respect to $\{(B_r(x_\gamma),\mathbf{x}_\gamma)\}_{\gamma\in \Gamma}$. We claim that $\tilde{g}$ is a Riemannian metric. Again it suffices to prove this statement on $\left(B_r(x_0),\mathbf{x}_0\right)$. 

Set  
\[
\begin{aligned}
\mathcal{U}:{X}&\longrightarrow \mathsf{M}_{m\times m}(\mathbb{R})\\
x&\longmapsto \left(\langle \nabla u_i,\nabla u_j\rangle\right){}{.}
\end{aligned}
\]

{}{For any $x\in {X}$, rewrite $\mathcal{U}(x)$ as the following block matrix 
\[
\mathcal{U}(x):=\begin{pmatrix} U(x) &U_1(x) \\
U_1^T(x)& U_2(x)\end{pmatrix}.
\]
}
The choice of {}{$\{u_i\}_{i=1}^m$} implies that $\tilde{g}$ has a local expression as 
\[
\tilde{g}=\sum\limits_{i=1}^m \tilde{d}u_i\otimes \tilde{d}u_i=\sum\limits_{i=1}^n \tilde{d}u_i\otimes \tilde{d}u_i+\sum\limits_{i=n+1}^m \sum\limits_{k,l=1}^n\dfrac{\partial u_i}{\partial  u_k}\dfrac{\partial u_i}{\partial u_l}\tilde{d}u_k\otimes \tilde{d}u_l.
\]

By (\ref{111eqn4.18}), for $i=n+1,\ldots,m$, $l=1,\ldots,n$ and any $x\in B_r(x_0)$, we have
\[
\dfrac{\partial u_i}{\partial u_l}(x)=\sum\limits_{j,k=1}^n b_{jl}(x)b_{jk}(x)\langle \nabla u_i,\nabla u_k \rangle (x)=\left(B^TBU_1(x)\right)_{li}=\left(U^{-1}U_1(x)\right)_{li},
\]
which implies that 
\begin{equation}\label{111eqn4.25}
\tilde{g}(x)=\sum\limits_{i=1}^n \tilde{d}u_i\otimes \tilde{d}u_i+\sum\limits_{k,l=1}^n \left(U^{-1}U_1U_1^T U^{-1}(x)\right)_{kl}\tilde{d}u_k\otimes \tilde{d}u_l,\ \ \forall x\in B_r(x_0).
\end{equation}

Since $\mathcal{U}^2-\mathcal{U}\equiv 0$ on $B_r(x_0)$, $U^2+U_1U_1^T-U\equiv 0$ on $B_r(x_0)$. By (\ref{111eqn4.25}), 
\begin{equation}\label{1111eqn4.28}
\tilde{g}(x)=\sum\limits_{j,k=1}^n\left(U^{-1}\right)_{jk}(x)\tilde{d}u_j\otimes \tilde{d}u_k, \ \ \text{on}\ B_r(x_0),
\end{equation}
which is positive definie on $B_r(x_0)$.  Moreover, $u^{jk}\in C^{1,\alpha}\left(B_r(x_0)\right)$ $(j,k=1,\ldots,n)$. Applying the regularity theorem for second order elliptic PDE (for example \cite[Theorem 6.17]{GT01}) to (\ref{111eqn4.26}), we see that $\phi_i\in C^{3,\alpha}\left(B_r(x_0)\right)$ ($i=1,\ldots, m$). Thus the regularity of $\tilde{g}$ can be improved to $C^{2,\alpha}$. Then (\ref{1111eqn4.28}) shows that $u^{jk}\in C^{2,\alpha}\left(B_r(x_0)\right)$ $(j,k=1,\ldots,n)$.
  
Applying a proof by induction, $\tilde{g}=g$ is actually a smooth Riemannian metric with respect to the smooth differentiable structure $\{(B_r(x_\gamma),\mathbf{x}_\gamma)\}_{\gamma\in \Gamma}$. This implies that $({X},\mathsf{d})$ is isometric to {}{an} $n$-dimensional smooth Riemannian manifold $(M^n,g)$. To see that $(M^n,g)$ is a closed Riemannian manifold, it suffices to use Theorem \ref{111thm4.3} again to show that the tangent space at any point is not isometric to the upper plane $\mathbb{R}^n_+$.
\end{proof}
\begin{proof}[{}{Proof of Corollary \ref{cor1.11}}]
{}{Without loss of generality, we may assume that $\mathfrak{m}({X})=1$.} Among lines in the proof, each limit process and each convergence of the series {}{is} guaranteed by Proposition \ref{heatkernel2}, which can be checked via similar estimates in Lemma \ref{20211220b}.

First calculate that

\begin{equation}\label{eqn4.2}
{}{n=\left\langle g,g\right\rangle=\left\langle c(t)g_t,g\right\rangle
=c(t)\sum\limits_{i=1}^\infty e^{-2\mu_i t} \left|\nabla \phi_i\right|^2.}
\end{equation}

Integrating (\ref{eqn4.2}) on ${X}$, we have 
\[
{}{n=c(t)\sum\limits_{i=1}^\infty e^{-2\mu_i t}\mu_i.}
\] 

Let $\phi_1,\ldots,\phi_m$ be {}{an} $L^2(\mathfrak{m})$-orthonormal basis of the eigenspace corresponding to the first eigenvalue $\mu_1$. Then 
\begin{equation}\label{eqn4.3}
{}{\left|\sum\limits_{i=1}^m d\phi_i \otimes d\phi_i-\frac{e^{2\mu_1 t}}{c(t)}g\right|_{\mathsf{HS}}}\leqslant \sum\limits_{i=m+1}^\infty e^{2\mu_1t-2\mu_i t} \left|d\phi_i \otimes d\phi_i\right|_{\mathsf{HS}}=\sum\limits_{i=m+1}^\infty e^{2\mu_1t-2\mu_i t} \left|\nabla \phi_i\right|^2.
\end{equation}

Again the integration of (\ref{eqn4.3}) on ${X}$ gives
\begin{equation}\label{111eqn3.19}
{}{\int_{X} \left|\sum\limits_{i=1}^m d\phi_i \otimes d\phi_i-\frac{e^{2\mu_1 t}}{c(t)}g\right|_{\mathsf{HS}}\mathrm{d}\mathfrak{m}}\leqslant \sum\limits_{i=m+1}^\infty e^{2\mu_1 t-2\mu_i t}\mu_i.
\end{equation}

Since
\[
\lim\limits_{t\rightarrow \infty} {}{\frac{e^{2\mu_1 t} }{c(t)} }=\frac{\mu_1}{n}+\lim\limits_{t\rightarrow \infty} \frac{1}{n}\sum\limits_{i=2}^\infty e^{2\mu_1t-2\mu_i t}\mu_i=\frac{\mu_1}{n},
\] 
 (\ref{111eqn3.19}) implies that
 \[
 \int_{X} \left|\sum\limits_{i=1}^m d\phi_i \otimes d\phi_i-\frac{\mu_1}{n}g\right|_{\mathsf{HS}}\mathrm{d}\mathfrak{m}=0.
 \]

 In other words,
 \[
  \sum\limits_{i=1}^m d\phi_i \otimes d\phi_i= \frac{\mu_1}{n}g.
 \]  

For other eigenspaces, it suffices to use a proof by induction to conclude.
\end{proof}

\section{Diffeomorphic finiteness theorems}\label{sec5}

This section is {}{dedicated} to prove Theorem \ref{thm1.8} and Theorem \ref{thm1.12}. To fix the notation, for a Riemannian manifold $(M^n,g)$, denote by $\mathrm{vol}_g$ its volume element, by $\mathrm{K}_g$ its sectional curvature, by $\mathrm{Ric}_g$ its Ricci curvature tensor, by $\mathrm{inj}_g(p)$ the injectivity radius at $p$ and by $(\nabla^g)^k$, $\Delta^g$ the $k$-th covariant derivative and the Laplacian with respect to $g$, by $\mathsf{d}_g$ the {}{metric} induced by $g$.

To begin with, let us recall some results about the convergence of Sobolev functions on varying spaces. See \cite{AH18, GMS13, AST16}.

\begin{thm}[Compactness of Sobolev functions]\label{222thm5.1}
{}{Let $\{\left({X}_i,\mathsf{d}_i,\mathcal{H}^n\right)\}$ be a sequence of non-collapsed $\mathrm{RCD}(K,n)$ spaces} with {}{$\sup_i\mathrm{diam}({X}_i,\mathsf{d}_i)<\infty$} and 
\[
\left({X}_i,\mathsf{d}_i,\mathcal{H}^n\right)\xrightarrow{\mathrm{mGH}} \left({X},\mathsf{d},\mathcal{H}^n\right).
\]
Let $f_i\in H^{1,2}\left({X}_i,\mathsf{d}_i,\mathcal{H}^n\right)$ with $\sup_i\|{}{f_i}\|_{H^{1,2}\left({X}_i,\mathsf{d}_i,\mathcal{H}^n\right)}<\infty$. Then there exists $f\in H^{1,2}\left({X},\mathsf{d},\mathcal{H}^n\right)$, and a subsequence of $\{f_i\}$ which is still denoted as $\{f_i\}$ such that $\{f_i\}$ $L^2$-strongly converges to $f$ and 
\[
\liminf\limits_{i\rightarrow \infty} \int_{X_i} \left|\nabla^{{X}_i} f_i\right|^2 \mathrm{d}\mathcal{H}^n \geqslant \int_{X} \left|\nabla^{{X}} f\right|^2 \mathrm{d}\mathcal{H}^n.
\]
\end{thm}
\begin{thm}[Stability of Laplacian]\label{111thm5.2} {}{Let $\{\left({X}_i,\mathsf{d}_i,\mathcal{H}^n\right)\}$, $\left({X},\mathsf{d},\mathcal{H}^n\right)$ be taken as in Theorem \ref{222thm5.1}}. Let $f_i\in D\left(\Delta^{{X}_i}\right)$ with
\[
\sup\limits_i \left(\|f_i\|_{H^{1,2}\left({X}_i,\mathsf{d}_i,\mathcal{H}^n\right)}+\left\|\Delta^{{X}_i} f_i\right\|_{L^2\left(\mathcal{H}^n\right)}\right)<\infty.
\]
If $\{f_i\}$ $L^2$-strongly converges to $f$ on ${X}$ $($by Theorem \ref{222thm5.1} $f \in H^{1,2}({X},\mathsf{d},\mathcal{H}^n)$$)$, then the following statements hold.
\begin{enumerate}
\item $f\in D(\Delta^{X})$.
\item $\{\Delta^{{X}_i}f_i\}$ $L^2$-weakly converges to $\Delta^{{X}}f$.
\item $\{\left|\nabla^{{X}_i}f_i\right|\}$ $L^2$-strongly converges to $\left|\nabla^{{X}}f\right|$.
\end{enumerate}
\end{thm}

We are now in the position to prove the following theorem.
\begin{thm}\label{abcthm5.3}
$\mathcal{M}(K,n,D,\tau)$ has finitely many members up to diffeomorphism.
\end{thm}

\begin{proof}
Assume the contrary, i.e. there exists a sequence of Riemannian manifolds $\{( M_i^n,g_i)\}\subset \mathcal{M}(K,n,D,\tau)$, which are pairwise non-diffeomorphic. 

On each $( M_i^n,g_i)$, there exists $m_i\in\mathbb{N}$, such that 
\begin{equation}\label{1111eqn5.2}
g_i=\sum\limits_{j=1}^{m_i}d\phi_{i,j}\otimes d\phi_{i,j},
\end{equation}
where $\phi_{i,j}$ is a non-constant eigenfunction of $-\Delta^{g_i}$ with the corresponding eigenvalue $\mu_{i,j}$ and satisfies that $\left\|\phi_{i,j}\right\|_{L^2(\mathrm{vol}_{g_i})}\geqslant \tau>0$ ($i\in\mathbb{N}$, $j=1,\ldots, m_i$). By taking trace of (\ref{1111eqn5.2}) with respect to $g_i$, we know
\begin{equation}\label{eqn4.29}
n=\sum\limits_{j=1}^{m_i}\left|\nabla^{g_i}\phi_{i,j}\right|^2.
\end{equation}

Integration of (\ref{eqn4.29}) on $( M_i^n,g_i)$ shows that 
\[
n\text{vol}_{g_i}( M^n_i)\geqslant \tau^2\sum\limits_{j=1}^{m_i} \mu_{i,j}.
\]

The Bishop-Gromov volume comparison theorem and  {}{Li-Yau's first eigenvalue lower bound \cite[Theorem 7]{LY80}} imply that 
\begin{equation}\label{eqn4.30}
C_1(K,n)D^n\geqslant n\text{vol}_{g_i}( M_i^n) {}{\geqslant}\tau^2\sum\limits_{j=1}^{m_i} \mu_{i,j}\geqslant C_2(K,n,D)\tau^2 m_i \geqslant C_2(K,n,D)\tau^2.
\end{equation}

Moreover, for each $\phi_{i,j}$, 
\begin{equation}\label{1112eqn5.4}
\|\phi_{i,j}\|_{L^2(\mathrm{vol}_{g_i})}^2 {}{=}\mu_{i,j}^{-1}\int_{M_i^n} |\nabla^{g_i} \phi_{i,j}|^2 \mathrm{dvol}_{g_i} \leqslant n\mu_{i,j}^{-1}\mathrm{vol}_{g_i}(M_i^n)\leqslant C(K,n,D,\tau).
\end{equation}

{}{Since (\ref{eqn4.30}) implies that $1\leqslant \inf_i m_i\leqslant \sup_i m_i\leqslant C(K,n,D,\tau)$, after passing to a subsequence, we may take $m\in\mathbb{N}$} such that
\begin{equation}\label{1111eqn5.3}
g_i=\sum\limits_{j=1}^{m}d\phi_{i,j}\otimes d\phi_{i,j}, \ \forall i\in \mathbb{N}.
\end{equation}

Moreover, by (\ref{eqn4.30}), we may assume that
\begin{equation}\label{111eqn5.4}
\lim\limits_{i\rightarrow \infty} \mu_{i,j}=\mu_j \in [C_2(K,n,D),\tau^{-2}C_1(K,n)D^n],\ \ j=1,\ldots,m.
\end{equation}

According to Theorem \ref{11thm2.15} and (\ref{eqn4.30}), $\{( M_i^n,g_i)\}$ can also be required to satisfy
\[
\left( M_i^n,\mathsf{d}_{g_i},\text{vol}_{g_i}\right)\xrightarrow{\mathrm{mGH}} \left({X},\mathsf{d},\mathcal{H}^n\right)
\]
for some non-collapsed RCD$(K,n)$ space  $({X},\mathsf{d},\mathcal{H}^n)$. In particular, combining (\ref{eqn4.30})-(\ref{111eqn5.4}) with Theorems \ref{222thm5.1} and \ref{111thm5.2}, we know that on $({X},\mathsf{d},\mathcal{H}^n)$,
\[
g=\sum\limits_{j=1}^m d\phi_j\otimes d\phi_j,
\]
where each $\phi_j$ is an eigenfunction of $-\Delta$ with the eigenvalue $\mu_j$. Therefore, from Theorem \ref{thm1.5}, we deduce that $(
{X},\mathsf{d})$ is isometric to an $n$-dimensional smooth closed Riemannian manifold $(M^n,g)$. However, due to \cite[Theorem A.1.12]{ChCo1}, $ M_i^n$ is diffeomorphic to $M^n$ for any sufficiently large $i$. A contradiction.
\end{proof}

The proof of Theorem \ref{thm1.8} mainly uses the estimates in Section \ref{sec4} and a stronger version of Gromov convergence theorem given by Hebey-Herzlish \cite{HH97}. For reader's convenience, Hebey-Herzlish's theorem is stated below.
\begin{thm}\label{11thm5.4}
Let $\{(M_i^n,g_i)\}$ be a sequence of {}{$n$-dimensional} closed Riemannian manifolds such that {}{
\[  \sup\limits_i\mathrm{vol}_{g_i}(M_i^n)<\infty,\ \inf\limits_i\inf\limits_{p\in M_i^n} \mathrm{inj}_{g_i}(p)>0,
\]
and for all $k\in \mathbb{N}$, 
\[
\sup\limits_i \sup\limits_{M_i^n}\left|(\nabla^{g_i})^k\mathrm{Ric}_{g_i}\right|<\infty.
\]
}
Then there exists a subsequence which is still denoted as $\{(M_i^n,g_i)\}$, such that it $C^\infty$-converges to a closed Riemannian manifold $(M^n,g)$.
\end{thm}

The following Cheeger-Gromov-Taylor's estimate of the injectivity radius is also necessary for the proof of Theorem \ref{thm1.8}. 
\begin{thm}[{\cite[Theorem 4.7]{CGT82}}]\label{111thm5.5}
Let $(M^n,g)$ be a complete $n$-dimensional Riemannian manifold with $|K_g|\leqslant \kappa<\infty$. Then there exists a constant $c_0=c_0(n)>0$, such that for any $0<r\leqslant \frac{\pi}{4\sqrt{\kappa}}$, 
\[
\mathrm{inj}_g(p)\geqslant c_0 r\frac{\mathrm{vol}(B_r(p))}{\int_0^r V_{-(n-1)\kappa,n}\mathrm{d}t},\ \forall p\in M^n.
\]
\end{thm}

\begin{proof}[Proof of Theorem \ref{thm1.8}]
By Theorem \ref{abcthm5.3}, without loss of generality, we may take a sequence $\{( M^n,g_i)\}\subset\mathcal{M}(K,n,D,\tau)$ such that $\{( M^n,g_i)\}$ mGH converges to $(M^n,g)$ and {}{that} (\ref{eqn4.30})-(\ref{111eqn5.4}) still hold. Denote by $B_r^i(p)$ the $r$-radius ball (with respect to $\mathsf{d}_{g_i}$) centered at $p\in {}{M^n}$ for notation convenience.

\textbf{Step 1} Uniform two-sided sectional curvature bound on $( M^n,g_i)$. 

According to the estimates in Section \ref{sec4}, combining (\ref{eqn4.30})-(\ref{111eqn5.4}), we may choose a uniform $r>0$, such that for every arbitrary but fixed $B_{4096r}^i(p)\subset {}{M^n}$, there exists a coordinate function $\mathbf{x}^i=(u^i_1,\ldots,u^i_n):B_{4096r}^i(p)\rightarrow \mathbb{R}^n$ satisfying the following properties.
\begin{enumerate}
\item $\mathbf{x}^i$ is $\dfrac{3}{2}$-bi-Lipschitz from $B_{4096r}^i(p)$ to $\mathbf{x}^i(B_{4096r}^i(p))$ (by Lemma \ref{1lem4.3}).
\item Set {}{$(g_i)_{jk}:=g_i\left(\dfrac{\partial}{\partial u_j^i},\dfrac{\partial}{\partial u_k^i}\right)$.}  Then it holds that 
\begin{equation}\label{1eqn5.6}
\frac{1}{2}I_n\leqslant (g_i)_{jk}\leqslant 2I_n,\ \text{on $B_{4096r}^i(p)$ (by Lemma \ref{1lem4.3} and (\ref{1111eqn4.28}))}.
\end{equation}
\end{enumerate}

We first give a $C^{2,\alpha}$-estimate of $g_i$ on each $( M^n,g_i)$ for any $\alpha\in (0,1)$. Applying (\ref{1eqn4.8}) and (\ref{1eqn5.6}) implies that on $B_{4096r}^i(p)$
\begin{equation}\label{11111eqn5.8}
C\geqslant \left|\nabla^{g_i} (g_i)^{jk}\right|^2=\sum\limits_{\beta,\gamma=1}^n (g_i)^{\beta\gamma}\frac{\partial}{\partial u^i_\beta} (g_i)^{jk} \frac{\partial}{\partial u^i_\gamma}  (g_i)^{jk}\geqslant \frac{1}{2}\sum\limits_{\beta=1}^n \left(\frac{\partial}{\partial u^i_\beta} (g_i)^{jk}\right)^2,
\end{equation}
for some $C=C(K,n,D,\tau)$ which may vary from line to line.

Then $\left\| (g_i)^{jk}\right\|_{C^{\alpha}(B_{4096r}^i(p))}\leqslant C$ follows from (\ref{11111eqn5.8}) and the local bi-Lipschitz property of $\mathbf{x}^i$ ($j,k=1,\ldots,n$).

For $j=1,\ldots,n$, $\left|\nabla^{g_i} \phi_{i,j}\right|\leqslant C$ yields that $\|\phi_{i,j}\|_{C^\alpha\left(B_{4096r}^i(p)\right)}\leqslant C$. This implies that $\left\|\Delta^{g_i} u_{i,j}\right\|_{C^\alpha\left(B_{4096r}^i(p)\right)}\leqslant C$ since each $u_{i,j}$ is the linear combination of $\phi_{i,j}$ constructed as in Lemma \ref{1lem4.3}. Then the the classical Schauder interior estimate (see for example \cite[Theorem 6.2]{GT01}), together with the PDE (\ref{111eqn4.26}) implies that  $\left\|\phi_{i,j}\right\|_{C^{2,\alpha}\left(B_{256r}^i(p)\right)}\leqslant C$ since $\mathbf{x}^i \left(B_{256r}(p)\right)\subset B_{512r}({}{\mathbf{x}^i(p)})\subset \mathbf{x}^i\left(B_{1024r}^i(p)\right)\subset B_{2048r}({}{\mathbf{x}^i(p)}) \subset\mathbf{x}^i\left(B_{4096r}^i(p)\right)$. As a result, $\left\|\Delta^{g_i} u_{i,j}\right\|_{C^{2,\alpha}\left(B_{256r}^i(p)\right)}\leqslant C$. Moreover, (\ref{1111eqn5.3}) shows that
\[
\left\| (g_i)_{jk}\right\|_{C^{1,\alpha}\left(B_{256r}^i(p)\right)}, \left\| (g_i)^{jk}\right\|_{C^{1,\alpha}\left(B_{256r}^i(p)\right)}\leqslant C,\ j,k=1,\ldots,n.
\]

Applying again the Schauder interior estimate to $\phi_{i,j}$ in the PDE (\ref{111eqn4.26}), we know $\left\|\phi_{i,j}\right\|_{C^{3,\alpha}\left(B_{16r}^i(p)\right)}\leqslant C$. Consequently,
\[
\left\| (g_i)_{jk}\right\|_{C^{2,\alpha}\left(B_{16r}^i(p)\right)}, \left\| (g_i)^{jk}\right\|_{C^{2,\alpha}\left(B_{16r}^i(p)\right)}\leqslant C,\ j,k=1,\ldots,n.
\]

Since the calculation of sectional curvature only involves the terms in form of $(g_i)_{jk}$, $(g_i)^{jk}$, $\dfrac{\partial }{\partial u^i_\beta} (g_i)^{jk}$, $\dfrac{\partial }{\partial u^i_\beta} (g_i)_{jk}$, $\dfrac{\partial^2 }{\partial u^i_\beta \partial u^i_\gamma} (g_i)_{jk}$ ($j,k,\beta,\gamma=1,\ldots,n$), $|\mathrm{K}
_{g_i}|$ has a uniform upper bound $C_0=C_0(K,n,D,\tau)$.

\textbf{Step 2} {}{Uniform lower injectivity radius bound on $( M^n,g_i)$.}

By Step 1, we may take $r'=\min\{r, C_0^{-1}\}$, which is still denoted as $r$. In order to use Theorem \ref{111thm5.5}, we need nothing but the lower bound of $\mathrm{vol}_{g_i}(B^i_r(p))$.  It suffices to apply (\ref{eqn4.30}) and Bishop-Gromov volume comparison theorem again to show that 
\begin{equation}\label{12345eqn5.7}
 \tilde{C}(K,n,D,\tau)r^n\leqslant \mathrm{vol}_{g_i}(B^i_r(p))\leqslant C(K,n)D^n,
\end{equation}
because (\ref{12345eqn5.7}), Theorem \ref{111thm5.5} as well as the two-sided sectional curvature bound obtained in Step 1 then imply that 
$\inf\limits_{p\in M^n} \mathrm{inj}_{g_i}(p)\geqslant \tilde{C}(K,n,D,\tau)$.

\textbf{Step 3} Improvement of the regularity.

In order to apply Theorem \ref{11thm5.4}, it suffices to show that for any $k\geqslant 0$, there exists $C_k(K,n,D,\tau)$ such that $|(\nabla^{g_i})^k \mathrm{Ric}_{g_i}|(p)\leqslant C_{k}(K,n,D,\tau)$ holds for any arbitrary but fixed $p\in M^n$. Since the case $k=0$ is already proved in Step 1, we prove the case $k=1$. 
 
Using the Schauder interior estimate again and an argument similar to Step 1 gives the following $C^{4,\alpha}$-estimate of $\phi_{i,j}$:
\[
\left\|\phi_{i,j}\right\|_{C^{4,\alpha}\left(B_{r}^i(p)\right)}\leqslant C_1(K,n,D,\tau),
\] 
which implies that 
\[
\left\| (g_i)_{jk}\right\|_{C^{3,\alpha}\left(B_{r}^i(p)\right)}, \left\| (g_i)^{jk}\right\|_{C^{3,\alpha}\left(B_{r}^i(p)\right)}\leqslant C_1(K,n,D,\tau),\ j,k=1,\ldots,n.
\]

Therefore, we see
\[
\sup_{M^n} |\nabla^{g_i}\mathrm{Ric}_{g_i}|\leqslant C_1(K,n,D,\tau).
\]

Now by using the proof by induction, for any $k\geqslant 2$, there exists $C_{k}=C_{k}(K,n,D,\tau)$ such that
 \[
\sup_{M^n} |\left(\nabla^{g_i}\right)^k\mathrm{Ric}_{g_i}|\leqslant C_{k}(K,n,D,\tau),
\] 
which suffices to conclude.

\end{proof}
\begin{proof}[Proof of Theorem \ref{thm1.12}]
The proof is almost the same as that of Theorem \ref{abcthm5.3}, and we omit some details. Assume the contrary, i.e. there exists a sequence of pairwise non-diffeomorphic Riemannian manifolds $\{( M_i^n,g_i)\}$ such that $( M_i^n,g_i)\in \mathcal{N}\left(K,n,D,i^{-1},\tau\right)$ for any $i\in \mathbb{N}$. Then for each $\{( M_i^n,g_i)\}$, the almost isometric immersion condition ensures the existence of some $m_i\in\mathbb{N}$, such that
\begin{equation}\label{eqn5.7}
\frac{1}{\mathrm{vol}_{g_i}(M_i^n)}\int_{M_i^n}\left|\sum\limits_{j=1}^{m_i} d\phi_{i,j}\otimes d\phi_{i,j}- g_i \right|\mathrm{dvol}_{g_i}\leqslant \frac{1}{i}.
\end{equation}

Thus
\begin{equation}\label{1234eqn5.1}
\begin{aligned}
\frac{\tau^2 \mu_{i,j}}{\mathrm{vol}_{g_i}(M_i^n)}&\leqslant \frac{1}{\mathrm{vol}_{g_i}(M_i^n)}\int_{M_i^n}|\nabla^{g_i} \phi_{i,j}|^2\mathrm{dvol}_{g_i}\\
\ & \leqslant \frac{1}{\mathrm{vol}_{g_i}(M_i^n)}\int_{M_i^n}{}{\left(\sum\limits_{j,k=1}^{m_i}\left\langle \nabla^{g_i} \phi_{i,j},\nabla^{g_i}\phi_{i,k}\right\rangle^2\right)^{\frac{1}{2}}}\mathrm{dvol}_{g_i}\\
\ &\leqslant \frac{1}{\mathrm{vol}_{g_i}(M_i^n)}\int_{M_i^n}\left|\sum\limits_{j=1}^{m_i} d\phi_{i,j}\otimes d\phi_{i,j}- g_i \right|\mathrm{dvol}_{g_i}+ \frac{1}{\mathrm{vol}_{g_i}(M_i^n)}\int_{M_i^n}| g_i |\mathrm{dvol}_{g_i}\\
\ &\leqslant \frac{1}{i}+\sqrt{n}.
\end{aligned}
\end{equation}

Applying {}{Li-Yau's first eigenvalue lower bound \cite[Theorem 7]{LY80}} and Bishop-Gromov volume comparison theorem to (\ref{1234eqn5.1}) shows that 
\begin{equation}\label{eqn5.8}
C_1(K,n,D)\leqslant \mu_{i,j}\leqslant C_2(K,n,D,\tau).
\end{equation}

It then follows from (\ref{1234eqn5.1}) and (\ref{eqn5.8}) that 
\begin{equation}\label{1eqn5.9}
C_3(K,n,D,\tau)\leqslant \mathrm{vol}_{g_i}(M_i^n)\leqslant C_4(K,n,D)\ \text{and}\ \tau\leqslant \|\phi_{i,j}\|_{L^2(\mathrm{vol}_{g_i})}\leqslant C_5(K,n,D).
\end{equation}

To see $\{m_i\}$ has an upper bound, it suffices to notice that 
\[
\begin{aligned}
\ &\left|\sum\limits_{j=1}^{m_i} \left\|\phi_{i,j}\right\|_{L^2(\mathrm{vol}_{g_i})}^2\mu_{i,j}-n\mathrm{vol}_{g_i}(M_i^n) \right|\\
=&\left|\int_{M_i^n}\left\langle\sum\limits_{j=1}^{m_i} d\phi_{i,j}\otimes d\phi_{i,j}- g_i ,g_i\right\rangle\mathrm{dvol}_{g_i}\right|\\
\leqslant& \sqrt{n}\int_{M_i^n}\left|\sum\limits_{j=1}^{m_i} d\phi_{i,j}\otimes d\phi_{i,j}- g_i \right|\mathrm{dvol}_{g_i}\leqslant \sqrt{n}\ C_4(K,n,D)\frac{1}{i}
\end{aligned}
\]

As a result, $m_i\leqslant C_6(K,n,D,\tau)$. Therefore there exists $m\in \mathbb{N}$ and a subsequence of $\{( M_i^n,g_i)\}$ which is still denoted as $\{( M_i^n,g_i)\}$, such that each $( M_i^n,g_i)$ admits an $i^{-1}$-almost isometrically immersing eigenmap into $\mathbb{R}^m$. In addition, $\{( M_i^n,g_i)\}$ can also be required to satisfy
\[
\left( M_i^n,\mathsf{d}_{g_i},\text{vol}_{g_i}\right)\xrightarrow{\mathrm{mGH}} \left({X},\mathsf{d},\mathcal{H}^n\right)
\]
for some non-collapsed RCD$(K,n)$ space  $({X},\mathsf{d},\mathcal{H}^n)$. Again combining (\ref{eqn5.7})-(\ref{1eqn5.9}) with Theorems \ref{222thm5.1} and \ref{111thm5.2}, we see that on $({X},\mathsf{d},\mathcal{H}^n)$,
\[
g=\sum\limits_{j=1}^m d\phi_j\otimes d\phi_j,
\]
where each $\phi_j$ is an eigenfunction of $-\Delta$ with the eigenvalue $\mu_j:=\lim\limits_{i\rightarrow\infty}\mu_{i,j}$. Finally, it suffices to apply Theorem \ref{thm1.5} and \cite[Theorem A.1.12]{ChCo1} to deduce the contradiction.
\end{proof}
\section{Examples}

In this section, some examples about the IHKI condition of Riemannian manifolds are provided. {}{Let us first emphasis that if $(M^n,g)$ is an $n$-dimensional compact IHKI Riemannian manifold, then it follows from Corollary \ref{cor1.11} and Takahashi theorem \cite[Theorem 3]{Ta66} that for any $t>0$, $\rho^{M^n}_{t}:(p\mapsto\rho^{M^n}(p,p,t))$} is a constant function. By Lemma \ref{llem3.1}, we see that 

\begin{enumerate}
\item\label{20221201} For any $k,n\in \mathbb{N}$, $\underbrace{\mathbb{S}^n\times\cdots \times \mathbb{S}^n}_{2^k \text{times}}$ is IHKI.

\item For any $p,q\in\mathbb{N}$, the compact Lie group $\mathrm{SO}(2p+q)/\mathrm{SO}(2p)\times \mathrm{SO}(q)$ with a constant positive Ricci curvature is IHKI since it is homogeneous and irreducible. 

\end{enumerate}

Example \ref{exmp4.5} gives the sharpness of Theorem \ref{thm1.2}. The construction of Example \ref{exmp4.5} needs the following two lemmas.

\begin{lem}\label{prop4.3}
Let $( M^m,g)$, $( N^n,h)$, $( M^m\times N^n,\tilde{g})$ be  $m,n,(m+n)$-dimensional $\mathrm{IHKI}$ Riemannian manifolds respectively, where $\tilde{g}$ is the standard product Riemannian metric. Then for any $t>0$, it holds that $(\rho^ {M^m}_t)^n=(\rho^ {N^n}_t)^m$.
\end{lem}
\begin{proof}
Owing to Lemmas \ref{llem3.1} and \ref{1lem3.15}, we have  
{}{\begin{equation}\label{eqn4.4}
\begin{aligned}
c^{ M^m\times  N^n}(t)g_t^{ M^m\times N^n } (p,q)&=c^{ M^m\times  N^n}(t)\rho^{ M^m}_{2t}g_t^{ N^n}(q) +c^{ M^m\times  N^n}(t) \rho^{ N^n}_{2t}g_t^{ M^m}(p)\\
\ &=\rho^{ M^m}_{2t}\frac{c^{ M^m\times  N^n}(t)}{c^{N^n}(t)}h(q) + \rho^{ N^n}_{2t}\frac{c^{ M^m\times  N^n}(t)}{c^{M^m}(t)}g(p)\\
\ &=\tilde{g}(p,q).
\end{aligned}
\end{equation}}

Then from (\ref{eqn4.4}), {}{$\rho^ {N^n}_{2t}c^{N^n}(t)=\rho^ {M^m}_{2t}c^ {M^m}(t)$ for any $t>0$}. Moreover, {}{for any $p\in M^m$, we calculate that
\[
\begin{aligned}
\frac{\partial }{\partial t} \rho^ {M^m}_{2t}(p)
=\ & \frac{\partial }{\partial t} \int_{ M^m}\left(\rho^ {M^m}(p,p',t)\right)^2 \text{dvol}_g(p')\\
=\ & 2 \int_{ M^m}\Delta^ {M^m}_{p'}\rho^ {M^m}(p,p',t) \rho^ {M^m}(p,p',t) \text{dvol}_g(p')\\
=\ & -2 \int_{ M^m}\left|\nabla^ {M^m}_{p'}\rho^ {M^m}(p,p',t)\right|^2 \text{dvol}_g(p')
=-2 \left\langle g_t^ {M^m},g\right\rangle(p)=- \frac{2m}{ c^{ M^m}(t)}.
\end{aligned}
\]
}

Analogously {}{$\dfrac{\partial }{\partial t} \rho^ {N^n}_{2t}=- \dfrac{2n}{ c^{ N^n}(t)}$}, and thus $n \rho^ {N^n}_{2t}\dfrac{\partial }{\partial t} \rho^ {M^m}_{2t}  =m   \rho^ {M^m}_{2t}\dfrac{\partial }{\partial t} \rho^ {N^n}_{2t}$. Therefore there exists $\tilde{c}>0$, such that 
\[
\left(\rho^ {M^m}_t\right)^n=\tilde{c}\left(\rho^ {N^n}_t\right)^m,\ \ \forall t>0. 
\]

To see $\tilde{c}=1$, it suffices to use a blow up argument and Theorem \ref{thm2.26} to show that $\lim\limits_{t\rightarrow 0 }t^{\frac{m}{2}}\rho^ {M^m}_t=\left(4\pi\right)^{-\frac{m}{2}}$ and $\lim\limits_{t\rightarrow 0 }t^{\frac{n}{2}}\rho^ {N^n}_t=\left(4\pi\right)^{-\frac{n}{2}}$.
\end{proof}

\begin{lem}\label{lem4.4}
Let $( M^n,g)$ be an $n$-dimensional closed $\mathrm{IHKI}$ Riemannian manifold, then it holds that
{}{\[
\lim\limits_{t\rightarrow \infty} \frac{t}{c^ {M^n}(t)\rho^{ M^n}_{2t}}=0.
\]}
\end{lem}
\begin{proof}

Set {}{$0=\mu_0< \mu_1\leqslant \ldots\rightarrow +\infty$} as the eigenvalues of $-\Delta$ counting with multiplicities. Then it suffices to notice that 
{}{\[
\frac{1}{c^ {M^n}(t)}=\frac{1}{n\text{vol}_g( M^n)}\sum\limits_{i=1}^\infty e^{-2\mu_i t}\mu_i, \ \ \rho^ {M^n}_{2t}=\frac{1}{\text{vol}_g( M^n)}\sum\limits_{i=0}^\infty e^{-2\mu_it}
\]}
 and let $t\rightarrow \infty$.
\end{proof}

\begin{exmp}\label{exmp4.5}
Set $\mathbb{S}^n(k):=\left\{(x_1,\ldots,x_{n+1})\in\mathbb{R}^{n+1}:x_1^2+\cdots+x_{n+1}^2=k^2\right\}$. Observe that {}{$c^{\mathbb{S}^n(k)}(1)=k^{n+2}c^{\mathbb{S}^n}(k^{-2})$}, $\rho^{\mathbb{S}^n(k)}_2=k^{-n}\rho^{\mathbb{S}^n}_{2k^{-2}}$. By Lemma \ref{lem4.4}, 
{}{\[
\lim\limits_{k\rightarrow 0} c^{\mathbb{S}^n(k)}(1)\rho^{\mathbb{S}^n(k)}_2=\infty.
\]}

This implies that for any small $r>0$, there exists $s=s(r)$ such that 
{}{\[
c^{\mathbb{S}^1(r)}(1)\rho^{\mathbb{S}^1(r)}_2=c^{\mathbb{S}^2(s)}(1)\rho^{\mathbb{S}^2(s)}_2.
\] }

Consider the product Riemannian manifold $\left(\mathbb{S}^1(r)\times \mathbb{S}^2(s), g_{\mathbb{S}^1(r)\times\mathbb{S}^2(s)}\right)$. By (\ref{eqn4.4}), there exists $c(r)>0$, such that $c(r)\Phi_1^{\mathbb{S}^1(r)\times \mathbb{S}^2(s)}$ realizes an isometric immersion into $L^2\left(\mathrm{vol}_{g_{\mathbb{S}^1(r)\times\mathbb{S}^2(s)}}\right)$. 

If $\left(\mathbb{S}^1(r)\times \mathbb{S}^2(s), g_{\mathbb{S}^1(r)\times\mathbb{S}^2(s)}\right)$ is IHKI, then by Proposition \ref{prop4.3}, it holds that 
\begin{equation}\label{eqn4.5}
\rho^{\mathbb{S}^2(s)}_t=\left(\rho^{\mathbb{S}^1(r)}_t\right)^2=\rho^{\mathbb{S}^1(r)\times \mathbb{S}^1(r)}_t, \ \ \forall t>0.
\end{equation}

 Therefore by taking integral of (\ref{eqn4.5}), we see that for any $t>0$,
\begin{equation}\label{eqn4.6}
\text{vol}\left(\mathbb{S}^2(s)\right)\sum\limits_{i=0}^\infty \exp\left(-r^{-2}\mu_i^{\mathbb{S}^1\times \mathbb{S}^1}t\right) =\text{vol}\left(\mathbb{S}^1(r)\times \mathbb{S}^1(r)\right)\sum\limits_{i=0}^\infty \exp\left(-s^{-2}\mu_i^{\mathbb{S}^2}t\right).
\end{equation}

Then $\mathrm{vol}\left(\mathbb{S}^2(r_2)\right)=\mathrm{vol}\left(\mathbb{S}^1(r_1)\times \mathbb{S}^1(r_1)\right)$ follows by letting $t\rightarrow 0$ in (\ref{eqn4.6}), which implies that $s(r)=r$.  (\ref{eqn4.6}) then becomes
\begin{equation}\label{eqn4.7}
\sum\limits_{i=1}^\infty \exp\left(-r^{-2}\mu_i^{\mathbb{S}^1\times \mathbb{S}^1}t\right) =\sum\limits_{i=1}^\infty \exp\left(-r^{-2}\mu_i^{\mathbb{S}^2}t\right),\  \forall t>0 .
\end{equation}

Since $\mu_1^{\mathbb{S}^1\times \mathbb{S}^1}=\mu_4^{\mathbb{S}^1\times \mathbb{S}^1}=2<\mu_5^{\mathbb{S}^1\times \mathbb{S}^1}$ and $\mu_1^{\mathbb{S}^2}=\mu_3^{\mathbb{S}^2}=2<\mu_4^{\mathbb{S}^2}$, multiplying $\exp(2r^{-2}t)$ to both sides of (\ref{eqn4.7}) and letting $t\rightarrow \infty$, the right hand side of (\ref{eqn4.7}) converges to 3, while the left hand side of (\ref{eqn4.7}) converges to 4. A contradiction.
\end{exmp}

There is also a simple example which does not satisfy the condition 2 in Corollary \ref{cor4.7}.
\begin{exmp}
Consider the product manifold $(\mathbb{S}^1\times \mathbb{R}, g_{\mathbb{S}^1\times \mathbb{R}})$. It is obvious that
\[
\begin{aligned}
\pi g_t^{\mathbb{S}^1\times \mathbb{R}}=&\dfrac{1}{(4\pi t)^{\frac{1}{2}}}\sum\limits_{i=1}^\infty e^{-i^2t } g_{\mathbb{S}^1}+\dfrac{c_1^\mathbb{R}}{t^{\frac{3}{2}}}\sum\limits_{i=0}^\infty e^{-i^2t} i^2 g_\mathbb{R}\\
\geqslant &\dfrac{1}{(4\pi t)^{\frac{1}{2}}}g_{\mathbb{S}^1}+\dfrac{c_1^\mathbb{R}}{t^{\frac{3}{2}}}g_\mathbb{R},
\end{aligned}
\]

As a result, $g_t^{\mathbb{S}^1\times\mathbb{R}}\geqslant \dfrac{c_1^\mathbb{R}}{\pi }t^{-\frac{3}{2}}g_{\mathbb{S}^1\times \mathbb{R}}$ for any sufficiently large $t>0$ but 
\[
{}{\lim\limits_{t\rightarrow \infty} t^{-2}c(t)=\lim\limits_{t\rightarrow \infty} t^{-2}\frac{\pi}{c_1^\mathbb{R}}t^{\frac{3}{2}}=0.}
\]
\end{exmp}



\bigskip


\begin{thebibliography}{10}
\bibitem[ABS19]{ABS19}G. Antonelli, E. Bru\`{e}, D. Semola: Volume bounds for the quantitative singular strata of non collapsed RCD metric measure spaces, Anal. Geom. Metr. Spaces, \textbf{7} (2019), no. 1, 158–178.

\bibitem[AGS14a]{AGS14a}L. Ambrosio, N. Gigli, G. Savar\'{e}: Calculus and heat flow in metric measure spaces and applications to spaces with Ricci bounds from below, Invent. Math. \textbf{195}(2014), no. 2, 289–391.

\bibitem[AGS14b]{AGS14b}\textbf{------}: Metric measure spaces with Riemannian Ricci curvature bounded from below, Duke Math. J. \textbf{163}(2014), no. 7, 1405-1490.

\bibitem[AGS15]{AGS15}\textbf{------}: Bakry-\'{E}mery curvature-dimension condition and Riemannian Ricci curvature bounds, Ann. Probab. \textbf{43}(2015), no. 1, 339–404.

\bibitem[AH17]{AH17}L. Ambrosio, S. Honda: New stability results for sequences of metric measure spaces with uniform Ricci bounds from below, in Measure Theory in Non-Smooth Spaces, 1-51, De Gruyter Open, Warsaw, 2017.

\bibitem[AH18]{AH18}\textbf{------}: Local spectral convergence in RC$\text{D}^\ast(K, N)$ spaces, Nonlinear Anal. \textbf{177} (2018), part A, 1–23.

\bibitem[AHPT21]{AHPT21}L. Ambrosio, S. Honda, J. Portegies, D. Tewodrose: Embedding of RCD($K,N$) spaces in $L^2$ via eigenfunctions, J. Funct. Anal. \textbf{280}(2021), no. 10, Paper No. 108968, 72 pp. 

\bibitem[AHT18]{AHT18}L. Ambrosio, S. Honda, D. Tewodrose: Short-time behavior of the heat kernel and Weyl's law on $\mathrm{RCD}^\ast(K,N)$-spaces, Ann. Global Anal. Geom. \textbf{53}(2018), no. 1, 97-119.

\bibitem[AMS16]{AMS16}L. Ambrosio, A. Mondino, G. Savar\'{e}: On the Bakry-\'Emery condition, the gradient estimates and the Local-to-Global property of RC$\text{D}^\ast(K, N)$ metric measure spaces, J. Geom. Anal. \textbf{26} (2016), no. 1, 24–56.

\bibitem[AMS19]{AMS19}------: Nonlinear diffusion equations and curvature conditions in metric measure spaces, Mem. Amer. Math. Soc. \textbf{262} (2019), no. 1270. 

\bibitem[AST16]{AST16}L. Ambrosio, F. Stra, D. Trevisan: Weak and strong convergence of derivations and stability of flows with respect to MGH convergence, J. Funct. Anal. \textbf{272} (2017), no. 3, 1182–1229.

\bibitem[AT04]{AT04}L. Ambrosio, P. Tilli: Topics on analysis in metric spaces. Oxford Lecture Series in Mathematics and its Applications, 25. Oxford University Press, Oxford, 2004. 

\bibitem[B85]{B85}P. B\'{e}rard: Volume des ensembles nodaux des fonctions propres du laplacien, S\'{e}minaire de th\'{e}orie spectrale et g\'{e}om\'{e}trie \textbf{3}(1985), 1-9.

\bibitem[BBG94]{BBG94}P. B\'{e}rard, G. Besson, S. Gallot: Embedding Riemannian manifolds by their heat kernel, Geom. Funct. Anal.  \textbf{4}(1994), no. 4, 373-398.

\bibitem[BGHZ21]{BGHZ21}C. Brena, N. Gigli, S. Honda, X. Zhu: Weakly non-collapsed RCD spaces are strongly non-collapsed, Journal f\"{u}r die reine und angewandte Mathematik (Crelles Journal), \textbf{2023}(2023), no. 794, 215-252.


\bibitem[BS20]{BS20}E. Bru\`{e}, D. Semola: Constancy of the dimension for RC$\text{D}^\ast(K,N)$ spaces via regularity of Lagrangian flows, Comm. Pure and Appl. Math. \textbf{73}(2020), no. 6, 1141-1204.






\bibitem[ChCo1]{ChCo1}J. Cheeger, T.  Colding: On the structure of spaces with Ricci curvature bounded below. I, J. Differential Geom. \textbf{46}(1997), no. 3, 406–480.


\bibitem[CGT82]{CGT82} J. Cheeger, M. Gromov, M. Taylor: Finite propagation speed, kernel estimates for functions
of the Laplace operator, and the geometry of complete Riemannian manifolds, J. Diff. Geom. \textbf{17} (1982), 15-53.





\bibitem[CN12]{CN12}T. Colding, A. Naber: Sharp H\"older continuity of tangent cones for spaces with a lower Ricci curvature bound and applications, Ann. of Math. (2) \textbf{176}(2012), no. 2, 1173-1229.



\bibitem[D02]{D02}Y. Ding: Heat kernels and Green's functions on limit spaces, Comm. Anal. Geom. \textbf{10} (2002), no. 3, 475–514.

\bibitem[D97]{D97}E. Davis: Non-Gaussian aspects of heat kernel behavior, J. London Math. Soc. \textbf{55} (1997), no. 2, 105–125.

\bibitem[DG16]{DG16}G. De Philippis, N. Gigli: From volume cone to metric cone in the nonsmooth setting, Geom. Funct. Anal. \textbf{26} (2016), no. 6, 1526–1587.

\bibitem[DG18]{DG18}\textbf{------}: Non-collapsed spaces with Ricci curvature bounded from below, J. \'Ec. polytech. Math. \textbf{5} (2018), 613–650.


\bibitem[EKS15]{EKS15}M. Erbar, K. Kuwada, K. Sturm: On the equivalence of the entropic curvature-dimension condition and Bochner’s inequality on metric measure spaces, Invent. Math. \textbf{201} (2015), no. 3, 993–1071.

\bibitem[F87]{F87}K. Fukaya: Collapsing of Riemannian manifolds and eigenvalues of Laplace operator, Invent. Math. \textbf{87} (1987), 517–547.


\bibitem[G81]{G81}M. Gromov: Structures m\'{e}triques pour les vari\'{e}t\'{e}s reimanniennes, redige par J. Lafontaine et
P. Pansu, Textes math. $\text{n}^{\circ}$ 1 Cedic-Nathan, Paris, 1981.

\bibitem[G13]{G13}N. Gigli: The splitting theorem in non-smooth context, ArXiv preprint:1302.5555.

\bibitem[G14]{G14}\textbf{------}: An overview on the proof of the splitting theorem in non-smooth context, Anal. Geom. Metr. Spaces \textbf{2}(2014), no. 1, 169–213.

\bibitem[G15]{G15}\textbf{------}: On the differential structure of metric measure spaces and applications, Mem. Amer. Math. Soc. \textbf{236}(2015), no. 1113.

\bibitem[G18]{G18}\textbf{------}: Nonsmooth differential geometry---an approach tailored for spaces with Ricci curvature bounded from below, Mem. Amer. Math. Soc. \textbf{251} (2018), no. 1196.

\bibitem[GH18]{GH18}N. Gigli, B. Han: Sobolev spaces on warped products, J. Funct. Anal. \textbf{275} (2018), no. 8, 2059–2095.




\bibitem[GMS13]{GMS13}N. Gigli, A. Mondino, G. Savar\'{e}: Convergence of pointed non-compact metric measure spaces and stability of Ricci curvature bounds and heat flows, Proc. Lond. Math. Soc. (3) \textbf{111} (2015), no. 5, 1071–1129.

\bibitem[GP16]{GP16}N. Gigli, E. Pasqualetto: Equivalence of two different notions of tangent bundle on rectifiable metric measure spaces, Comm. Anal. and Geom. \textbf{30}(2022), no. 1, 1-51.

\bibitem[GR18]{GR18}N.Gigli, C. Rigoni: Recognizing the flat torus among $\text{RCD}^\ast(0,N)$ spaces via the study of the first cohomology group. Calc. Var. Partial Differ. Equ. \textbf{57}, 104 (2018).

\bibitem[GR19]{GR19}N. Gigli, C. Rigoni: A note about the strong maximum principle on RCD spaces, Canad. Math. Bull. \textbf{62} (2019), no. 2, 259–266.

\bibitem[GR20]{GR20}N. Gigli, C. Rigoni: Partial derivatives in the nonsmooth setting, J. Funct. Anal. \textbf{283}(2022), no. 4, Paper No. 109528.

\bibitem[GT01]{GT01}D. Gilbarg, N. Trudinger: Elliptic partial differential equations of second order.
Classics in Mathematics, Springer-Verlag, Berlin, 2001. Reprint of the 1998 edition.


\bibitem[HH97]{HH97}E. Hebey, M. Herzlich : Harmonic coordinates, harmonic radius and convergence of Riemannian manifolds, Rendiconti di Matematica, Serie VII, \textbf{17}(1997), 569-605.

\bibitem[H15]{H15}S. Honda: Ricci curvature and $L^p$-convergence, J. Reine Angew Math. \textbf{705} (2015), 85–154.



\bibitem[H21]{H21}------: Isometric immersions of RCD spaces, Comment. Math. Helv. \textbf{96}(2021), no. 3, 515–559.


\bibitem[HS21]{HS21}S. Honda, Y. Sire: Sobolev mappings between RCD spaces and applications to
harmonic maps: a heat kernel approach, ArXiv preprint: 2105.08578.

\bibitem[J14]{J14}R. Jiang: Cheeger-harmonic functions in metric measure space revisited, J. Funct. Anal. \textbf{266} (2014), no. 3, 1373–1394.







\bibitem[JLZ16]{JLZ16}R. Jiang, H. Li, H. Zhang: Heat kernel bounds on metric measure spaces and some applications, Potential Anal.\textbf{ 44} (2016), no. 3, 601–627.


\bibitem[K15a]{K15a}C. Ketterer: Cones over metric measure spaces and the maximal diameter theorem, J. Math. Pures Appl. (9) \textbf{103} (2015), no. 5, 1228–1275.

\bibitem[K15b]{K15b}\textbf{------}: Obata's rigidity theorem for metric measure spaces, Anal. Geom. Metr. Spaces \textbf{3}(2015), no. 1, 278–295.

\bibitem[K19]{K19}Y. Kitabeppu: A sufficient condition to a regular set being of positive measure on RCD spaces, Potential Anal. \textbf{51}(2019), no. 2, 179–196.

\bibitem[KM21]{KM21}V. Kapovitch, A. Mondino: On the topology and the boundary of $N$-dimensional RCD$(K,N)$ spaces, Geom. Topol. \textbf{25} (2021), 445 -495.

\bibitem[L81]{L81}P. Li: Minimal immersions of compact irreducible homogeneous Riemannian manifolds, J. Differential Geom. \textbf{16} (1981), no. 1, 105–115.

\bibitem[LV09]{LV09}J. Lott, C. Villani: Ricci curvature for metric-measure spaces via optimal transport, Ann. of Math. (2) \textbf{169} (2009), no. 3, 903–991.

\bibitem[LY80]{LY80}P. Li, S. Yau: Estimates of eigenvalues of a compact Riemannian manifold, Geometry of the Laplace operator (Proc. Sympos. Pure Math., Univ. Hawaii, Honolulu, Hawaii, 1979), 205–239.


\bibitem[MN19]{MN19}A. Mondino, A. Naber: Structure theory of metric-measure spaces with lower
Ricci curvature bounds, J. Eur. Math. Soc, \textbf{21} (2019), 1809–1854.

\bibitem[MW19]{MW19} A. Mondino and G. Wei: On the universal cover and the fundamental group of an $\mathrm{RCD}^\ast(K, N)$-space, J.
Reine Angew. Math. \textbf{753} (2019), 211–237.

\bibitem[O07]{O07}S. Ohta: On the measure contraction property of metric measure spaces. Comment. Math. Helv.,
\textbf{82}(2007), 805–828.


\bibitem[S14]{S14}G. Savar\'e: Self-improvement of the Bakry-\'{E}mery condition and Wasserstein contraction of the heat flow in RCD$(K,\infty)$ metric measure spaces, Discrete Contin. Dyn. Syst. \textbf{34} (2014), no. 4, 1641–1661.



\bibitem[St95]{St95}K. Sturm: Analysis on local Dirichlet spaces. II. Upper Gaussian estimates for the fundamental solutions of parabolic equations, Osaka J. Math. \textbf{32} (1995), no. 2, 275–312.

\bibitem[St96]{St96}\textbf{------}: Analysis on local Dirichlet spaces. III. The parabolic Harnack inequality, J. Math. Pures Appl. (9) \textbf{75} (1996), no. 3, 273–297.

\bibitem[St06a]{St06a}\textbf{------}: On the geometry of metric measure spaces. I, Acta Math. \textbf{196} (2006), no. 1, 65–131.

\bibitem[St06b]{St06b}\textbf{------}: On the geometry of metric measure spaces. II, Acta Math. \textbf{196} (2006), no. 1, 133–177. 


\bibitem[Ta66]{Ta66}T. Takahashi: Minimal immersions of Riemannian manifolds, J. Math. Soc. Japan \textbf{18} (1966), 380–385.

\bibitem[Ta96]{Ta96} M. Taylor: Partial Differential Equations, Volume 1,2,3. Springer-Verlag. New York, NY, 1996

\bibitem[ZZ19]{ZZ19}H. Zhang, X. Zhu: Weyl's law on $\text{RCD}^\ast$(K,N) metric measure spaces, Comm. Anal. Geom. \textbf{27} (2019), no. 8, 1869–1914.


\end{thebibliography}
\end{document}